\documentclass[10pt,activeacute,amsfonts,reqno]{amsart}
\usepackage{a4wide}
\usepackage[english]{babel}
\usepackage{inputenc}
\usepackage{mathabx}
\usepackage{subfigure}
\usepackage{amsmath,amsthm,amsfonts,amssymb,amscd}
\usepackage{amssymb,amsfonts}
\usepackage[all,arc]{xy}
\usepackage{enumerate}
\usepackage{mathrsfs}
\usepackage{marginnote}
\usepackage{mathtools}
\usepackage[dvipsnames]{xcolor}
\usepackage{amsthm,thmtools,xcolor}

\usepackage[all]{xy}
\usepackage{tikz-cd}
\usetikzlibrary{cd}
\usepackage[square,sort,comma,numbers,nonamebreak]{natbib}
\theoremstyle{plain}
\newtheorem{thm}{Theorem}[section]

\newtheorem{lem}{Lemma}[section]
\newtheorem{cor}{Corollary}[section]
\newtheorem{prop}{Proposition}[section]

\newtheorem{rem}{\textit{Remark}}[section]

\newtheorem{lemx}{Lemma}

\newtheorem{defn}{Definition}[section]

\theoremstyle{definition}

\theoremstyle{remark}

\numberwithin{equation}{section}
\usepackage{color}
\usepackage[colorlinks = true,
linkcolor = red,
urlcolor  = blue,
citecolor = blue,
anchorcolor = black]{hyperref}
\usepackage{hyperref}
\numberwithin{equation}{section}

\DeclareMathOperator{\sech}{\mathrm{sech}}
\DeclareMathOperator{\diag}{\mathrm{diag}}

\newcommand\underrel[3][]{\mathrel{\mathop{#3}\limits_{%
			\ifx c#1\relax\mathclap{#2}\else#2\fi}}}
\providecommand{\abs}[1]{\left\lvert#1\right\rvert}
\providecommand{\norm}[1]{\left\lVert#1\right\rVert}

\title[Belinski-Zakharov Einstein Field Equations]{Global existence and long time behavior in Einstein-Belinski-Zakharov soliton spacetimes}

\author[Claudio Mu\~noz]{Claudio Mu\~noz}  
	\address{Departamento de Ingenier\'{\i}a Matem\'atica and Centro
de Modelamiento Matem\'atico (UMI 2807 CNRS), Universidad de Chile, Casilla
170 Correo 3, Santiago, Chile.}
	\email{cmunoz@dim.uchile.cl}
	\thanks{C.M. was partially funded by Chilean research grants FONDECYT 1191412, 1231250, and Centro de Modelamiento Matemático (CMM), ACE210010 and FB210005, BASAL funds for centers of excellence from ANID-Chile.}

\author[J. Trespalacios]{Jessica Trespalacios*}
\address{Departamento de Ingenier\'{\i}a Matem\'atica, Universidad de Chile, Casilla
	170 Correo 3, Santiago, Chile.}
\email{jtrespalacios@dim.uchile.cl}
\thanks{J.T.'s work was funded in part by the National Agency for Research and Development (ANID)/ DOCTORADO NACIONAL/2019 -21190604, Chilean research grants FONDECYT 1191412 and 1231250, Centro de Modelamiento Matemático (CMM), ACE210010 and FB210005, BASAL funds for centers of excellence from ANID-Chile.}


\date{\today}

\usepackage{refcheck}
\begin{document}

	\begin{abstract}
		We consider the vacuum Einstein field equations under the Belinski-Zakharov symmetry, {\color{blue} which leaves the problem as a 1+1 dimensional quasilinear system of PDEs.  Depending on the chosen signature of the metric, these spacetimes contain most of  the well-known special solutions in General Relativity}. In this paper, {\color{blue} we consider the case of cosmological metrics, in the Belinsky-Zakharov notation}, and prove global existence of small Belinski-Zakharov spacetimes under a natural nondegeneracy condition. We also construct new energies and virial functionals to provide a description of the energy decay of smooth global cosmological metrics inside the light cone. Finally, some applications are presented in the case {\color{blue} of the particular metrics called} generalized Kasner solitons. 
		
	\end{abstract}
	
	\maketitle
	
	
	\section{\color{blue} Introduction}\label{Sec:1}
	
The \textit{Einstein vacuum} equations determine a $4-$dimensional manifold $\mathcal{M}$  with a Lorentzian metric $\tilde{g}$ with vanishing \textit{Ricci} curvature
\begin{equation}\label{EV}
		\text{R}_{\mu\nu} (\widetilde g)=0.
	\end{equation}
These equations can be written under certain gauge choices as a {\color{blue}particular} system of quasilinear equations. This is a remarkable aspect of the general relativity theory, in contrast to Newton's gravitation theory: the equation \eqref{EV} is nontrivial even in the absence of matter. The focus of this paper is the understanding of outstanding solutions of \eqref{EV} in the setting of Belinski-Zakharov spacetimes. 

{\color{blue}
More precisely, the main motivation of this work is to provide a rigorous analytical framework for understanding the global dynamics and long-time behavior of Einstein vacuum spacetimes within the Belinski-Zakharov class, in the case of nearly constant determinant metrics. These solutions, which include physically relevant geometries such as black holes, solitons of Kasner type, and gravitational waves, have been extensively studied from the viewpoint of integrability, finding generalized Lax-pair structures and using well-designed dressing type techniques. However, from the point of view of classical PDE techniques, their stability and global-in-time properties remain not well understood. By combining techniques from the analysis of nonlinear PDEs with the algebraic structure of integrable systems, we aim to establish new global existence and decay mechanisms.
}	

\subsection{The Belinski-Zakharov integrability ansatz}
{\color{blue} Symmetry has been a successful method for understanding complicated dynamics in a series of works related to dispersive models, see e.g. \cite{fustero1986einstein, Einstein1937, silva2019scaling} and references therein.} Belinski and Zakharov \cite{belinskii1978integration,belinski1979stationary} (see also Kompaneets \cite{kompaneets} and \cite{Zakharov1978,zakharov1980int}) proposed an application for the Inverse Scattering Transform for spacetimes that admit two commuting Killing vector fields. Using this ansatz, Einstein's vacuum field equations can be recast as a 1+1 system of four {\color{blue} quasilinear wave equations}.  In this paper, we will follow their ansatz and describe rigorously symmetric spacetimes and their long-time dynamics.

Belinski and Zakharov recalled the particular case in which the metric tensor $\widetilde g_{\mu \nu}$ depends on two variables only, which correspond to spacetimes that admit two commuting Killing vector fields, i.e., an Abelian two-parameter group of isometries. This assumption allowed them to propose the so-called Belinski-Zakharov transform to obtain solitonic solutions, so-called “gravisolitons".   The Gravisolitons have an unusual number of features; however, it is known that spacetimes highly important in physics and cosmology applications, such as Kasner spacetimes, can be identified as gravisolitons \cite{belinski2001gravitational, belinskii1978integration}. {\color{blue} Subsection \ref{Disc} contains further details on these solutions.}

Following Belinski and Zakharov, we will choose here a metric tensor depending on a time-like coordinate $x^0$, and one space-like coordinate $x^1$ (possibly nonnegative). This choice, {\color{blue}as will be clear below}, corresponds to considering non-stationary gravitational fields, often referred to as Gowdy models \cite{Gowdy},  even when no compact spatial sections are considered. They are also referred to as generalized Einstein–Rosen spacetimes \cite{carmeli1984einstein}. In the particular case where one has diagonal metrics, these are called Einstein-Rosen spacetimes, first considered in 1937 by Einstein and Rosen \cite{Einstein1937}.

Set $x^0 = t$ and $x^1= x$. In this case, the coordinates are typically expressed using Cartesian coordinates in which $x^a,x^b \in\{y,z\}$, where the Latin indexes $a,b\in\{2,3\}$. Then, the spacetime interval {\color{blue} associated to the metric tensor $\widetilde g_{\mu \nu}$} is a simplified block diagonal form\footnote{\color{blue} In general, we consider a spacetime interval of a block diagonal form
$ds^2= g_{ij}dx^idx^j+g_{ab}dx^adx^b$,
where we use the summation convention on $i,j,k,...\in \{0, 3 \}$ and $a,b,c,...\in \{1, 2\}$, and we assume that $g_{ij}$ and $g_{ab}$ depend only on the coordinates $x^k$.  By using a proper change of coordinates, one can always cast the matrix $g_{ij}$ locally to the conformally flat form $g_{ij} = f\eta_{ij}$ where $f>0 $ and $\diag(-e,1)$
is a constant matrix with $e=\pm 1$. Note that $\eta_{\mu \nu}$ is the metric of a flat Minkowski 4-dimensional spacetime, while $\eta_{ij}$ is the metric of its 2-dimensional subspace. $\eta^{ij}$ denotes the inverse of $\eta_{ij}$. We will denote the determinants of each of the $2\times 2$ blocks by
$\det (g_{ij}):= -ef^2,$  and  $\det (g_{ab}):= e\alpha^2$.  Their product is the determinant of the full metric tensor
$\det \tilde g_{\mu \nu}=-f^2\alpha ^2.$
	}:
\begin{equation}\label{intervalo}
	ds^2=f(t,x)(dx^2-dt^2)+g_{ab}(t,x)dx^adx^b.
\end{equation}
Recall that repeated indices mean {\color{blue} summation}, following the classical Einstein convention. Here, with a slight abuse of notation, we shall denote $g=g_{ab}$. Due to the axioms of general relativity, the tensor $g$ must be real and symmetric. {\color{blue} We will denote the determinants of $g$ by $ \det (g_{ab}):= \alpha^2,$ then,  the determinant of the full metric tensor $\tilde g$ is given by  $\det (\tilde g_{\mu \nu})=-f^2\alpha ^2.$}

In order to reduce the Einstein vacuum equations \eqref{EV}, one needs to compute the Ricci curvature tensor in terms of the components of the metric $g=g_{ab}$. The consideration of the metric in the form \eqref{intervalo} leads to components  $\text{R}_{0a}$ and $\text{R}_{3a}$ of the Ricci tensor that are identically zero. Therefore, one can see that Einstein's vacuum equations \eqref{EV} decompose into two sets of equations. The first one follows from $\text{R}_{ab} = 0$; this equation can be written as the single tensor equation
\begin{equation}\label{eq:PCE2}
	\partial_t\left(\alpha \partial_t g g^{-1}\right) -\partial_x\left(\alpha  \partial_x g g^{-1} \right)=0, \quad \det g =\alpha ^2.
\end{equation}
We shall refer to this equation as the \textit{reduced Einstein equation.} The trace of the equation \eqref{eq:PCE2} reads 
\begin{equation}\label{onda}
	\partial_{t}^2 \alpha -\partial_{x}^2 \alpha=0.
\end{equation}
Therefore, the function $\alpha(t,x)$ satisfies the 1D wave equation. {\color{blue}The equations \eqref{eq:PCE2}-\eqref{onda}} may be recast as equivalent to the “dynamical part” of the Einstein equations.  The second set of equations expresses the metric coefficient $f(t,x)$ in terms of explicit terms of $\alpha$ and $g$. 


Using inverse scattering techniques, Belinski and Zakharov \cite{belinskii1978integration} considered \eqref{eq:PCE2}, giving the first foundational results, see also \cite{Zakharov1979}. They proposed the application of the inverse scattering method to the equations of general relativity and the procedure of calculating exact solitonic solutions of the equations. They introduce a Lax-pair for \eqref{eq:PCE2}-\eqref{onda}, together with a general method for solving it. Localized and multi-coherent structures were found, but they are not solitons in the standard sense, unless $\alpha$ is constant. A more in-depth study on the subject is also made in \cite{belinskii1978integration, belinski2001gravitational}. More recently, Hadad \cite{yaronhadad_2013} explored the Belinski-Zakharov transformation for the 1+1-dimensional setting, obtaining explicit formulae for solutions constructed on arbitrary diagonal backgrounds, in particular on the Einstein-Rosen background. 


 The local behavior of the spacetime described before is defined by the function $\alpha$. In our setting, $\alpha$ will be an always positive and bounded function. These characteristics will be provided by  the initial conditions that will be imposed on the problem. The gradient of the function $\alpha(t,x)$ can be \textit{timelike, spacelike} or \textit{null}. The case where $\alpha$ is spacelike everywhere in spacetime $((\partial_x \alpha)^2-(\partial_t \alpha)^2 >0)$  corresponds to spacetimes said ``with cylindrical symmetry'', which corresponds to  the Einstein-Rosen spacetime, for example. They give an approach to the description of gravitational waves. When the gradient of $\alpha$ is globally null, $((\partial_x \alpha)^2-(\partial_t \alpha)^2 =0)$, it corresponds to the plane-symmetric waves. Finally, the last case, when the gradient of $\alpha$ is globally timelike  $((\partial_x \alpha)^2-(\partial_t \alpha)^2 <0)$ is used to describe cosmological models and colliding gravitational waves, see \cite{carmeli1984einstein,belinski2001gravitational, Belinsky, Einstein1937}.  It is precisely the {\color{blue} timelike case} that is the focus of this work. This classification for the gradient of the function $\alpha$ is necessary in order to propose an appropriate definition of energy and to be capable of giving a description of the decay of the solution associated with the system.

 In a previous work \cite{JT}, one of us considered the case when $\alpha$ is a constant function. Such consideration simplified the system \eqref{sistema1} and identified it with the Principal Chiral Field model (PCF). This approach allowed us to give a first global existence result and local decay in space. It should be noted that, in the case of constant $\alpha$, the results obtained cannot be extrapolated to the case of the Einstein equation in vacuum since essentially PCF is not exactly the case $\alpha=const.$ in \eqref{sistema1}, but instead one has to completely eliminate the equation for $f$. Recently \cite{AMT}, the case of nondegenerate solitons of the PCF equation was considered, and stability was proved in a class of decaying perturbations satisfying certain small energy conditions. 

{\color{blue}
\subsection{Main results, rough version}
A different situation is obtained when considering the case in which $\alpha(t,x)$ is a more general function; in this case, the results are completely identifiable with the Einstein equation, so it automatically becomes a more interesting and complex problem to analyze. Unfortunately, we are forced to consider only half of the $\alpha$ axis because, in general, the points $\alpha=0$  correspond to the physical singularity through which the metric cannot be extended \cite{Belinsky, belinski2001gravitational}. Therefore, we will assume $\alpha$ closer to the value 1. Our main results are a detailed long time description of the Einstein-BZ equations, specially from the point of view of the long time asymptotics, and are stated as follows.

\begin{thm}\label{MT000}
Assume that $\alpha$ is a suitable small perturbation of the constant state 1, and $g$ follows the geometric decomposition
\[
	g=\alpha\left( \begin{array}{cc}
		\cosh{\Lambda} +\cos 2\phi \sinh{\Lambda} &  \sin 2\phi \sinh{\Lambda} \\
		\sin 2\phi \sinh{\Lambda}     &  \cosh{\Lambda}-\cos 2\phi \sinh{\Lambda} 
	\end{array} \right),
\]
for unknown functions $\Lambda,\phi$, with $\Lambda$ far away from zero. Then the following are satisfied:
\begin{enumerate}
\item \eqref{eq:PCE2}-\eqref{onda} has globally defined small solutions. 
\item If additionally $\alpha$ is time-like, in the sense that 
\[
\begin{aligned}
& \alpha(t,x)>0 ,\quad \partial_t \alpha(t,x) >0,\\  
& (\partial_t \alpha)^2(t,x)-(\partial_x \alpha)^2(t,x)<0, \quad \forall (t,x)\in [0,\infty)\times \mathbb R. 
\end{aligned}
\]
any globally defined and sufficiently decaying solution $(\Lambda,\phi)$ has a conserved energy. This energy differs from the classical energy and momentum, which are zero for the Einstein's field equations.
\item Every bounded in time, finite energy solution to \eqref{eq:PCE2}-\eqref{onda} with $\alpha>c_0>0$ and $\partial_t\alpha$ rapidly decaying, uniformly in time, satisfies the following decay estimate: for each $|v|<1$,
\[
\lim_{t\longrightarrow +\infty} \int_{|x-vt|\lesssim t(\log t)^{-2} } \left[ (\partial_t\Lambda)^2+(\partial_x\Lambda)^2+\sinh^2(\Lambda)\Big((\partial_t \phi)^2+(\partial_x\phi)^2\Big) \right] dx=0,
\]
under the $(\Lambda,\phi)$ coordinates.
\end{enumerate}
\end{thm}

The previous statement shows that Einstein's equations under the Belinski-Zakharov symmetry ansatz obey a well-defined energy-like structure, with escape of energy to infinity. This is proved by classical virial methods. The main results in their explicit form are further developed in Section \ref{Sec:2}. Item (1) above corresponds to Theorem \ref{GLOBAL0}; item (2) corresponds to Theorem \ref{MT20}, with energy given by \eqref{EnergyF}, and item (3) is developed in Theorem \ref{MT2}. Some applications to certain soliton solutions of Kasner type are given in Theorem \ref{MT3}. 
}

\subsection{Discussion}\label{Disc}

{\color{blue} Let us discuss here the validity of the BZ symmetry ansatz. Salam and Strathdee \cite{Salam} discussed black holes as possible solitons. It is important to recall that the structure of the metric \eqref{intervalo} is not restrictive, since, from the physical point of view, we find many applications that can be described according to (\refeq{intervalo}). Such spacetimes describe {\color{blue} some} cosmological solutions of general relativity, as well as gravitational waves and their interactions \cite{belinski2001gravitational}. Among them, one can find
 \begin{itemize}
 \item classical solutions of the Robinson-Bondi plane waves \cite{BONDI1957}, 
 \item the Einstein-Rosen cylindrical wave solutions and their two polarization generalizations \cite{Einstein1937,carmeli1984einstein}, 
 \item the homogeneous cosmological models of Bianchi types I--VII including the Kasner model \cite{Kasner1922},
 \item (in the ``static'' setting) the Schwarzschild and Kerr solutions, and Weyl axisymmetric solutions, see Section 8.3 in  \cite{belinski2001gravitational}, 
\item 2-solitons, corresponding in a particular case to the Kerr-NUT (Newman-Unti-Tamburino) black-hole solution of three parameters, including Kerr, Schwarzschild, and Taub-NUT metrics \cite{TM15}.
\end{itemize} 

For additional bibliography, the reader may consult \cite{letelier1985static,letelier1986,krasinski2006inhomogeneous} and references therein. All this shows that, despite its relative simplicity, a metric of the type \eqref{intervalo} encompasses a wide variety of physically relevant compact objects. Additionally, Belinski-Zakharov metrics contain the so-called Gowdy spacetimes \cite{Gowdy,Moncrief}, where the initial topology differs from our setting. See also Section 4.1 in \cite{belinski2001gravitational} for a deeper discussion.

There are actually interesting connections between gravisolitons, from the point of view considered by Belinski and Zakharov, and the LIGO results (the study of gravitational waves appearing from the merging of two compact massive objects).  An important point to remark here is that many black hole solutions can be recovered as solutions to BZ equations (see the reply below). Indeed, BZ metrics can be classified depending on the signature of the parameter $\alpha$: in the time-like case, or cosmological (according to BZ), one finds Kasner and other more complex solutions, usually called soliton or multisoliton solutions issued from a certain background metric \cite[Ch. 2 and 4]{belinski2001gravitational}. There is an ADM formulation in this case and a time direction that makes things treatable using virial techniques, as in our paper.  In the case $\alpha$ space-like \cite[Ch. 8]{belinski2001gravitational}, one finds many black hole solutions, such as Schwarzschild, Kerr, and other classical black hole solutions (see also the answer below).  These connections point out the fact that Einstein's field equations, under the Belinski-Zakharov (BZ) symmetry, in the case $\alpha$ space-like, have solutions resembling the static interaction of several solitons, in this case, multi-black-holes, including the Tomimatsu–Sato black-hole solutions \cite[p. 216]{belinski2001gravitational}. See also \cite{Villa} for a review in this direction.

Therefore, similar to the $N$-soliton solution in KdV, Einstein's field equations under the BZ symmetry have $N-$black-hole solutions. An important point to remark is that some of these solutions are not asymptotically flat at spatial infinity.} {\color{teal}However, from the physical point of view, these solutions are interesting in physics and mathematical physics. Non-asymptotically flat $N$-black hole solutions generated by Belinski-Zakharov (BZ) transforms have significant physics interest, particularly in higher dimensions or theories with extra fields (like scalar fields), revealing complex geometries not possible in 4D flat spacetime. See, for example the case of \textit{rotating black lens solutions} \cite{Chen-T} and the Pomeransky \cite{Pome} improvement of the BZ method. Indeed, in the last reference, the study of solutions that are higher dimensional compact objects is motivated by theories with a number of compact dimensions, and it is stated that ``it is easy to see, that in five-dimensional space-time the analogue of Schwarzschild solution, the Tangherlini solution, cannot be obtained as a 2-soliton
solution on flat Minkowski background''.} {\color{blue} In \cite{JKS}, the authors propose that the dynamics of binary black holes may exhibit universal features linked to integrable structures, especially from the point of view given by the impressive ``near to integrable'' properties presented by linearized Master equations in Kerr (Teukolsky) and Schwarzschild (Reege-Wheeler, Zirelli) black holes. Also, the authors conjecture that gravitational waveforms from mergers can be understood through hidden symmetries and near-integrability of the underlying equations. Ablowitz and Segur \cite[p. 336]{AS} mention the BZ formalism as one of the most promising applications of the Inverse Scattering Transform, in view that the Kerr solution can be obtained via suitable B\"acklund transformations, already worked by various researchers in the field of dispersive PDEs. In \cite{Villa}, the author argues that gravitational solitons differ from standard nonlinear solitons in several aspects, including new phenomena such as multi-soliton coalescence, a phenomenon that emits low-amplitude radiation. Additionally, the ``pair-of-pants'' solution for the fusion of two black holes can be interpreted in such a way, in the sense that the two compact objects merging may be suitably represented by this mathematical theory. {\color{blue} In a forthcoming work \cite{MT3}, we study the long time behavior of perturbations of Kasner like solutions in the 3D setting under the Belinski-Zakharov ansatz, and provide an interesting complementary point of view of the work \cite{Igor2023}. There are interesting connections with the work of Lenells-Mauersberger \cite{LM}, precisely from the point of view of the Ernst equation.}

 The study of Einstein's field equations has a long history of important developments. Choquet-Bruhat \cite{CB,CB2} gave a foundational mathematical description of the evolution of initial data. A complete mathematical understanding of well-known black holes has taken many years. The stability of the Kerr black hole was recently obtained in a series of works by Klainerman, Szeftel and Giorgi \cite{klainerman5,klainerman2,klainerman3,klainerman4,Giorgi}. In the case of the Schwarzschild black hole, Dafermos, Holzegel, Rodnianski and Taylor \cite{Dafermos1,Dafermos2,Dafermos3} showed codimensional stability and asymptotic stability. Finally, Hintz and Vasy \cite{HV} proved nonlinear stability of Kerr under de Sitter gravity.

In the study of nonlinear dispersive equations and solitary waves, virial-type identities have proven to be a central tool for establishing local asymptotic stability of solitary waves. This method was initiated in the seminal work of Martel and Merle on generalized KdV equations \cite{MM01,MM02,MM05,MM09,MM15}, where suitable localized virial estimates were combined with monotonicity arguments to control radiation and isolate the soliton component. Extensions of this approach to multi-solitons \cite{MMT12,MMT13} and the Zakharov–Kuznetsov equation \cite{CMPS}, have shown its robustness beyond one-dimensional settings. Altogether, these works illustrate the versatility of virial methods in nonlinear dispersive PDEs, providing a unifying framework for proving asymptotic stability in a variety of contexts.

}

\subsection*{Organization of this work} This paper is organized as follows. {\color{blue} Section \ref{Sec:2} presents detailed versions of the main results of the work.} Section \ref{LE}  presents a summary of the local existence result for system \eqref{sistema1}, which relies, as in \cite{JT}, on a particular energy estimate. In the Section \ref{Sect:3} we prove the small initial data global existence result, namely Theorem \ref{GLOBAL0}. Section \ref{Sect:4} is focused on presenting a formalism suitable for the energy and momentum densities for \eqref{sistema1}, in the particular case of cosmological type solutions. Then in Section \ref{Sect:5b} we present and prove the long-term behavior result, Theorem \ref{MT2}. Finally, Section \ref{Sect:5} is devoted to an application in the case of Kasner metrics.

\subsection*{\color{blue} Acknowledgments} Part of this work was done while the second author visited U. Paris-Saclay (France), U. C\'ordoba (Spain) and Georgia Tech (USA). She thanks Profs. Fr\'ed\'eric Rousset, Jacques Smulevici, J\'er\'emie Szeftel, Miguel A. Alejo, Magdalena Caballero, Gong Chen, Luca Fanelli, and Paola Rioseco for stimulating discussions and very useful comments. The first author deeply thanks Miguel A. Alejo for fruitful discussions and comments, and Banff Center (Canada) where part of this work was done. {\color{blue}We deeply thank the referee for his/her constructive and important comments and criticisms that helped to improve this paper to its current version.}

\section{Main results}\label{Sec:2}

{\color{blue} In order to explicitly state our main results, we recall the geometric coordinates introduced by Belinski and Zakharov.}

\subsection{Geometric coordinates} The fact that the $2\times 2$ tensor $g$ is symmetric allows one to diagonalize it for fixed $t$ and $x$. One writes $g= RDR^{T}$, where $D$ is a diagonal tensor and $R$ is a rotation tensor, of the form
\begin{align}\label{DR}
	D=\left( \begin{array}{cc}
		\alpha e^{\Lambda}  & 0 \\
		0 &\alpha  e^{-\Lambda} 
	\end{array} \right), \qquad  R=\left( \begin{array}{cc}
		\cos \phi   & -\sin \phi \\
		\sin \phi & \cos \phi 
	\end{array} \right).
\end{align}
Clearly 
\begin{equation}\label{determinante}
\det g = \alpha^2.
\end{equation}
Here $\Lambda$ is the scalar field that determines the {\color{blue}eigenvalues} of $g$, and the scalar field $\phi$ determines the deviation of $g$ from being a diagonal tensor. Since $\phi$ is considered as an angle, we assume $\phi\in [0,2\pi]$. Therefore, $\Lambda, \phi$ and $\alpha$ in \eqref{DR} can be considered as the three degrees of freedom in the symmetric tensor $g$, \cite{yaronhadad_2013}. Written explicitly, the tensor $g$ is given now by
\begin{equation}\label{diag1}
	g=\alpha\left( \begin{array}{cc}
		\cosh{\Lambda} +\cos 2\phi \sinh{\Lambda} &  \sin 2\phi \sinh{\Lambda} \\
		\sin 2\phi \sinh{\Lambda}     &  \cosh{\Lambda}-\cos 2\phi \sinh{\Lambda} 
	\end{array} \right).
\end{equation}
Some analog representations have been used in various associated results, for example, in the Einstein-Rosen metric \cite{carmeli1984einstein}.	Note that Minkowski $\widetilde g_{\mu\nu}=(-1,1,1,1)$ can be recovered by taking {\color{blue}$f=1,$} $\Lambda=0$, $\alpha=1$ and  $\phi$ free.  The equation (\refeq{eq:PCE2}) reads now
\begin{equation}\label{sistema1}
	\begin{cases}
		\partial_t(\alpha \partial_t\Lambda) - \partial_x(\alpha \partial_x\Lambda) = 2\alpha  \sinh{2\Lambda}((\partial_t\phi)^2-(\partial_x\phi)^2), \\
		\partial_t(\alpha \sinh^2 \Lambda \partial_t\phi )- \partial_x(\alpha \sinh^2 \Lambda \partial_x\phi )=  0,\\
		\partial_{t}^2\alpha-\partial_{x}^2\alpha= 0,\\
		\partial_{t}^2(\ln f)- \partial_{x}^2(\ln f) = G,
	\end{cases}
\end{equation}
where $G=G[\Lambda, \phi, \alpha]$ is given by 
\begin{equation}\label{sistema2}
	\begin{aligned}
		& ~{} G := - \left( \partial_{t}^2(\ln \alpha)-\partial_{x}^2(\ln \alpha) \right)- \dfrac{1}{2\alpha^2} ((\partial_t\alpha)^2-(\partial_x\alpha)^2) \\
		&~{} \qquad  - \dfrac{1}{2}((\partial_t\Lambda)^2-(\partial_x\Lambda)^2)- 2\sinh^2 \Lambda((\partial_t\phi)^2-(\partial_x\phi)^2).
	\end{aligned}
\end{equation}
Note that the equation for $\alpha(t,x)$ is the standard one-dimensional wave equation, and can be solved independently of the other variables. Also, given $\alpha(t,x)$, $\Lambda(t,x)$, and $\phi(t,x)$, solving for $\ln f(t,x)$ reduces to use d'{} d'Alembert's formula for linear one-dimensional wave with nonzero source term. Consequently, the only nontrivial equations in \eqref{sistema1} are given by the equations for $\Lambda(t,x)$ and $\phi(t,x)$, for $\alpha$ solution to linear 1D wave equation. {\color{blue} Einstein's field equations provide another restriction in t, which is redundant with \eqref{sistema1}. See \cite{MT3} for a deeper discussion.}

\medskip

As one can see from \eqref{sistema1}, solutions are not unique. These fields satisfy the gauge invariance
\begin{equation}\label{gauge}
\begin{aligned}
	& (\Lambda,\phi ,\alpha, f) \quad \hbox{solution, then} \\
	&  \left(\Lambda,\phi  + k\pi ,C_1\alpha, C_2 f\right) \quad \hbox{is also solution}, \qquad k\in\mathbb Z, \quad C_1 ,C_2>0.
\end{aligned}
\end{equation}
Since $\alpha \mapsto C_1 \alpha$ is just a conformal transformation in \eqref{diag1}, {\color{blue} we shall consider solutions $\alpha$ with asymptotic value 1 at infinity,  and $f$ with asymptotic value $c_1>0$ at infinity} in \eqref{gauge}. It should be noted that although  \eqref{sistema1} are strictly non-linear in the fields $\Lambda(t,x)$, $\phi(t,x)$, $\alpha(t,x)$ and $f(t,x)$, it shares many similarities with the classical linear wave and Born-Infeld equations \cite{Alejo2018}: given any $\mathcal{C}^2$ real-valued profiles $h(s), k(s), \ell(s), m(s)$, then the following functions are solutions for \eqref{sistema1}:
\begin{equation*}
	\begin{aligned}
		&\Lambda(t,x)=h(x\pm t), \qquad \phi(t,x)=k(x\pm t),\\
		&\alpha(t,x)= \ell (x\pm t), \qquad  f(t,x)= m(x\pm t) .
	\end{aligned}
\end{equation*}
This property will be key when establishing the connection between the local theory that will be presented in the following section and the analysis of explicit solutions to the equation in Section \refeq{Sect:5}.

\subsection{Main results}

Our first result in this paper is the global existence of solutions. For (\refeq{sistema1}), we consider constraints on the initial conditions for $\alpha(t,x)$. Using the d'Alembert formula, we have an explicit expression for $\alpha$ that allows us to obtain tight control over appropriate terms by also using the central structure related to null forms. Although the nonlinearity is not purely defined in terms of null forms, we can follow and adapt properly in the case of variable coefficients the weighted energy estimates proposed in \cite{Luli2018} to approach the problem and finally obtain a small data global existence result for \eqref{sistema1}.

\begin{thm}[Small data global existence]\label{GLOBAL0}
	Let $\lambda>0, c_1>0$ be fixed, and set 
	\begin{equation}\label{tLambda_talpha}
	\Lambda=:  \lambda + \tilde{\Lambda}, \quad \hbox{and} \quad \alpha:=1+\tilde{\alpha}.
	\end{equation}
	Consider the wave system \eqref{sistema1} posed in $\mathbb{R}^{1+1}$, with the following initial conditions:
	\begin{equation}
		\text{(IC)} \quad \begin{cases}
			(\phi,\tilde{\Lambda},\alpha,f)|_{\{t=0\}}= (\varepsilon\phi_0,\varepsilon\tilde{\Lambda}_0,1+\tilde{\alpha}_0, c_1+ f_0), \\
			(\partial_t\phi,\partial_t\tilde{\Lambda},\partial_t\alpha, \partial_t f)|_{\{t=0\}}= (\varepsilon\phi_1,\varepsilon\tilde{\Lambda}_1, \alpha_1, f_1),\\
			(\phi_0,\tilde{\Lambda}_0,\tilde{\alpha}_0,f_0)\in \left( C_c^{\infty}(\mathbb{R}) \right)^4,\\ 
						(\phi_1,\tilde{\Lambda}_1,\alpha_1,f_1) \in C_c^{\infty}(\mathbb{R})\times C_c^{\infty}(\mathbb{R}) \times \mathcal{S}(\mathbb{R}) \times \mathcal{S}(\mathbb{R}).  
		\end{cases}\label{Ciniciales}
	\end{equation} 
	Assume the following bounds on the initial conditions:
	\begin{enumerate}
		\item	$\alpha_1(\cdot) >0$,
		\item $ \max_{n=0,1,2} \left( \| \partial_x^{(n)}\tilde{\alpha}_0\|_{\infty} + \| \partial_x^{(n)} \alpha_1\|_{\infty}\right)< \frac12 \gamma$, 
		where $\gamma$ is a fixed sufficiently small constant, but independent on $\varepsilon$. 
		\item  $||f_0||_{\infty}\leq \frac{c_1}{2}$,
		\item the initial data satisfy the compatibility conditions required by Einstein's field equations.
	\end{enumerate}
	Then, there exists $\varepsilon_0$ sufficiently small such that if $\varepsilon < \varepsilon_0$, the unique solution remains smooth for all time. 
\end{thm}
\begin{rem}
Note that the conditions on $\alpha_1$ and $f$ are less demanding than the ones required for $\alpha_0$. Indeed, one only needs data in the Schwartz class $S(\mathbb R)$ and compact support is not necessary; this will be useful in some applications. 
\end{rem}


{\color{blue} Regarding the particular setting of $1+1$ dimensions, we have an added difficulty, since waves do not decay as good as {\color{blue} as occur in higher dimensions} \cite{klainerman, Christodoulou1986, lindblad2008,lindblad2010, alinhac2001null}, specially in the case of data with no fast decay. However, the special structure in the nonlinearity can give rise to important results related to the asymptotic behavior of solutions, as in the case of the wave map \cite{Gu}. In a recent paper {\color{blue}\cite{Luli2018}, the authors} used weighted estimates for linear waves in $\mathbb{R}^{1+1}$, and the \textit{null condition}\footnote{For the forthcoming analysis it is it is convenient to introduce a fundamental null form, which is defined as the following bilinear form:
	\begin{equation*}
		Q_0 (\phi,\tilde{\Lambda})= m^{\alpha \beta}\partial_{\alpha} \phi \partial_{\beta} \tilde{\Lambda},
	\end{equation*}
	where $m_{\alpha \beta}$ to denote the standard Minkowski metric on $\mathbb{R}^{1+1}$.}, to construct global solutions for the associated nonlinear equation. These energy estimates allowed them to improve the decay on the null form. Furthermore, D. Zha in \cite{Zha} extended {\color{blue} the analysis in \cite{Luli2018}} to the quasilinear case; in this analysis, {\color{blue} is considered} the previous weighted energy estimates of Luli, Yang and Yu with positive weights and introduces some space-time weighted energy estimates.}

\begin{rem}
It is important to emphasize that the system \eqref{sistema1} differs from the corresponding models studied in references \cite{Luli2018} and \cite{Zha}. First of all, our model is {\color{blue} directly} motivated by Einstein's field equations under Belinski-Zakharov's ansatz in the cosmological setting, having direct physical meaning. While the weighted-energy estimates proposed by Luli,  Yang and Yu in \cite{Luli2018} are employed in the development of our analysis, we require further studies since the structure of the nonlinearity in \eqref{sistema1}  does not properly define a null form, and has additional linear terms that change the nature of the problem, making it different from the classical quadratic null-structure. On the other hand, our nonlinearity does not fit into the assumptions imposed in \cite[page 2, condition (1.7)]{Zha}. More precisely, in our case, we require new conditions on the initial data of the function $\alpha$ representing the determinant of the Riemannian component of the metric. Very importantly, without these conditions, one cannot ensure the essential cosmological nature of the problem, a condition not represented in previous works. These conditions are also necessary for the estimation of the novel term $Q_0(\ln (\alpha),\tilde \Lambda)$. Finally, energy and momentum identities found below are new and not present in the previosu literature, as far as we understand. In other words, the Belinski-Zakharov's setting induces new conserved/almost conserved quantities (or densities) not present in the general Einstein's field equations under general conditions. This was only present before in linearizations of the field equations (cf. e.g. the Regge-Wheeler and Zerilli linear models in linearized Schwarzschild dynamics), but recall that in this case, the equations are purely nonlinear, meaning the conservation are truly nonlinear as well.
\end{rem}

Recall that $\alpha$ is a solution to the linear wave equation in 1D but \emph{far from zero}. Along the paper, we will see that this condition is necessary and natural in view of \eqref{determinante}. Consequently, one only expects decay in the $\dot H^1\times L^2$ norm, precisely as in \cite{Alejo2018}. A direct consequence of Theorem \ref{GLOBAL0} is the global existence of the Belinski-Zakharov metric \eqref{intervalo}:

\begin{cor}
Under the assumptions in Theorem \ref{GLOBAL0}, $g$ and $f$ in \eqref{intervalo} are globally well-defined. 
\end{cor}

\begin{rem}[On the applicability of previous results]\label{rem2p4}
{\color{teal}
Recently, Fournodavlos-Rodnianski-Speck \cite{Igor2023} considered the case of cosmological perturbations of the Kasner family in several dimensions, proving full stability of this family in higher dimensions (vacuum case), and stability under symmetries in the 3D (scalar field case).  In this work, we assume initial conditions associated with the fields $\Lambda$, $\alpha$, $\phi$ and $f$ in the BZ symmetry that allow us to describe the spatial metric $\mathring{g}$ on the hypersurface associated with $t\geq 0$. It is of independent interest to address a case that is, in some sense, complementary, especially in the 3D subregime and under the Belinski and Zakharov symmetry, and possibly under codimensional restrictions.}

\end{rem}

\begin{rem}
The problem of global existence under general data size is delicate. Indeed, \eqref{sistema1} is clearly singular in the case where $\Lambda$ reaches the zero value. Also, the global structure of the solution in this case seems not clear unless one has further assumptions on the initial data. These geometric problems will be considered in a forthcoming publication.
\end{rem}

The second result in this work concerns the decay of a specific type of the solutions of the Einstein equations in the vacuum. Specifically of cosmological-type solutions, which are of special interest in physics and cosmology. This type of solutions includes the Kasner type spacetimes, as well as some Bianchi type models, see \cite{belinski2001gravitational}. We will prove, using well-chosen virial estimates that for solutions to \eqref{sistema1} with finite energy (in particular, globally defined small solutions from Theorem \ref{GLOBAL0}), they must decay to zero locally in space, provided that the gradient of the function $\alpha(t,x)$ is globally timelike. 

\medskip

Indeed, virial functionals can describe in great generality the decay mechanism for models where standard scattering is not available, either because the dimension is too small, or the nonlinearity is long range, see e.g. \cite{Virial1, Virial2}.  We will prove this result, inspired by the results obtained for the Born-Infeld equation in 1+1 dimensions \cite{Alejo2018}. 
\medskip

Before proving this result, we introduce the following modified energy of the system, which in the case of cosmological type solutions will be highly relevant (see Section \ref{Sect:4}):
\begin{equation}\label{EnergyF}
E[ \Lambda, \phi; \alpha](t):= -\int [\kappa \partial_t \alpha ( h_1-2  h_2)](t,x)dx,
\end{equation}
where $\kappa(t,x)=\dfrac{\alpha }{(\partial_x\alpha)^2-(\partial_t\alpha)^2}$,
\begin{equation}\label{h1_0}
 h_1(t,x)= (\partial_t \Lambda)^2+ (\partial_x \Lambda)^2+4\sinh^2(\Lambda)((\partial_x \phi)^2+(\partial_t \phi)^2),
\end{equation}
and
\begin{equation*}
h_2(t,x)= \partial_t\Lambda \partial_x\Lambda +4\sinh^2(\Lambda)\partial_t\phi \partial_x\phi.
\end{equation*}
This (nonconserved) energy is a modified version of the one introduced by Hadad \cite{yaronhadad_2013}, which was not sufficiently useful for our purposes. Here \eqref{EnergyF} has important modifications to ensure the positivity of the energy functional. Compared with our previous results \cite{JT} in the case of the Principal Chiral equation, here the energy and momentum terms require deeper understanding and much more work than before.

\medskip

For this theorem, we shall assume the cosmological condition
\begin{equation}\label{cosmological type}
\begin{aligned}
& \alpha(t,x)>0 ,\quad \partial_t \alpha(t,x) >0,\\  
& (\partial_t \alpha)^2(t,x)-(\partial_x \alpha)^2(t,x)<0, \quad \forall (t,x)\in [0,\infty)\times \mathbb R. 
\end{aligned}
\end{equation}

\begin{thm}[Existence of a modified energy]\label{MT20}
Let $(\Lambda, \phi, \alpha)(t)$ be a smooth solution of the system \eqref{sistema1} such that  $\alpha$ satisfies \eqref{cosmological type}. Then the modified energy $E[\Lambda, \phi; \alpha](t)$ is well-defined and nonnegative.
\end{thm}

Recall that the existence of a suitable energy is one of the key elements needed to study long-time behavior in Hamiltonian-type systems. In our setting, the energy $E$ will not be preserved in time, but under suitable conditions, already satisfied by solutions in Theorem \ref{GLOBAL0}, it will be bounded in time. The following remark clarifies this point: 

\begin{rem}[On the cosmological type condition]
Condition \eqref{cosmological type} is not empty. Indeed, in the case of small data as in Theorem \ref{GLOBAL0}, a sufficient condition to ensure \eqref{cosmological type} is that 
\begin{equation*}
|\alpha_0'(x) | < \alpha_1(x), \quad \forall x\in\mathbb R. 
\end{equation*}
This condition is in concordance with \eqref{Ciniciales}, where $\alpha_1$ has been chosen to belong to a not compactly supported space.
\end{rem}

Now we are ready to state the result that we consider the most important in this work, explicit version of Theorem \ref{MT000}.

\begin{thm}[Decay of cosmological finite-energy spacetimes]\label{MT2}
Under the hypotheses in Theorem \ref{MT20}, assume in addition that one has 
\begin{enumerate}
\item[(a)] bounded energy condition:
\begin{equation}\label{condition00}
\sup_{t\geq 0} E[\Lambda, \phi; \alpha](t) <+\infty;
\end{equation}
\item[(b)] for some $c_0>0$ one has 
\begin{equation}\label{uniform}
\alpha(t,x) >c_0 \quad  \hbox{and} \quad  \partial_t\alpha \hbox{ is in the Schwartz class uniformly in time}.
\end{equation}
\end{enumerate}
Then, for any $v\in\mathbb R$, $|v|<1$, and $\omega(t)= t(\log t)^{-2},$ one has
\begin{equation}\label{desg200}		
\lim_{t\longrightarrow +\infty} \int_{|x-vt|\lesssim \omega(t) } \left[ (\partial_t\Lambda)^2+(\partial_x\Lambda)^2+\sinh^2(\Lambda)\Big((\partial_t \phi)^2+(\partial_x\phi)^2\Big) \right] dx=0.
\end{equation}
\end{thm}

\begin{rem}[On the finite energy condition]
Globally defined solutions obtained from Theorem \ref{GLOBAL0} satisfy the finite energy condition \eqref{condition00} thanks to suitable weighted estimates. Moreover, they also satisfy \eqref{uniform} in the case where the first line in \eqref{cosmological type} is satisfied. In that sense, Theorem \ref{MT2} is more general and might be satisfied by large solutions, as explained in Section \ref{Sect:5}, where applications to Kasner spacetimes are presented.
\end{rem}

A simple corollary in terms of the spacetime tensor $g$ can be obtained:

\begin{cor}
Under the hypotheses in Theorem \ref{MT2}, one has that $g$ in \eqref{intervalo} satisfies
\begin{equation}\label{desg2000}		
\lim_{t\longrightarrow +\infty} \int_{|x-vt|\le \omega(t) } \left( (\partial_t \det g)^2 +(\partial_x \det g)^2  \right)(t,x) dx=0.
\end{equation}
\end{cor}
Vanishing property \eqref{desg2000} can be understood as the manifestation that the spacetime is of cosmological type, and information propagates with the speed of light, supported on the light cone.

\medskip

\noindent
{\bf {\color{blue} Applications} to gravisolitons.} One of the motivations of Belinski and Zakharov was to show the existence of gravitational solitons (gravisolitons). From the mathematical point of view, these are of solitonic type, although they exhibit a number of features unusual in this type of solutions \cite{belinski2001gravitational}. In this paper, we apply Theorem \ref{GLOBAL0} and \ref{MT2} to the cosmological 1-soliton obtained from a {\em nonsingular generalized Kasner metric}, see \eqref{alfaalfa} and \eqref{Kasner1solitonLambda}-\eqref{Kasner1solitonPhi} for the explicit formula. In particular, we shall prove (Corollaries \ref{cor5p1} and \ref{cor5p2}):

\begin{thm}\label{MT3}
The cosmological 1-soliton $(\Lambda, \phi,\alpha)$ obtained from a nonsingular generalized Kasner metric of parameter $d\geq1$ is globally defined under suitable small perturbations in the case where $\alpha$ satisfies the hypotheses of Theorem \ref{GLOBAL0}, and satisfies
\[
\lim_{t\to +\infty }\int_{|x-vt| \leq \omega(t)}  \left[ (\partial_t\Lambda)^2+(\partial_x\Lambda)^2+\sinh^2(\Lambda)\Big((\partial_t \phi)^2+(\partial_x\phi)^2\Big) \right]  (t,x) dx =0.
\]
in the case where $\alpha$ is of cosmological type and satisfies the hypotheses of Theorem \ref{MT2}. 
\end{thm}

Notice that conditions in Theorem \ref{MT3} are essentially only depending on $\alpha$, and in some sense, this function determines the final behavior of solutions. The generalized Kasner metric discussed in Theorem \ref{MT3} avoids some undesirable bad behavior at the time origin, although we believe that standard Kasner metrics should satisfy a result similar to Theorem \ref{MT3}. 

\medskip
{\color{RoyalBlue}
\begin{rem}\label{rem2p8} 
We give now a {\bf brief} and incomplete description of related results where the symmetry assumptions are close to the ones that we use in this work. In the series of works \cite{Christo_91, Christo_94, Christo_99}, Christodoulou established a rigorous theoretical framework supporting the cosmic censorship conjecture (CCC) by analyzing the gravitational collapse of spherically symmetric scalar fields and showing that, while naked singularities may form, they are unstable under perturbations of the initial data.
Another important use of symmetries for the same problem is related with the proof that, for a generic initial data, the maximum global hyperbolic development (MGHD)\footnote{ Yvonne Choquet-Bruhat showed that it is possible to formulate the Einstein vacuum equations can be viewed as an initial value problem \cite{CB}, and given the initial data there is a part of spacetime, the so-called maximum global hyperbolic development (MGHD), which is uniquely determined up to isometry.} is inextensible. In vacuum, Ringstr\"om provided key censorship results in the framework of the Gowdy symmetry, see \cite{Ringstrom2, Ringstrom}. These results on CCC are essentially confined to the vacuum case, while the use of symmetries in the case of Einstein's field equation with matter and/or other fields such as  in \cite{DaRe0,DaRe1,DaRe2} has not been included here by length reasons, but it is extensive and numerous. See the previous results and references therein for full details.


In a different problem and point of view, Smulevici studied the same issue for $\mathbb{T}^2$-symmetric\footnote{A spacetime $(M, g)$ is said to be $\mathbb{T}^2$-symmetric if the metric is invariant under the action of the Lie group $\mathbb{T}^2$ and the group orbits are spatial. These solutions constitute a class of spacetimes admitting a torus action.} spacetimes with positive cosmological constant \cite{Smulevici}.  From a different perspective, in the series of papers \cite{LeFloch, LeFloch1, LeFloch2} the
$\mathbb{T}^2$-symmetric spacetimes on $\mathbb{T}^3$ are considered. In these works, the components of the metric depend only on $t$ and $\theta$, while the variables $x, y,$ and $\theta$ describe $\mathbb{S}^1$ (the one-dimensional torus or circle), which is identified with the interval $[0, 2\pi]$ with periodic boundary conditions. These considerations allow them to propose a new energy functional on $\mathbb{S}^1$, inspired by the Gowdy-to-Ernst transformation\footnote{A large class of twisted Gowdy solutions of Bianchi type II can be obtained starting from ordinary Gowdy solutions using
the Gowdy-to-Ernst transformation. This transformation was introduced in the study of spikes in Gowdy spacetimes \cite{Rendall} 
and was later used in the study of the initial singularity \cite{Ringstrom3}.}.  Moncrief in \cite{Moncrief}, using $U(1)\times U(1)$ symmetry at $\mathbb{T}^3$, can control the global evolution of the solutions. If we consider a spacetime with two Killing fields but without compactification (e.g., cylindrical or flat coordinates without toroidal identification), as in our case, the analysis is different; the problem of conditions at infinity reappears, thus, decay conditions are needed for the energy to be finite.  Huneau \cite{Huneau2}, assuming the existence of a translational spacelike Killing field in the asymptotically flat case, studied the global existence of solutions to Einstein's equations. This hypothesis allowed her to move from the 3+1-dimensional problem to a 2+1-dimensional one and, from that, to rewrite the  Einstein equations in a suitable form. To control the slow decay in 2+1-dimensions, the metric solution is decomposed in a particular way. This decomposition includes the description of the components of the metric by means of cut-off functions.  This scheme includes Minkowski and Einstein-Rosen spacetimes, but not cosmological models such as Kasner-type spacetimes. 
It is important to remark that when considering $\mathbb{T}^2$-symmetry, one obtains a compact space without boundary, which avoids problems associated with conditions at infinity (as in asymptotically flat cases). This allows them to work with globally defined initial data, without having to impose decay conditions at infinity.

\end{rem} 
}

\begin{rem}\label{rem2p9}
{\color{teal}
The Einstein constraint equations constitute a problem of great interest as they give rise to a nontrivial system of elliptic equations, which has been studied from several interesting perspectives. Premoselli \cite{Premoselli}, using conformal method, obtains an admissible initial data for the conformal Einstein-scalar constraint system.  Huneau \cite{Huneau} established the existence of solutions for these compatibility conditions, assuming the existence of a translational spacelike Killing field in the asymptotically flat case. This hypothesis allowed her to move from the 3+1-dimensional problem to a 2+1-dimensional one and, from that, to rewrite the  Einstein constraint equations in a suitable form. Thus, it is possible to demonstrate the existence of initial data compatible with the constraint equations in this scenario. Also, in \cite{Huneau2}, in the same setting, Huneau studies the global existence of solutions to Einstein's equation. To control the slow decay in 2+1-dimensions, the metric solution is decomposed in a particular way. This decomposition includes the description of the components of the metric by means of cut-off functions.  
 This scheme includes Minkowski and Einstein-Rosen spacetimes, but not cosmological models such as Kasner-type spacetimes. }
\end{rem}


\section{Local existence}\label{LE}

Before presenting the proof of global existence for the system, it is important to make some remarks to convince us that we first have a theory of local existence for the system \eqref{sistema1}. The first thing that we need is to set the initial conditions for the one-dimensional wave equation for $\alpha$, which allow us to obtain a bounded and positive solution of this equation. These conditions are not only needed to establish the local existence, but also to obtain the global existence and to be subsequently able to make an analysis of the long-term behavior of the corresponding finite energy solution, as we will see in the further sections. In order to develop the results related to the local theory for the nonlinear wave equation, let us write the function  $\Lambda(t,x)$ in the form
	 \begin{equation}\label{tilda}
	 \Lambda(t,x):= \lambda + \tilde{\Lambda}(t,x), \quad \lambda \neq 0.
	 \end{equation}
	 Notice that this choice makes sense with the energy in \eqref{EnergyF}, in the sense that $\Lambda\in \dot H^1$ and $\partial_t\Lambda \in L^2$. Without loss of generality, we assume $\lambda>0$. We consider the following vector notation 
	  \begin{equation*}
	 	\begin{cases}\Psi=\left( \tilde{\Lambda},\phi \right),\quad 
	 		\partial \Psi= \left( \partial_t \tilde{\Lambda}, \partial_x \tilde{\Lambda}, \partial_t \phi, \partial_x \phi \right),\\
			\vspace{0,1cm}
	 		\abs{ \partial \Psi }^2= \big| \partial_t \tilde{\Lambda}\big|^2 +\big| \partial_x \tilde{\Lambda} \big|^2+\abs{\partial_t \phi }^2+\abs{\partial_x \phi}^2,\\
			\vspace{0,1cm}
			  F(\Psi,\partial \Psi)= \left( F_1,F_2\right),\\
	 		F_1(\Psi,\partial \Psi):= 2\sinh(2\lambda +2\tilde{\Lambda})\left((\partial_x \phi)^2-(\partial_t\phi)^2\right),\\
			\vspace{0,1cm}
	 		F_2(\Psi,\partial \Psi):= \dfrac{\sinh(2\lambda +2\tilde{\Lambda})}{\sinh^2(\lambda +\tilde{\Lambda})} \left(\partial_t \phi \partial_t \tilde{\Lambda} -\partial_x \phi \partial_x \tilde{\Lambda} \right).
	 	\end{cases}
	 \end{equation*}
	 With this notation, the initial value problem for \eqref{sistema1} can be studied by first focusing on the following initial-value problem for $(\Psi, \partial_t \Psi)$:
	 \begin{equation}
	 	\begin{cases*}
	 		\partial_{\mu} (m^{\mu \nu}\alpha\partial_{\nu}\Psi)=F(\Psi,\partial \Psi)\\
	 		(\Psi,\partial_t \Psi)|_{\{t=0\}}=(\Psi_0, \Psi_1) \in  \mathcal{H}.
	 	\end{cases*}\label{IVP}
	 \end{equation}
	 Here $m^{\mu \nu}$ are the components of the Minkowski metric with $\mu, \nu \in \left\{0,1\right\}$, and the function $\alpha : =1 +\tilde{\alpha},$ satisfies the following initial valued problem
	 \begin{equation}\label{IC_alpha}
	 \begin{cases}
	 \partial_{t}^2\alpha-\partial_{x}^2\alpha= 0\\
	 (\alpha, \partial_t \alpha)|_{t=0}= (1+\tilde \alpha_0, \alpha_1)\\
	 \mbox{with} \quad (\tilde \alpha_0, \alpha_1)\in C_c^{\infty}(\mathbb{R}) \times \mathcal{S}(\mathbb{R}).
	 \end{cases}
	 \end{equation}
	 Assume the following bounds on the initial conditions in \eqref{IC_alpha}:
	\begin{enumerate}
		\item	$\alpha_1(\cdot) >0$,
		\item $ \max_{n=0,1,2} \left( \| \partial_x^{(n)}\tilde{\alpha}_0\|_{\infty}\right) + \max_{n=0,1}\left( \| \partial_x^{(n)} \alpha_1\|_{\infty}\right)< \frac12 \gamma$, where $\gamma$ is a fixed sufficiently small constant, but independent on $\varepsilon$. 
	\end{enumerate}
	Notice that these conditions are already stated in Theorem \ref{GLOBAL0}. In addition, we will seek solutions in the space
	 \begin{equation*}
	(\Psi,\partial_t \Psi)\in \mathcal{H}:=H^1(\mathbb{R})\times H^{1}(\mathbb{R}) \times L^2(\mathbb{R}) \times L^2(\mathbb{R}).
	 \end{equation*}
	 Notice that from \eqref{tilda}, $\Lambda\in \dot H^1$. We are also going to impose the following condition on the initial data  
	 \begin{equation}\label{condicion1}
	 	\norm{\left(\Psi_0,\Psi_1\right)}_{\mathcal{H}} \leq \dfrac{\lambda}{2D}.
	 \end{equation}
	 where $D$ is a suitable constant. In order to state a local existence result for the initial value problems \eqref{IVP}, {\color{blue} it} is important to recall the following result \cite{sogge}:
\begin{lem}
Let $\psi: I \times \mathbb{R} \longrightarrow \mathbb{R}$, $I\subseteq \mathbb{R},$ be the solution of the  initial value problem
	\begin{equation*}
		\begin{cases*}
			\partial_{\mu} (a^{\mu \nu}\partial_{\nu}\psi)= f(t,x), \quad (t,x) \in I\times \mathbb{R},\\
			(\psi, \partial_t \psi)|_{\{t=0\}}=(\psi_0, \psi_1) \in H^{k}(\mathbb{R})\times H^{k-1}(\mathbb{R}),
		\end{cases*}
	\end{equation*}
	where  $k$ is a positive integer and $a$ and all its derivatives (of all orders) are bounded in $[0,T]\times \mathbb{R}$. Then for some positive constant $C=C(k),$ the following energy estimate holds   
	\begin{equation}\label{Eenergia}
	\begin{aligned}
		& \sup_{t\in [0,T]} \norm{ (\psi, \partial_t \psi)}_{\mathcal{H}} \\
		& \qquad \leq  C\left(  \norm{ (\psi_0,\psi_1) }_{H^{1}(\mathbb{R})\times L^{2}(\mathbb{R})}+ \int_{0}^{T} \norm{f}_{H^{k-1}(\mathbb{R})}(t) dt \right)\exp\left( C\int_0^T \norm{\partial a}_{L^{\infty}(\mathbb{R})}(t)\right).
		\end{aligned}
	\end{equation} 
\end{lem}	 
Now, we can propose the following result for the initial-value problem \eqref{IVP}:
 \begin{prop}\label{LOCAL}
	 	If $(\Psi_0, \Psi_{1})$ satisfies the condition  (\ref{condicion1}) with an appropriate constant $D\geq1$, then: 
	 	\begin{itemize}
	 		\item[(1)] \mbox{(Existence and uniqueness of local-in-time solutions)}. There exists 
	 		\[ 
			T=T\left( \norm{ \left( \tilde{\Lambda}_0, \phi_0 \right) }_{H^1(\mathbb{R}) \times H^1(\mathbb{R})}, \norm{ \left( \tilde{\Lambda}_1, \phi_1 \right) }_{L^2(\mathbb{R}) \times L^2(\mathbb{R})},\lambda \right) > 0,\]
	 		such that there exists a (classical) solution $\Psi$ to  (\refeq{IVP}) with 
	 		\begin{equation*}
	 			(\Psi,\partial_t \Psi)\in L^{\infty}([0,T];\mathcal{H}).
	 		\end{equation*}
	 		Moreover, the solution is unique in this function space.
			
			\medskip
			
	 		\item[(2)] \mbox{(Continuous dependence on the initial data)}. Let $\Psi_{0}^{(i)}, \Psi_{1}^{(i)}$ be sequence such that $\Psi_{0}^{(i)} \longrightarrow \Psi_{0}$ in $H^1(\mathbb{R})\times H^{1}(\mathbb{R})$ and $\Psi_{1}^{(i)} \longrightarrow \Psi_{1}$ in $L^2(\mathbb{R})\times L^2(\mathbb{R})$ as $i \longrightarrow \infty.$ Then taking $T>0$ sufficiently small, we have 
	 		\begin{equation*}
	 			\norm{ \left(\Psi^{(i)}-\Psi, \partial_t(\Psi^{(i)}-\Psi) \right) }_{L^{\infty}([0,T]; \mathcal{H})} \longrightarrow 0.
	 		\end{equation*}	 
	 		Here $\Psi$ is the solution arising from data $(\Psi_0,\Psi_1)$ and $\Psi^{(i)}$ is the solution arising from data $\left(\Psi_0^{(i)},\Psi_1^{(i)} \right).$
	 	\end{itemize}	 	
	 \end{prop}
\begin{proof}[Sketch of proof]
The idea of the proof is standard in the literature (see e.g. Luk \cite{Luk}), in this case we must identify the component of $a$ in \eqref{Eenergia} con $a^{00}=a^{11}:= \alpha(t,x)$, and $a^{01}=a^{10}=0.$ Then we can use the energy estimate  \eqref{Eenergia}. The rest of the proof, for this particular system, can be seen in detail in Proposition 1 \cite{JT}; only an adaptation in the estimation of energy to be used is required. This ends the proof of Proposition \ref{LOCAL}.
\end{proof}
\section{Proof of Global Existence -- Theorem \ref{GLOBAL0}}\label{Sect:3}

\subsection{Preliminaries} 
We start out this section present to some basic definitions, and certain important results that will be useful to describe our result. {\color{blue} We shall use the notation from \cite{Luli2018} and our previous work \cite{JT}}. We will use two coordinate systems: the standard Cartesian coordinates $(t,x)$ and the null coordinates $(u, \underline{u})$: 
\begin{equation}\label{u bar u}
	u:=\dfrac{t+x}{2}, \quad \underline{u}:= \dfrac{t-x}{2},
\end{equation}
and consider the two null vector fields defined globally as 
\begin{equation}\label{L Lbar}
L=\partial_t + \partial_x,\quad \underline{L}= \partial_t - \partial_x.
\end{equation}
In the  same way as in \cite{Alinhacbook,Luli2018, JT} we consider the weight function $\varphi$ defined as 
\begin{equation}\label{varphi}
\varphi(u):=(1+|u|^2)^{1+\delta}\quad  \mbox{with} \quad 0< \delta <1/3.
\end{equation}

Recall that from initial conditions \eqref{Ciniciales} we have $\alpha_0:= 1+ \tilde \alpha_0$ and also have the following facts, which are easy to check:
\begin{itemize}
	\item[(i)] Since $\tilde\alpha_0\in C_c^{\infty}(\mathbb{R})$ and $\alpha_1 \in \mathcal{S}(\mathbb{R})$, with $\alpha_1>0$, one has for some fixed constant $K_1,K_2>0$ such that,
	\begin{equation}\label{eqn:alpha0}
		||\tilde \alpha_0 ||_{\infty}< \frac{\gamma}{2}, \quad   \quad | \alpha_0^{(n)}(2u)| \leq \frac{K_1\gamma}{\varphi^{3/4}(u)}, \quad n=1,2,
	\end{equation}
	and 
	\begin{equation}\label{eqn:alpha1}
		|\alpha_1^{(n)}(2u)| \leq \frac{K_2\gamma}{\varphi^{3/4}(u)}, \quad  n=0,1.
	\end{equation}
	\item[(ii)] Using the classical {\color{blue}d'Alembert's} formula in the third equation in \eqref{sistema1} which correspond to one-dimensional wave equation for $\alpha$, we obtain :
	\begin{equation}\label{eqn:alpha_DA}
		\alpha(t,x)= \frac{1}{2}\left(2+\tilde \alpha_0 (2u)+\tilde \alpha_0(-2\underline{u})+\int_0^{2u} \alpha_1(s)ds-\int_{0}^{-2\underline{u}} \alpha_1(s)ds\right),
	\end{equation}
	\item[(iii)] Moreover, the derivatives of the function $\alpha$ can will be describe as:
	\begin{equation}\label{eqns:alpha}
		\begin{cases}
			\partial_x\alpha= \frac{1}{2}\left(\tilde \alpha_0'(2u)+\tilde \alpha_0'(-2\underline{u})+\alpha_1(2u)-\alpha_1(-2\underline{u}) \right)\\
			\partial_t\alpha=\frac{1}{2}\left(\tilde \alpha_0'(2u)-\tilde \alpha_0'(-2\underline{u})+\alpha_1(2u)+\alpha_1(-2\underline{u}) \right)\\
			\partial_{x}^2 \alpha= \frac{1}{2}\left(\tilde \alpha_0''(2u)+\tilde \alpha_0''(-2\underline{u})+\alpha_1'(2u)-\alpha_1'(-2\underline{u}) \right)\\
			\partial_{t,x}\alpha=  \frac{1}{2}\left(\tilde \alpha_0''(2u)-\tilde \alpha_0''(-2\underline{u})+\alpha_1'(2u)+\alpha_1'(-2\underline{u}) \right).
		\end{cases}
	\end{equation}
	\item[(iv)]  The following relations for the null vector field $L$ and $\underline{L}$ hold: 
	\begin{equation}\label{Llogalpha}
		\begin{aligned}
			|L(\ln \alpha)|\leq &~{}  |\partial_x\alpha +\partial_t \alpha|=|\tilde \alpha_0'(2u)+\alpha_1(2u)|\lesssim \frac{K_1 \gamma}{\varphi^{3/4}(u)}.\\
			|L(\partial_x(\ln \alpha))|\leq  &~{} \frac{1}{2} \left| \left(\tilde \alpha_0'(2u)+\tilde\alpha_0'(-2\underline{u})+\alpha_1'(2u)-\alpha_1'(-2\underline{u}) \right)\left(\tilde \alpha_0'(2u)+\alpha_1(2u) \right) \right|\\
			&~{}	+|\tilde\alpha_0''(2u)+\alpha_1'(2u)| \lesssim \frac{K_1 \gamma}{\varphi^{3/4}(u)}.\\
			|\underline{L}(\partial_x(\ln \alpha))|\leq  &~{} \frac{1}{2} \left| \left(\tilde \alpha_0'(2u)+\tilde\alpha_0'(-2\underline{u})+\alpha_1'(2u)-\alpha_1'(-2\underline{u}) \right)\left(\tilde \alpha_0'(2u)+\alpha_1(2u)\right) \right|\\
			&~{}	+|\tilde\alpha_0''(2u)+\alpha_1'(2u)| \lesssim \frac{K_1 \gamma}{\varphi^{3/4}(u)}.
		\end{aligned}
	\end{equation}
\end{itemize}
From now on, we will consider  the conformal Killing vector fields on $\mathbb{R}^{1+1}$ given by
\begin{equation*}
	(1+\abs{u}^2)^{1+\delta}L, \quad (1+\abs{\underline{u}^2})^{1+\delta}\underline{L},
\end{equation*}
with $0 < \delta << 1$. We also consider the following integration regions: 
\begin{itemize}
	\item $S_{t_0}$  denotes the following time slice in $\mathbb{R}^{1+1}$: 
	\begin{align*}
		S_{t_0} := \left\{(t,x):\; t=t_0 \right\}.
	\end{align*}
	\item $D_{t_0}$ denotes the following region of spacetime 
	\begin{equation*}
		D_{t_o}:= \left\{ (t,x):\;  0 \leq t \leq t_0 \right\}, \quad D_{t_0}=\bigcup_{0\leq t \leq t_0} S_{t_0}.
	\end{equation*}
\end{itemize}
The level sets of the functions $u$ and $\underline{u}$ define two global null foliations of $D_{t_0}$. More precisely, given $t_0>0$, $u_0$ and $\underline{u}_0$, we define the rightward null curve segment $C_{t_0,\underline{u}_0}$ as:
\begin{equation}\label{Ctuu}
	C_{t_0,\underline{u}_0} := \left\{(t,x): \; u=\frac{t-x}{2} =\underline{u}_0,\, 0\leq t \leq t_0\right\},
\end{equation}
and the segment of the null curve to the left $C_{t_0,{u}_0}$ as:
\begin{equation}\label{Ctu}
	C_{t_0,u_0} := \left\{(t,x): \; u=\frac{t+x}{2} =u_0,\, 0\leq t \leq t_0\right\}.
\end{equation}
The space time region $D_{t_0}$ is foliated by $C_{t_0,\underline{u}_0}$ for $\underline{u}\in \mathbb{R}$, and by $C_{t_0, u_0}$ for $u\in \mathbb{R}$. 

\medskip

Finally, we will consider the following energy estimate\footnote{\color{blue} Note that the proof of this energy estimate (see \cite{Luli2018}) is only valid for the one-dimensional wave equation. In this dimension any vector field $F(u)\partial u+G(v)\partial v$  is conformally Killing.}  proposed in \cite{alinhac2009hyperbolic, Luli2018} for the scalar linear {\color{blue} one-dimensional} wave equation $\square \psi = \rho$ ($\tau\in [0, t]$ in $C_{t,\underline{u}}$ and $C_{t,u}$). There exists $C_0>0$ such that 
\begin{equation}
	\begin{aligned}\label{EEnergia}
		& ~{}\int_{S_t} \left[(1+\abs{\underline{u} }^2)^{1+\delta}\abs{\underline{L}\psi}^2+ (1+\abs{u}^2)^{1+\delta}\abs{L\psi}^2\right]dx \\
		& \qquad +\sup_{u\in \mathbb{R}}\int_{C_{t,u}} (1+\abs{\underline{u}}^2)^{1+\delta}\abs{\underline{L}\psi}^2d\tau + \sup_{\underline{u}\in \mathbb{R}}\int_{C_{t,\underline{u}}} (1+\abs{u}^2)^{1+\delta}\abs{L\psi}^2 d\tau \\
		&~{} \qquad \leq C_0\int_{S_0}\left[ (1+\abs{\underline{u}}^2)^{1+\delta}\abs{\underline{L}\psi}^2+(1+\abs{u}^2)^{1+\delta}\abs{L\psi}^2 \right]dx\\
		& ~{} \qquad \quad + C_0 \iint_{D_t} \left[(1+\abs{\underline{u}}^2)^{1+\delta}\abs{\underline{L}\psi} + (1+\abs{u}^2)^{1+\delta}\abs{L\psi}\right]\abs{\rho} d\tau dx.
	\end{aligned}
\end{equation}

\subsection{Global existence for $(\tilde \Lambda,\phi)$} Recall that $\alpha$ was already solved in \eqref{eqn:alpha_DA} and from \eqref{sistema1} $\ln f$ is completely determined if we know $(\Lambda,\phi)$. Now we state a modified version of the main theorem, written in the variables $(\tilde \Lambda,\phi)$, introduced in \eqref{tLambda_talpha}.

\medskip

	For the forthcoming analysis it is it is convenient to introduce a fundamental null form, which is defined as the following bilinear form:
	\begin{equation*}
		Q_0 (\phi,\tilde{\Lambda})= m^{\alpha \beta}\partial_{\alpha} \phi \partial_{\beta} \tilde{\Lambda},
	\end{equation*}
	where $m_{\alpha \beta}$ to denote the standard Minkowski metric on $\mathbb{R}^{1+1}$. Then, using this definition, one can rewrite the first two equations of the system \eqref{sistema1} in terms of null forms as follows: 
	\begin{equation}
		\begin{cases*}
			\square \tilde{\Lambda}=Q_0(\ln \alpha, \tilde{\Lambda})-2\sinh(2\lambda +2\tilde{\Lambda})Q_0(\phi,\phi),\\
			\square \phi =Q_0(\ln \alpha, \phi)+\dfrac{\sinh(2\lambda +2\tilde{\Lambda})}{\sinh^2(\lambda +\tilde{\Lambda})}Q_0(\phi,\tilde{\Lambda}).
		\end{cases*}\label{problema3}
	\end{equation}
	It can also be noticed that the null structure is ``quasi-preserved'' after differentiating with respect to $x$, in the sense that 
	\begin{equation}\label{derivadanullC}
		\partial_xQ_0(\phi,\tilde{\Lambda})=Q_0(\partial_x \phi,\tilde{\Lambda})+Q_0(\phi,\partial_x \tilde{\Lambda}).
	\end{equation}
	Additionally, we have the following relation between  the null form and the Killing vector fields $L$ and $\underline{L}$
	\begin{equation}\label{importante}
		Q_0(\partial_{x}^p \phi, \partial_{x}^q \tilde{\Lambda}) \lesssim \abs{L\partial_x^p \phi}\abs{\underline{L}\partial_x^q \tilde{\Lambda}}+  \abs{\underline{L}\partial_x^p \phi}\abs{L\partial_x^q \tilde{\Lambda}},
	\end{equation}
where the implicit constant is independent of $(\tilde \Lambda, \phi)$.

\medskip

Motivated by {\color{blue} estimate} (\refeq{EEnergia}) and \cite{Alinhacbook,Luli2018,JT}, we define the space-time weighted energy norms, valid for $k=0,1$:
\begin{equation}\label{energies}
	\begin{aligned}
		&	\mathcal{E}_k(t)=\int_{S_t} \left[(1+\left|\underline{u} \right|^2)^{1+\delta}\abs{\underline{L}\partial_x^k\tilde{\Lambda}}^2+  (1+\left|u \right|^2)^{1+\delta}\abs{L\partial_x^k \tilde{\Lambda}}^2\right]dx,\\
		&	\overline{\mathcal{E}}_k(t)=\int_{S_t} \left[(1+\left|\underline{u} \right|^2)^{1+\delta}\abs{\underline{L}\partial_x^k\phi}^2+  (1+\left|u \right|^2)^{1+\delta}\abs{L\partial_x^k \phi}^2\right]dx,\\
		& \mathcal{F}_k(t)= \sup_{u\in \mathbb{R}} \int_{C_{t,u}} (1+\left|\underline{u} \right|^2)^{1+\delta} \left| \underline{L}\partial_x^k\tilde{\Lambda} \right|^2ds+ \sup_{\underline{u}\in \mathbb{R}} \int_{C_{t,\underline{u}}} (1+\left|u \right|^2)^{1+\delta}\abs{L\partial_x^k \tilde{\Lambda}}^2ds,\\
		& \overline{\mathcal{F}}_k(t)= \sup_{u\in \mathbb{R}} \int_{C_{t,u}} (1+\left|\underline{u} \right|^2)^{1+\delta} \left| \underline{L}\partial_x^k\phi \right|^2ds+ \sup_{\underline{u}\in \mathbb{R}} \int_{C_{t,\underline{u}}} (1+\left|u \right|^2)^{1+\delta}\abs{L\partial_x^k \phi}^2ds.
	\end{aligned}
\end{equation}
Then, using \eqref{energies} we define the total energy norms as follows:
\[
\mathcal{E}(t)= \mathcal{E}_0 (t)+\mathcal{E}_1(t).
\] 
Analogously one defines $\mathcal{F}(t)$, $\overline{\mathcal{E}}(t)$, and $\overline{\mathcal{F}}(t).$ 
\begin{rem}
	We note that if $t=0$ then from \eqref{Ctuu} and \eqref{Ctu} one has $\mathcal{F}(0)=\overline{\mathcal{F}}(0)=0$. Also, for $\mathcal{E}(t)$ the initial data determines a constant $C_1$ so that
	\begin{equation}\label{def_C1}
		\mathcal{E}(0)=C_1 \varepsilon^2.
	\end{equation}
This exact bound will be used by the end of the proof of global existence, more specifically in \eqref{cota_casi_final}.	
\end{rem}
We are now ready to state and prove the main result of this section:
\begin{thm}\label{thm_aux}
	Under the assumptions in Theorem \ref{GLOBAL0}, the following are satisfied. Assume that the solution $(\tilde{\Lambda},\phi)$ of the system \eqref{problema3}  exists for $t\in [0,T^{*}]$ satisfying the bounds
	\begin{align}\label{supuesto}
		&	\mathcal{E}(t)+\mathcal{F}(t) \leq 6C_0C_1\varepsilon^2,\\
		&	\overline{\mathcal{E}}(t)+ \overline{\mathcal{F}}(t) \leq 6C_0\overline{C}_1\varepsilon^2\label{supuesto1},
	\end{align}
	and
	\begin{equation}\label{condicionlambda}
		\sup_{t \in [0,T^{*}]}\norm{\tilde{\Lambda}}_{L^{\infty}(\mathbb{R})} \leq \dfrac{\lambda}{2}.
	\end{equation}
	Then for all $t\in [0,T^*]$ there exists a universal constant $\varepsilon_0$ (independent of $T^{*}$) such that the previous estimates are improved for all $\varepsilon \leq \varepsilon_0$.
\end{thm}

The previous result ensures that the solution $(\tilde{\Lambda},\phi)$ constructed via an iterative method is global in time and satisfies the bounds \eqref{supuesto}-\eqref{condicionlambda}. With this result, we can finally conclude the proof of Theorem \ref{GLOBAL0}.

\subsection{Proof of Theorem \ref{thm_aux}} 
For simplicity, we work with the first equation of the system \eqref{problema3}. An analogous study of the equation for the field $\phi$ shows the same outcome, proving that $\phi$ is also globally defined.

\medskip

The proof is based on the bootstrap method; i.e., we will assume that the weighted energies $\mathcal{E}(t)$, $\mathcal{F}(t)$  are bounded by some {\color{blue} particular}
 constant. Then, we will show that the corresponding solution defined in $[0,T^*]$ decays. Since the initial data are small, this allows us to show that the weighted energies are bounded by some better constant. Thus, by continuity, we conclude that the weighted energy cannot grow to infinity in any finite time interval and therefore, using the local existence theorem, the solution exists for all time. 

\medskip

This procedure has been done before in several works, see e.g. \cite{Luli2018, JT}. However, in this work, we have several complications coming from the new wave field $\alpha$, which has to be correctly estimated in order to preserve the wave-like character of the system \eqref{problema3}.

\medskip

Deriving the first equation of (\refeq{problema3}) and using  \eqref{derivadanullC}  we obtain:
\begin{equation}\label{derivadas}
	\square \partial_x \tilde{\Lambda}  = \rho_1 +\rho_2,
\end{equation}
where
\begin{equation}\label{rho_12}
	\begin{cases*}
		\rho_1:= Q_0(\partial_x(\ln \alpha), \tilde{\Lambda})+Q_0(\ln \alpha, \partial_x \tilde{\Lambda}),\\
		\rho_2:= -2\Big[\sinh(2\lambda +2\tilde{\Lambda})\left(Q_0(\partial_x\phi,\phi)+Q_0(\phi,\partial_x\phi)\right)+2\partial_x\tilde{\Lambda}\cosh(2\lambda+2\tilde{\Lambda})Q_0(\phi,\phi)\Big].
	\end{cases*}
\end{equation}
We can see that the null structure is ``quasi-preserved'' after differentiating with respect to $x$. We will use a bootstrap argument as in the (3+1)-dimensional case \cite{klainerman}. Fix $\delta \in (0,1)$. Under the assumptions (\refeq{supuesto})-(\refeq{supuesto1})-(\refeq{condicionlambda}) for all $t\in [0,T^{*}]$, we assume that the solution remains regular, to later show that these bounds are maintained, with a better constant. 

\medskip

Consider $k=0,1$. Using \eqref{EEnergia} on \eqref{problema3}, with $\psi=\partial^k_x \tilde{\Lambda}$ and  \eqref{derivadas}-\eqref{rho_12}. Taking the sum over $k=0,1$, we obtain
\begin{equation}\label{EE2}
	\begin{aligned}
		&\mathcal{E}(t)+\mathcal{F}(t) \leq 2C_0 \mathcal{E}(0)\\
		&  \quad   +2C_0  \iint_{D_t} \left( (1+\left|\underline{u} \right|^2)^{1+\delta} | \underline{L}\tilde{\Lambda}|+ (1+\left|u \right|^2)^{1+\delta}|L \tilde{\Lambda}|\right)  \left| Q_0(\ln \alpha,\tilde{\Lambda})\right|  |Q_0(\phi,\phi)| \\
		&  \quad   +2C_0  \iint_{D_t} \left( (1+\left|\underline{u} \right|^2)^{1+\delta} | \underline{L}\tilde{\Lambda}|+ (1+\left|u \right|^2)^{1+\delta}|L \tilde{\Lambda}|\right) \left| \sinh(2\lambda +2\tilde{\Lambda}) \right|  |Q_0(\phi,\phi)| \\
		&   \quad          +2C_0\iint_{D_t} \left( (1+\left|\underline{u} \right|^2)^{1+\delta} | \underline{L}\partial_x\tilde{\Lambda}|+ (1+\left|u \right|^2)^{1+\delta}| L\partial_x \tilde{\Lambda}|\right) |\rho_1| \\
		& \quad +2C_0 \iint_{D_t} \left( (1+\left|\underline{u} \right|^2)^{1+\delta} | \underline{L}\partial_x\tilde{\Lambda}|+ (1+\left|u \right|^2)^{1+\delta}| L\partial_x \tilde{\Lambda}|\right) |\rho_2|  =: \sum_{j=0}^{5} A_j.
	\end{aligned}
\end{equation}
In the framework of the energy integrals already established, and given the symmetry of the terms, it is sufficient to establish the control of the terms $A_1+A_3$ in \eqref{EE2}, as follows:
\begin{align}\label{defs_Ij}
A_1+A_3= &\overbrace{\iint_{D_t}\varphi(\underline{u})|\underline{L}\tilde{\Lambda}||Q_0(\ln \alpha, \tilde{\Lambda})|}^{I_1}+ \overbrace{\iint_{D_t}\varphi(\underline{u})|\underline{L}\tilde{\Lambda}||Q_0(\partial_x(\ln \alpha), \tilde{\Lambda})+Q_0(\ln \alpha, \partial_x \tilde{\Lambda})|}^{I_2} \nonumber \\
	& +\underbrace{\iint_{D_t}\varphi(u)|L\tilde{\Lambda}||Q_0(\ln \alpha, \tilde{\Lambda})|}_{I_3}+ \underbrace{\iint_{D_t}\varphi(u)|L\tilde{\Lambda}||Q_0(\partial_x(\ln \alpha), \tilde{\Lambda})+Q_0(\ln \alpha, \partial_x \tilde{\Lambda})|}_{I_4}.
\end{align} 
Let us start with the  integral $I_1$ in the term below,  using \eqref{importante} we get:
\begin{equation}\label{int1}
\begin{aligned}
	I_1:= &~{}  \iint_{D_t}\varphi(\underline{u})|\underline{L}\tilde{\Lambda}||Q_0(\ln \alpha, \tilde{\Lambda})| \\
	\lesssim &~{} \iint_{D_t}\varphi(\underline{u})|\underline{L}\tilde{\Lambda}|[|L(\ln \alpha)||\underline{L}\tilde{\Lambda}|+|\underline{L}(\ln \alpha)||L\tilde{\Lambda}
	|] =: I_{1,1} + I_{1,2}.
	\end{aligned}
\end{equation}
We will analyze in detail each part in this integral. For this, we recall the following result {\color{blue} \cite{Luli2018}}:
\begin{lem}[\cite{Luli2018}, Lemma 3.2]\label{lema1}
	Under assumptions  \eqref{supuesto} and \eqref{supuesto1}, there exists a universal constant  $C_2>0$ such that:
	\begin{align*}
		&	|L\tilde{\Lambda}(t,x)|\leq \dfrac{C_2 \varepsilon}{(1+|u |^2)^{1/2+\delta/2}}, &  	|L\phi(t,x)|\leq \dfrac{C_2 \varepsilon}{(1+|u |^2)^{1/2+\delta/2}},\\
		& 	|\underline{L}\tilde{\Lambda}(t,x)|\leq \dfrac{C_2 \varepsilon}{(1+|\underline{u} |^2)^{1/2+\delta/2}}, &  	|\underline{L}\phi(t,x)|\leq \dfrac{C_2 \varepsilon}{(1+|\underline{u} |^2)^{1/2+\delta/2}}.
	\end{align*} 	
\end{lem} 
Consider now Lemma \ref{lema1} and the definition of $\varphi$ in \eqref{varphi}. Also, consider the inequalities for $\alpha$ \eqref{eqn:alpha1}, and \eqref{Llogalpha}. We obtain 
\begin{align*}
	I_{1,1}:=\iint_{D_t}\varphi(\underline{u})|\underline{L}\tilde{\Lambda}|^2|L(\ln \alpha)|\lesssim \int_{\mathbb{R}}\frac{K_1\gamma}{\varphi^{3/4}(u)}\underbrace{\left[\int_{C_{t,u}}\varphi(\underline{u})|\underline{L}\tilde{\Lambda}|^2ds\right]}_{\lesssim \mathcal{F}(t)}du\lesssim K_1 \gamma \varepsilon^2.
\end{align*}
For the second integral in \eqref{int1} consider 
\begin{equation}\label{K}
K:=\max\{K_1,K_2\}.
\end{equation} 
Using again \eqref{eqns:alpha}, \eqref{eqn:alpha0}-\eqref{eqn:alpha1} and \eqref{Llogalpha} and \eqref{lema1} we have
\begin{align*}
	I_{1,2}:=& \iint_{D_t}\varphi(\underline{u})|\underline{L}\tilde{\Lambda}||L\tilde{\Lambda}||\underline{L}(\ln \alpha)|\lesssim \left(\iint_{D_t}\varphi(\underline{u})|\underline{L}\tilde{\Lambda}|^2|L\tilde{\Lambda}| \right)^{1/2}\left(\iint_{D_t}\varphi(\underline{u})\frac{K^2\gamma^2}{\varphi^{3/2}(\underline{u})}|L\tilde{\Lambda}| \right)^{1/2}\\
	&\lesssim {\color{black}\left(\int_{\mathbb{R}} \frac{C_2\varepsilon}{\varphi^{1/2}(u)}\int_{C_{t,u}}\varphi(\underline{u})|\underline{L}\tilde{\Lambda}|^2du\right)^{1/2}\left(\int_{\mathbb{R}}\frac{C_2\varepsilon}{\varphi^{1/2}(u)}\int_{C_{t,u}}\frac{4K^2\gamma^2}{\varphi^{1/2}(\underline{u})}du\right)^{1/2}}\\
	&\lesssim K (C_2 \varepsilon^{3})^{1/2}(C_2\gamma^2\varepsilon)^{1/2}=K C_2\gamma\varepsilon^2.
\end{align*}
For the integral $I_2$ in \eqref{defs_Ij}, we have from \eqref{importante} that
\begin{equation*}
\begin{aligned}
	I_2= &~{} \iint_{D_t}\varphi(\underline{u})|\underline{L}\tilde{\Lambda}||Q_0(\partial_x(\ln \alpha), \tilde{\Lambda})+Q_0(\ln \alpha, \partial_x \tilde{\Lambda})| \\
	\leq &~{}  \iint_{D_t}\varphi(\underline{u})|\underline{L}\tilde{\Lambda}|^2|L(\partial_x(\ln \alpha)| +\iint_{D_t} \varphi(\underline{u})|\underline{L}\tilde{\Lambda}||L\tilde{\Lambda}||\underline{L}\partial_x(\ln \alpha)| \\
	&~{} + \iint_{D_t} \varphi(\underline{u})|\underline{L}\tilde{\Lambda}||\underline{L}\partial_x\tilde{\Lambda}||L(\ln \alpha)| +\iint_{D_t}\varphi(\underline{u})|\underline{L}\tilde{\Lambda}||\underline{L}(\ln\alpha)||L\partial_x(\tilde{\Lambda})| \\
	=: &~{} I_{2,11} +I_{2,12}+I_{2,21}+I_{2,22}.
\end{aligned}	
\end{equation*}
Using \eqref{Llogalpha} and similar computations to the previous ones, we get
\begin{align*}
	I_{2,11}\lesssim K_1\gamma \varepsilon^2.
\end{align*}
Next, using Cauchy-Schwarz,
\begin{align*}
	I_{2,12}\lesssim & \left(\iint_{D_t}\varphi(\underline{u})|\underline{L}\tilde{\Lambda}|^2|L \tilde{\Lambda}| \right)^{1/2}\left(\iint_{D_t}\varphi(\underline{u})|L\tilde{\Lambda}||\underline{L} \partial_x(\ln\alpha)|^2 \right)^{1/2}\\
	\lesssim & \left( \int_{\mathbb{R}}\frac{C_2\varepsilon}{\varphi^{1/2}(u)}\underbrace{\left[\int_{C_{t,u}}\varphi(\underline{u})|\underline{L}\tilde{\Lambda}|^{2}ds \right]}_{\lesssim \mathcal{F}(t)}du\right)^{1/2} \left( \int_{\mathbb{R}}\frac{C_2\varepsilon}{\varphi^{1/2}(u)}\left[\int_{C_{t,u}}\frac{K^2\gamma^2}{\varphi^{1/2}(\underline{u})}ds \right]du\right)^{1/2}\\
	\lesssim &~{} C_2K\gamma \varepsilon^2.
\end{align*}
Using the same analysis as before,
\begin{align*}
	I_{2,21} \lesssim & \left(\iint_{D_t}\varphi(\underline{u})|\underline{L}\tilde{\Lambda}|^2|L(\ln \alpha)| \right)^{1/2}\left(\iint_{D_t}\varphi(\underline{u})|\underline{L}\tilde{\Lambda}|^2|L (\ln\alpha)| \right)^{1/2}\\
	\lesssim & \left(\int_{\mathbb{R}}\frac{K\gamma}{\varphi^{3/4}(u)}\left[\int_{C_{t,u}}\varphi(\underline{u})|\underline{L}\tilde{\Lambda}|^{2}ds \right]du\right)^{1/2}\left(\int_{\mathbb{R}}\frac{K\gamma}{\varphi^{3/4}(u)}\left[\int_{C_{t,u}}\varphi(\underline{u})|\underline{L}\partial_x\tilde{\Lambda}|^{2}ds \right]du\right)^{1/2}\\
	\lesssim&~{}  C_2K\gamma\varepsilon^2.
\end{align*}
The last estimate involves Cauchy-Schwarz to obtain
\begin{align*}
	I_{2,22}\lesssim &~{} \left( \iint_{Dt}\frac{\varphi(\underline{u})}{\varphi(u)}|\underline{L}\tilde{\Lambda}|^2\right)^{1/2} \left( \iint_{D_t}\varphi(\underline{u})|\underline{L}(\ln \alpha)|^2|L\partial_x\tilde{\Lambda}|^2\varphi(u)\right)^{1/2}\\
	\lesssim & ~{}  \varepsilon \left(  \int_{\mathbb{R}}\frac{K^2\gamma^2}{\varphi^{3/2}(\underline{u})}\left[ \int_{C_{t,\underline{u}}}\varphi(u)|L\partial_x \tilde{\Lambda}|\right]d\underline{u}\right)^{1/2}\lesssim K\gamma \varepsilon^2.
\end{align*}
The remaining integrals are analogous and we have for both expressions that :
\begin{align*}
	&I_{3}:=\iint_{D_t}\varphi(u)|L\tilde{\Lambda}||Q_0(\ln \alpha,\tilde{\Lambda})|\lesssim \gamma \varepsilon^2,\\
	& I_4:= \iint_{D_t}\varphi(u)|L\tilde{\Lambda}||Q_0(\partial_x(\ln \alpha), \tilde{\Lambda})+Q_0(\ln \alpha, \partial_x \tilde{\Lambda})|\lesssim \gamma \varepsilon^2.
\end{align*}
For the other term, which corresponds to $\rho_2$ in \eqref{derivadas}, the analysis is the same as described in our recently completed work  \cite{JT}. See this reference for full details.

\medskip

Finally, from  the energy estimate (\refeq{EEnergia}), we can arrange all the previous estimates together, and for universal constants $C_4, C_5,K$ with $K_1,K_2 \leq K,$ (see \eqref{K}), one has for all $t\in [0, T^{*}]$:
\begin{equation}\label{cota_casi_final}
	\mathcal{E}(t)+\mathcal{F}(t) \leq  (2C_0C_1+K\gamma)\varepsilon^2 + C_4 \varepsilon^3 + C_5 \varepsilon^4,
\end{equation}
where $C_1$ is given in \eqref{def_C1}. Now, if we take $\varepsilon_0 $ such that
\begin{equation}\label{condicionfinal}
	\varepsilon_0 \leq \dfrac{C_0 C_1}{C_4}, \quad  \varepsilon_0^2 \leq \dfrac{C_0 C_1}{C_5},
\end{equation}
and $\gamma$ such that 
\[    K\gamma < \frac{C_0 C_1}{2},
\]
we can see that for all $0 < \varepsilon \leq \varepsilon_0$ and for all $t\in [0, T]$, we have
\begin{equation*}
	\mathcal{E}(t)+\mathcal{F}(t) \leq \frac{9}{2}C_0C_1 \varepsilon^2.
\end{equation*}
By taking a suitable $\gamma$ and $\varepsilon_0$, we have the desired control. This improves the constant in (\refeq{supuesto}).

\medskip

To improve condition (\refeq{condicionlambda}),  using the Fundamental Theorem of Calculus, \eqref{L Lbar} and Lemma \refeq{lema1}, one can write $\tilde{\Lambda}(t,x)$, $t\geq 0$, in the following form:
\begin{equation}\label{tLambda}
\begin{aligned}
	\abs{\tilde{\Lambda}(t,x)}&\leq \varepsilon \abs{\tilde{\Lambda}_0(x)} +\int_{0}^t \abs{\partial_{\tau}\tilde{\Lambda}(\tau,x)}d\tau \\
	&\leq \varepsilon M_1 +\dfrac{1}{2}\int_0^t \abs{L\tilde{\Lambda}+\underline{L}\tilde{\Lambda}}d\tau\\
	& \leq \varepsilon M_1+ \dfrac{1}{2}\int_0^t \left(\dfrac{C_2 \varepsilon}{\varphi(u)^{1/2}} +\dfrac{C_2 \varepsilon}{\varphi(\underline{u})^{1/2}}\right)d\tau\\
	& \leq \varepsilon M_1 + \varepsilon C_2M_2\leq M \varepsilon,
\end{aligned}
\end{equation}
for some universal constant $M.$ Next, we take $\varepsilon_0 >0$ that satisfies the condition $(\refeq{condicionfinal})$ and such that
\begin{equation}\label{condfinal2}
	M\varepsilon_0 < \dfrac{\lambda}{4},
\end{equation}
taking $\sup$ over $t\in [0,T^{*}]$, we conclude that for all $0 < \varepsilon \leq \varepsilon_0$ we have improved via \eqref{condfinal2} the key estimate (\refeq{condicionlambda}). 

\medskip

The above estimates prove that the solution $\tilde{\Lambda}$ is global. A similar argument, as established before, shows that $\phi$ is also globally defined. This ends the existence proof in Theorem \ref{thm_aux}.

\subsection{End of proof of Theorem \ref{GLOBAL0}} Since $\alpha$, $\tilde \Lambda$ and $\phi$ have been completely determined in previous steps, we only need to determine the behavior of the {\color{blue} function $f$} in \eqref{sistema1}. Note that $\alpha$ satisfies \eqref{eqn:alpha_DA} and in the system \eqref{sistema1} we have that $\ln f$ satisfies the nonhomogeneous wave equation with {\color{blue}initial} conditions \eqref{Ciniciales}, then, we can use  d'Alembert's solution {\color{blue}to describe} the function $f$, but {\color{blue} before}, let us analyze the following result:

\begin{lem} Let $G$ be defined as in \eqref{sistema2}. Under the hypotheses of Theorem \ref{GLOBAL0}, and {\color{blue}as a consequence} of Theorem \ref{thm_aux}, the following is satisfied:
	\begin{itemize}
		\item For each $t \in \mathbb{R}$, $G(t,\cdot) \in (L^1\cap L^\infty)(\mathbb R)$;
		\item There exists $C>0$ such that $\sup_{t\geq 0} \|G(t)\|_{L^1\cap L^\infty} \leq C.$
       \end{itemize}
\end{lem}

\begin{proof}
	Since $G$ is given by \eqref{sistema2}, one has
	\[
	\begin{aligned}
		& ~{} G := - \left( \partial_{t}^2(\ln \alpha)-\partial_{x}^2(\ln \alpha) \right)- \dfrac{1}{2\alpha^2} ((\partial_t\alpha)^2-(\partial_x\alpha)^2) \\
		&~{} \qquad  - \dfrac{1}{2}((\partial_t\Lambda)^2-(\partial_x\Lambda)^2)- 2\sinh^2 \Lambda((\partial_t\phi)^2-(\partial_x\phi)^2).
	\end{aligned}
	\]
	From \eqref{eqns:alpha} and \eqref{eqn:alpha0}, we can simplify 
	\[
	\begin{aligned}
		&~{}G=	\dfrac{1}{2\alpha^2}(\tilde{\alpha}'_0(2u)+\alpha_1(2u))(\alpha_1(-2\underline{u})-\tilde{\alpha}'_0(-2\underline{u}))\\
		&~{} \qquad  - \dfrac{1}{2}((\partial_t\Lambda)^2-(\partial_x\Lambda)^2)- 2\sinh^2 \Lambda((\partial_t\phi)^2-(\partial_x\phi)^2) =: G_1 + G_2.
	\end{aligned}
	\]
	It can be seen that the regularity of the term $G$ depends on the initial conditions for the function $\alpha$, and on the functions $\Lambda, \phi$. The hypotheses in Theorem \ref{GLOBAL0} ensure that, for all $t\in \mathbb{R}$,  
	\[
	G_1 = \dfrac{1}{2\alpha^2}(\tilde{\alpha}'_0(2u)+\alpha_1(2u))(\alpha_1(-2\underline{u})-\tilde{\alpha}'_0(-2\underline{u})) \in S(\mathbb R).
	\]
Moreover, $\sup_{t\in\mathbb R}\|G_1(t)\|_{L^1 \cap L^\infty} \leq C.$ On the other hand, $G_2$ satisfies from \eqref{h1_0}
\[
\left| (\partial_t\Lambda)^2-(\partial_x\Lambda)^2- 2\sinh^2 \Lambda((\partial_t\phi)^2-(\partial_x\phi)^2 )\right| (t,x) \leq h_1(t,x).
\]
Now we use the following result to conclude:
\begin{lem}\label{EdCo}
\[
\left| h_1(t,x)\right| \lesssim \frac{\varepsilon^2}{\varphi(u)} +\frac{\varepsilon^2}{\varphi(\underline{u})}.
\]
\end{lem}

Assuming this result, we finally get $G(t,\cdot)\in L^1(\mathbb{R}) \cap C(\mathbb{R})$ with uniform bounds in time. 
\end{proof}

\begin{proof}[Proof of Lemma \ref{EdCo}]
First of all, we have (see \eqref{L Lbar}, Lemma \eqref{lema1} and \eqref{varphi})
\[
|\partial_t \Lambda |\lesssim \abs{L\tilde{\Lambda}}+\abs{\underline{L}\tilde{\Lambda}}  \leq  \dfrac{\varepsilon}{\varphi(u)^{1/2}} +\dfrac{ \varepsilon}{\varphi(\underline{u})^{1/2}} .
\]
Similarly, 
\[
|\partial_x\Lambda | +|\partial_x\phi | +|\partial_x\phi | \lesssim  \dfrac{\varepsilon}{\varphi(u)^{1/2}} +\dfrac{ \varepsilon}{\varphi(\underline{u})^{1/2}} .
\]
Finally, thanks to \eqref{tLambda},
\[
\sinh^2(\Lambda) \lesssim \sinh^2(\lambda).
\]
Gathering these results, we conclude. 
\end{proof}

The previous result allows us to describe the function $f$ using d'Alembert formula for the nonhomogeneous linear wave and \eqref{Ciniciales}. Consider the initial data problem  for $v(t,x):=\ln f(t,x)$ given by
\begin{equation}
\begin{cases}
\partial_t^2 v-\partial_x^2 v = G(t,x)\\
v(0,x)= \ln (f(0,x))=\ln (f_0+c_1)\\
\partial_t v(t,x)|_{\{t=0\}}= \frac{f_1}{c_1+f_0}.
\end{cases}
\end{equation}
We get
\begin{equation}\label{sol_ln_f}
	\begin{aligned}
		v(t,x):=&\frac{1}{2}\left[\ln (c_1+f_0(x+t))+\ln(c_1+f_0(x-t))\right]\\
		& +\frac{1}{2}\int_{x-t}^{x+t}\frac{f_1(s)ds}{c_1+f_0(s)}  +\frac{1}{2}\int_{0}^t \left[ \int_{x+s-t}^{x+t-s} G(s,y)dy\right]ds =: v_1+v_2+v_3 .
	\end{aligned}
\end{equation}
It is clear that $v_1$ and $v_2$ are globally defined, bounded in time and space members. On the other hand, thanks to \eqref{EdCo},
\[
\begin{aligned}
\left| \int_{x+s-t}^{x+t-s} G(s,y)dy \right| \lesssim &~{} \varepsilon^2  \int_{x+s-t}^{x+t-s} \left( \frac{1}{\varphi(s+y)} +\frac{1}{\varphi(s-y) }   \right) dy \\
\lesssim &~{}\varepsilon^2  \int_{x+2s-t}^{x+t} \frac{dy}{\varphi(y)}  +\varepsilon^2  \int_{x-t}^{x+t-2s} \frac{ dy}{\varphi(y) } \lesssim \frac{\varepsilon^2(t-s)}{\varphi(x+t)}+ \frac{\varepsilon^2(t-s)}{\varphi(x-t)}.
\end{aligned}
\]
Consequently,
\[
|v_3| \lesssim \int_0^t  \left( \frac{\varepsilon^2(t-s)}{\varphi(x+t)}+ \frac{\varepsilon^2(t-s)}{\varphi(x-t)} \right)ds  \lesssim  t^2\left( \frac{\varepsilon^2}{\varphi(x+t)}+ \frac{\varepsilon^2}{\varphi(x-t)} \right). 
\]
 Now we conclude the proof of the theorem. From \eqref{sol_ln_f} the function $f$ is given by 
\begin{align*}
	f(t,x)=\rho(t,x) \exp\left(\frac{1}{2}\int_{x-t}^{x+t}\frac{f_1(s)ds}{c_1+f_0(s)} +\frac{1}{2}\int_{0}^t \left[ \int_{x+s-t}^{x+t-s} G(s,y)dy\right]ds \right),
\end{align*}
with 
\[
\rho(t,x)= \sqrt{(c_1+f_0(x+t))(c_1+f_0(x-t))}.
\]
Notice that $f$ is strictly positive everywhere in time and space. Given the initial conditions imposed on the function $f$, the integrals are well-defined. Additionally, the function $f$ is positive, {\color{blue} consistently with the BZ proposal}. 

\section{Energy-momentum formulation}\label{Sect:4}

The aim of this section is first to introduce a correct definition of energy and momentum densities for one type of solutions of the Einstein equations in vacuum, and then to give a proper description of the decay of these solutions in the framework of the global existence theory presented in the previous section.

\medskip

The notion of energy and the law of conservation of energy play a key role in all mathematical-physical theories. The definition of energy in relativity is a complex matter, and this problem has been given a lot of attention in the literature \cite{wald2010general,wald2000}. For this and other reasons, it is very interesting to study and define what could be considered a good definition of ``energy''. However, the most likely candidate for the energy density of the gravitational field in general relativity would be a quadratic expression in the first derivatives of the components of the metric \cite{wald2010general}, or as in this case, in terms of the fields defining the components of the metric. For the particular case of spacetimes admitting two commutative Killing vectors, the energy formulation is constrained by the function $\alpha(t,x)$, which we recall, in this setting, is a positive solution of the one-dimensional wave equation. 

\medskip

What we must keep in mind is that these spacetimes can be used to describe both cylindrical gravitational waves and inhomogeneous cosmological models in vacuum, but the former are less suited to study decay properties, for the reasons exposed below. Roughly speaking, for gravitational cylindrical wave solutions, the gradient of $\alpha(t,x)$ must be spacelike, while for the description of cosmological models, it must be timelike. 

\medskip

In this section, we propose an adequate description of the energy and momentum densities, according to the type of spacetime being analyzed, i.e., subject to the sign of the gradient of the function $\alpha$. 

\subsection{Energy-Momentum formalism}  We begin by proposing an initial definition for energy and momentum densities of the system \eqref{sistema1}. In the spirit of the definition proposed by Hadad in \cite[p.73]{yaronhadad_2013}, we will expose this new description for these densities in the suitable terms of the field $\Lambda, \phi, \alpha$, and study whether or not it is a conserved quantity and to find local conservation laws.
\medskip

Recall \eqref{EnergyF}. In terms of the fields $\Lambda, \phi$ and the function $\alpha(t,x)$,  let us the introduce the following densities:
\begin{equation}\label{eq:energiaN}
	\begin{aligned}
		e(t,x):=&~{}\kappa \partial_t\alpha \left[\dfrac{(\partial_x \alpha)^2+(\partial_t \alpha)^2}{\alpha^2}+ 4\sinh^2(\Lambda)\Big((\partial_t \phi)^2+(\partial_x\phi)^2\Big)+(\partial_t\Lambda)^2+(\partial_x\Lambda)^2 \right]\\
		&~{}\quad  -2\kappa \partial_x\alpha\left(\dfrac{\partial_x\alpha \partial_t \alpha}{\alpha^2}+\partial_x\Lambda \partial_t\Lambda +4\partial_x \phi \partial_t \phi     \sinh^2(\Lambda) \right),
\end{aligned}
\end{equation}
where 
\begin{equation}\label{gamma}
	\kappa(t,x)=\dfrac{\alpha }{(\partial_x\alpha)^2-(\partial_t\alpha)^2},
\end{equation}
and
\begin{equation}
\begin{aligned}\label{eq:momentumN}
p(t,x):=&~{}\kappa \partial_x\alpha \left[\dfrac{(\partial_x \alpha)^2+(\partial_t \alpha)^2}{\alpha^2}+ 4\sinh^2(\Lambda)\Big((\partial_t \phi)^2+(\partial_x\phi)^2\Big)+(\partial_t\Lambda)^2+(\partial_x\Lambda)^2 \right]\\
		&~{}\quad  -2\kappa \partial_t\alpha\left(\dfrac{\partial_x\alpha \partial_t \alpha}{\alpha^2}+\partial_x\Lambda \partial_t\Lambda +4\partial_x \phi \partial_t \phi     \sinh^2(\Lambda) \right).
	\end{aligned}
\end{equation}
{\color{black} It should be noted that, in providing these densities, certain constraints, regarding the region in which $(\partial_x\alpha)^2-(\partial_t\alpha)^2$ is null, must be considered. These considerations will be studied in more detail in the following section}. Now, in order to have a suitable definition of these densities, we propose the following redefinition:
\begin{equation}\label{eq:energiaN_new}
	\begin{aligned}
		\tilde e=\tilde e[\Lambda,\phi,\alpha]:=&~{}\kappa \partial_t\alpha \left[ (\partial_t\Lambda)^2+(\partial_x\Lambda)^2+4\sinh^2(\Lambda)\Big((\partial_t \phi)^2+(\partial_x\phi)^2\Big) \right]\\
		&~{}\quad  -2\kappa \partial_x\alpha\left(\partial_x\Lambda \partial_t\Lambda +4\partial_x \phi \partial_t \phi     \sinh^2(\Lambda) \right),
\end{aligned}
\end{equation}
and
\begin{equation}
\begin{aligned}\label{eq:momentumN_new}
\tilde p= \tilde p[\Lambda,\phi,\alpha]:=&~{}\kappa \partial_x\alpha \left[ (\partial_t\Lambda)^2+(\partial_x\Lambda)^2 +4\sinh^2(\Lambda)\Big((\partial_t \phi)^2+(\partial_x\phi)^2\Big) \right]\\
		&~{}\quad  -2\kappa \partial_t\alpha\left(\partial_x\Lambda \partial_t\Lambda +4\partial_x \phi \partial_t \phi     \sinh^2(\Lambda) \right).	
	\end{aligned}
\end{equation}
For the densities $e,  p$, we can state the following identities
\begin{lem}\label{eqContinuidad0}
	Let $(\Lambda, \phi,\alpha)$ be a solution to \eqref{sistema1}. Let $ e(t,x)$ and $p(t,x)$ be as introduced in \eqref{eq:energiaN}-\eqref{eq:momentumN}. Assume that $(t,x)$ lies in an open region of spacetime such that $(\partial_x \alpha)^2-( \partial_t \alpha)^2\neq 0$. Then one has
	\begin{equation}\label{eqContinuidad}
		\begin{aligned}
		&\partial_t p(t,x)+\partial_x e(t,x)=0,\\
			&{\color{black}\partial_t e(t,x)+\partial_x p(t,x)=4\sinh^2(\Lambda)\left(\phi_t^2-\phi_x^2\right)+\Lambda_t^2-\Lambda_x^2
 +\partial_x\left(\dfrac{\alpha_x}{\alpha}\right) -\partial_t\left(\dfrac{\alpha_t}{\alpha}\right)
}.
		\end{aligned}
	\end{equation}
\end{lem} 

Equations \eqref{eqContinuidad} are a modified version of the continuity equations for the energy and momentum densities. A perfectly behaved relation was found in \cite{JT} in the case of the integrable Principal Chiral model. The situation here is more subtle, and there is no sign of a perfectly behaved continuity equation, {\color{blue}mainly} because of the functions $\alpha$ and $f$.

Part of the proof of the first equation is essentially contained in Hadad \cite{yaronhadad_2013}, but the technical details, as well as the proof of the second equation, are included in Appendix \ref{App:B}. As a corollary, we also have the following identities for the redefined densities $\tilde e(t,x), $ and $\tilde p(t,x)$:

\begin{cor}\label{eqContinuidad_new} Let $\tilde e(t,x)$ and $\tilde p(t,x)$ be as introduced in \eqref{eq:energiaN_new}-\eqref{eq:momentumN_new}. Under the assumptions of Lemma \ref{eqContinuidad0}, one has
	\begin{equation*}
		\begin{aligned}
		&\partial_t \tilde p(t,x)+\partial_x \tilde e(t,x)=0,\\
			&{\color{black}\partial_t \tilde e(t,x)+\partial_x \tilde p(t,x)=4\sinh^2(\Lambda)\left(\phi_t^2-\phi_x^2\right)+\Lambda_t^2-\Lambda_x^2.
}
		\end{aligned}
	\end{equation*}
\end{cor} 

\begin{proof}
The proof follows immediately from the definition and description of the densities obtained in the Lemma \refeq{eqContinuidad}.
\end{proof}

{\color{black}Note the symmetry in the terms defining the densities, however, the derivatives of the function $\alpha(t,x)$ make a significant change in the nature of these densities, (as compared to the Chiral field equation case, where $\alpha(t,x)$ was considered as a constant, see \cite{JT}).  This implies a deeper analysis regarding the correct formulation of energy densities. {\color{black}As mentioned in the introduction, the local behavior of the spacetime is defined by the nature of the function $\alpha(t,x)$. 

\medskip

This function may have a gradient spacelike in all the spacetime (corresponds to spacetimes with cylindrical symmetry), globally null (corresponds to the plane-symmetric waves), or timelike (cosmological type-solutions), see \cite{carmeli1984einstein, belinski2001gravitational, Belinsky, Einstein1937, Jerry2009} for more details. The following sections propose appropriate definitions of the energy and momentum densities associated with each type of solution, i.e., depending on the nature of the gradient of the alpha function.
 
 \subsection{Cosmological-type solutions} \label{CTS} 
 As mentioned before, spacetimes in the Belinski-Zakharov setting can be used to represent inhomogeneous vacuum cosmological models. In these, the universe is assumed to contain gravitational waves propagating in opposite spatial directions, see \cite{belinski2001gravitational, Belinsky}. To describe this class of models, it is appropriate to take the function $\alpha(t,x)$  timelike, i.e., with negative gradient norm.  Let us start with some preliminary definitions and results.

\begin{defn}[Timelike condition]\label{def:TL} Given the function $\alpha(t,x)$, we will say that $\alpha(t,x)$ is timelike, if its gradient satisfies 
 \begin{equation}\label{condicion:timelike}
(\partial_x \alpha)^2-(\partial_t \alpha)^2 < 0, \quad \forall (t,x)\in\mathbb R^2.
\end{equation}
In this case, we will say that our model is of cosmological type.
 \end{defn}
 
Definition \ref{def:TL} is taken from \cite[p. 965]{carmeli1984einstein}. It is relevant to remark that, as expressed in \cite{carmeli1984einstein}, other cosmological type models are of interest, such as Gowdy models. For more details, the reader can consult the aforementioned work and references therein.
 
 Using the same notation as in \eqref{eqn:alpha_DA} for the initial conditions for $\alpha$, as a solution to the wave equation, when $\alpha(t,x)$ is timelike everywhere, we have the following result:
 
 \begin{lem}\label{condicionSigno0}
Assume that $\alpha(t,x)$ is timelike in the whole spacetime. Then
\begin{enumerate}
\item[(i)] one has
\begin{equation}\label{alfax_alfat}
|\partial_x \alpha|< |\partial_t \alpha|.
\end{equation}
\item[(ii)] if additionally, $(t,x)$ is such that 
\begin{equation}\label{dt_alpha_p}
\partial_t \alpha (t,x) >0, 
\end{equation} 
then,  $|\partial_x \alpha| < \partial_t \alpha$ if and only if  the initial data of the function $\alpha$ satisfy
\begin{equation}\label{alfa0 alfa1}
	|\tilde \alpha'_0(\cdot)| < \alpha_1(\cdot).
\end{equation} 
\item[(iii)] if additionally 
\begin{equation}\label{alfa pos}
\alpha(t,x) >0 \quad \forall (t,x)\in \mathbb R^2,
\end{equation} 
then the parameter $\kappa$ defined in \eqref{gamma} is well-defined and it is negative, and 
\begin{equation}\label{mk dt alpha}
-\kappa \partial_t\alpha >0.
\end{equation}
\end{enumerate}
 \end{lem}
 \begin{proof}
The proof is obtained from a straightforward calculation and the use of \eqref{dt_alpha_p}.
\end{proof}
 \begin{rem}
 Notice that condition \eqref{condicion:timelike} is ensured if the initial data for $\alpha$ satisfies  \eqref{alfa0 alfa1}. In addition, this condition allows us to propose an $\alpha$ function that is consistent with the hypotheses of the Theorem \refeq{GLOBAL0}.
 \end{rem}
  
 Now, in order to define a positive energy density and being possible to set control of this density over the momentum density, we propose:
\begin{defn}
For cosmological-type solutions, the energy and momentum densities will be defined as
\begin{equation}\label{e_timelike_0}
 \hat e(t,x)=-\tilde e(t,x),
\end{equation}
and 
\begin{equation}\label{m_timelike} 
\hat p(t,x)= \tilde p(t,x).
\end{equation}
\end{defn}	

\subsection{Proof of Theorem \ref{MT20}}

Theorem \ref{MT20} will be a consequence of the following lemma.

\begin{lem}\label{positivity}
Under \eqref{alfax_alfat}, \eqref{dt_alpha_p} and \eqref{alfa pos}, the energy density defined in \eqref{e_timelike_0} is nonnegative. Moreover, one has an improved estimate
\[
 \hat e\geq  |\kappa|( | \partial_t\alpha|- | \partial_x\alpha|  ) \left[ (\partial_t\Lambda)^2+(\partial_x\Lambda)^2+4\sinh^2(\Lambda)\Big((\partial_t \phi)^2+(\partial_x\phi)^2\Big) \right] .
\]
\end{lem}

\begin{proof}[Proof of Lemma \ref{positivity}]
We compute: from \eqref{eq:energiaN_new} and Lemma \ref{condicionSigno0} (ii),
\begin{equation}\label{nonnegativity}
	\begin{aligned}
	 \hat e= -\tilde e=&~{}- \kappa \partial_t\alpha \left[ (\partial_t\Lambda)^2+(\partial_x\Lambda)^2+4\sinh^2(\Lambda)\Big((\partial_t \phi)^2+(\partial_x\phi)^2\Big) \right]\\
		&~{}\quad  + 2\kappa \partial_x\alpha\left(\partial_x\Lambda \partial_t\Lambda +4\partial_x \phi \partial_t \phi     \sinh^2(\Lambda) \right)\\
	\geq &~{} 	|\kappa \partial_t\alpha|\left[ (\partial_t\Lambda)^2+(\partial_x\Lambda)^2+4\sinh^2(\Lambda)\Big((\partial_t \phi)^2+(\partial_x\phi)^2\Big) \right] \\
	&~{}\quad  - 2|\kappa \partial_x\alpha| \left(|\partial_x\Lambda ||\partial_t\Lambda| +4|\partial_x \phi ||\partial_t \phi |    \sinh^2(\Lambda) \right)\\
	= &~{} 	|\kappa|( | \partial_t\alpha|- | \partial_x\alpha|  ) \left[ (\partial_t\Lambda)^2+(\partial_x\Lambda)^2+4\sinh^2(\Lambda)\Big((\partial_t \phi)^2+(\partial_x\phi)^2\Big) \right]  \\
	&~{}  + |\kappa \partial_x\alpha| \Big( (\partial_t\Lambda)^2+(\partial_x\Lambda)^2- 2 |\partial_x\Lambda ||\partial_t\Lambda|  \Big)\\
	&~{}  +4 |\kappa| |\partial_x\alpha| \sinh^2(\Lambda)\Big((\partial_t \phi)^2+(\partial_x\phi)^2 - 2 |\partial_x \phi ||\partial_t \phi |   \Big)\\
	 \geq &~{} |\kappa|( | \partial_t\alpha|- | \partial_x\alpha|  ) \left[ (\partial_t\Lambda)^2+(\partial_x\Lambda)^2+4\sinh^2(\Lambda)\Big((\partial_t \phi)^2+(\partial_x\phi)^2\Big) \right] \geq  0.
\end{aligned}
\end{equation}
The proof is complete.
\end{proof}

Furthermore, under this same hypothesis, we can establish an appropriate control of the energy density on the momentum density. Recall that the condition that $\partial_t \alpha>0$, necessarily implies that the function $\alpha_1 (s)> 0,  \forall s \in \mathbb{R}$, which is in correspondence with the setting proposed for the function $\alpha(t,x)$ in the previous global existence theory.
We can obtain the following result:
\begin{lem}\label{LEM} Under \eqref{alfax_alfat}, \eqref{dt_alpha_p} and \eqref{alfa pos},
	\begin{equation}\label{ME_new}
	|\hat p (t,x) | \leq \hat e(t,x) .
	\end{equation}
\end{lem} 
\begin{proof}
To simplify the notation, let us define:
\begin{equation}\label{h1}
 h_1(t,x)= (\partial_t \Lambda)^2+ (\partial_x \Lambda)^2+4\sinh^2(\Lambda)((\partial_x \phi)^2+(\partial_t \phi)^2) \geq 0,
\end{equation}
and
\begin{equation}\label{h2}
h_2(t,x)= \partial_t\Lambda \partial_x\Lambda +4\sinh^2(\Lambda)\partial_t\phi \partial_x\phi,
\end{equation}
then, the energy density and the momentum density can be written as 
\[
\hat e= -\kappa (\partial_t \alpha h_1 -2 \partial_x \alpha h_2), \quad \quad \hat p= \kappa (\partial_x\alpha h_1-2  \partial_t \alpha h_2).
\]
Recall that, $\hat e\geq |\kappa|( | \partial_t\alpha|- | \partial_x\alpha|  ) h_1\geq  0$ thanks to \eqref{nonnegativity}.
Now, let us prove \eqref{ME}. Using the Cauchy inequality and the condition \eqref{condicion:timelike}, one has $2|h_2| \leq h_1$. Therefore, using that $\kappa<0$, $ \kappa h_1 \leq 2\kappa h_2 $. Since $\partial_t \alpha+\partial_x \alpha>0$, one has

\[
\kappa h_1(\partial_t \alpha+\partial_x \alpha) \leq  2\kappa h_2(\partial_t \alpha+\partial_x \alpha).
\]
Consequently,
\[
 \kappa \partial_x \alpha h_1-2\kappa \partial_t \alpha h_2  <-\kappa \partial_t \alpha h_1+2\kappa \partial_x \alpha h_2,
\]
which proves that $\hat p \leq \hat e$.

For the other direction we have $\partial_t \alpha-\partial_x \alpha>0$, $2\kappa h_2 \leq - \kappa h_1$, so that
\begin{equation*}
\begin{aligned}
-\kappa h_1(\partial_t \alpha-\partial_x \alpha) &\geq 2\kappa h_2(\partial_t \alpha-\partial_x \alpha), \\
  \kappa \partial_x \alpha h_1-2\kappa \partial_t \alpha h_2 & \geq \kappa \partial_t \alpha h_1-2\kappa \partial_x \alpha h_2,\\
\hat p & \geq -\hat e. \\
\end{aligned}
\end{equation*}
Therefore, we obtain control of energy density over momentum density
\[ |\hat p(t,x)| \leq \hat e(t,x).
\]
The proof is complete.
\end{proof}
With the previous definitions  of hat-densities  \eqref{e_timelike_0} and \eqref{m_timelike} and the identities obtained in Corollary \refeq{eqContinuidad_new}, one has the following consequences (modified continuity equations): 
\begin{cor}\label{EC:timelike} Let $(\Lambda,\phi, \alpha)$ solutions of the system \eqref{sistema1}. Under the assumptions of Lemma \ref{eqContinuidad0}, one has\begin{equation}\label{CEtimelike}
		\begin{aligned}
		&\partial_t \hat p(t,x)-\partial_x \hat e(t,x)=0,\\
			&\partial_t \hat e(t,x)-\partial_x \hat p(t,x)=4\sinh^2(\Lambda)\left(\phi_x^2-\phi_t^2\right)+\Lambda_x^2-\Lambda_t^2 .
		\end{aligned}
	\end{equation}
\end{cor}

Lemma \refeq{LEM} and Corollary \refeq{EC:timelike} will become very important, in the sense that the control that new energy density has over the momentum density allows us to propose a virial estimate, and analyze the long-time behavior of the cosmological type solution, as we will be discussing in the subsequent sections. We now discuss the energy formulation for the case where the gradient of the function $\alpha$ is spacelike.


\subsection{Cylindrical Gravitational waves} Let $u^{(0)}=u^{(0)}(t,r)$ be a solution to the cylindrical wave equation in 2D:
\[
\partial_t^2 u^{(0)} = \frac1r \partial_r \left(r \partial_r u^{(0)}\right ), \quad (t,r)\in \mathbb R_t \times (0,\infty).
\]
As usual, $\alpha$ satisfies the 1D wave equation in $(t,r)$. Let us introduce the following line element of a cylindrically symmetric spacetime as follows:
\begin{equation}\label{stinterval0}
ds^2=f^{(0)}(-dt^2+dr^2)+e^{-u^{(0)}}(\alpha d\phi)^2 + e^{u^{(0)}}dz^2,
\end{equation}
with $x^a=\{ \phi, z\}$ and $x^i=\{ t,r\}$ and $r>0$. This line element belongs to the class of solutions considered in the Belinski-Zakharov spacetime setting, where 
\[
g = \left[ \begin{matrix} \alpha^2 e^{-u^{(0)}} & 0 \\  0 & e^{u^{(0)}} \end{matrix}\right] .
\]  
A particular case of the metric \eqref{stinterval0} is the one given by the Einstein-Rosen model, where $\alpha\equiv r$. See \eqref{stinterval} and \eqref{stinterval1} for more details.

\medskip
As mentioned before, the local behavior of the considered spacetime is defined by the gradient of the function $\alpha$.  In the case that this gradient is spacelike, it actually corresponds to cylindrical spacetimes. Notice that metric \eqref{stinterval0} is a particular case of \eqref{intervalo} in cylindrical coordinates,  where the fields described by the geometric representation \eqref{diag1}, are given as follows:  the field $\phi$ is a  constant, and the field $\Lambda =u^{(0)}$. In this section, we will consider precisely this general setting.  Consider the system \eqref{sistema1} with $\alpha$ as a positive solution of the one-dimensional wave equation, and satisfying the so-called space-like condition, which will be described below. Thus, we capture the essential condition describing the Einstein-Rosen gravitational wave metric \cite{carmeli1984einstein}.

As in the previous subsection, we introduce some preliminary definitions and results.

 \begin{defn}[Spacelike Condition] We say that $\alpha(t,x)$ is spacelike if its gradient satisfies 
 \begin{equation}\label{condicion:spacelike}
(\partial_r \alpha)^2-(\partial_t \alpha)^2 > 0, \quad \forall (t,r). 
\end{equation}
 \end{defn}

The spacelike condition \eqref{condicion:spacelike} contrasts with the timelike one in \eqref{condicion:timelike} not only by the obvious 
reason (opposite signs), but also because it will allow not decaying solutions to the problem. In this sense, one can guess that no general virial theorem is present in this situation, unless we assume additional hypotheses on $u^{(0)}$ and $\alpha$.

\medskip 
 
 Coming back to \eqref{stinterval0}, and using the same notation for the initial conditions for $\alpha$ as in \eqref{eqn:alpha_DA},  with $\alpha(t,x)$ spacelike everywhere, one has the following result:
 \begin{lem}\label{condicionSigno}
  If the function $\alpha$ is spacelike, 
\[
|\partial_t \alpha|< |\partial_r \alpha|,
\]
and the following are satisfied:
\begin{enumerate}
\item[(i)] if $(t,r)$ is such that 
\begin{equation}\label{dr_alpha_p}
\partial_r \alpha (t,r) >0,
\end{equation}
then, $|\partial_t \alpha| < \partial_r \alpha$ if and only if   the initial data of the function $\alpha$ satisfy
\begin{equation*}
	 |\alpha_1(\cdot)|< \tilde \alpha'_0(\cdot)
\end{equation*} 
\item[(ii)] the parameter $\kappa$ defined in \eqref{gamma} (with $x$ replaced by the variable $r$) is well-defined and positive, and 
\begin{equation}\label{mk dr alpha}
\kappa \partial_r\alpha >0.
\end{equation}
\end{enumerate}
 \end{lem}
 \begin{proof}[Proof of Lemma \ref{condicionSigno}] 
The proof is obtained from a straightforward calculation as in Lemma \ref{condicionSigno0}.
\end{proof}

Comparing with \eqref{dt_alpha_p} and \eqref{mk dt alpha}, one can see that \eqref{dr_alpha_p} and \eqref{mk dr alpha} are ``dual'' to the former ones. Although one can think that these properties are not harmful, it turns the case that this is exactly the case: these signs are bad for decay purposes by natural reasons: spacelike dynamics tends to be unphysical in reality.

\medskip

To ensure that we have an appropriate energy and to be possible to set a control of this density over the momentum one, we define
\begin{defn}
For cylindrical-type solutions, the energy and momentum densities are defined as
\begin{equation}\label{e_spacelike}
 \hat e(t,r) =\tilde p(t,r),
\end{equation}
and 
\begin{equation}\label{m_spacelike} 
\hat p(t,r)= \tilde e(t,r).
\end{equation}
\end{defn}	
Notice that in this case, when the gradient of the function $\alpha(t,r)$ is spacelike, the parameter $\kappa$ defined in \eqref{gamma}, is positive. Now, with these redefinitions, we provide analogous estimates to those obtained in the case of cosmological-type solutions. 

\begin{lem}\label{LEM1} If $\partial_r \alpha (t,r) > 0 $ globally in spacetime, then the energy density $\hat e(t,r)$  is always nonnegative. Moreover,  we have
	\begin{equation}\label{ME}
	|\hat p (t,r) | \leq \hat e(t,r) .
	\end{equation}
\end{lem} 
The proof of \eqref{ME_new}, considering the constraints, is obtained in a similar way as in the previous section, see Lemma \ref{LEM} for more details. Lemma \ref{LEM1} is always useful to understand the right notion of energy.

\medskip

Finally, similar to the previous section, with the formulation of energy and momentum densities given in \eqref{e_spacelike}-\eqref{m_spacelike},  the identity equations obtained in the Corollary \refeq{eqContinuidad_new}, provide the following modified continuity equations:
\begin{lem}\label{EC:spacelike} Let $(\Lambda,\phi, \alpha)$ solutions of the system \eqref{sistema1}, and $\alpha(t,r)$ spacelike, then, we have the following continuity equations
\begin{equation}\label{CEspacelike}
		\begin{aligned}
		&\partial_t \hat e(t,r)+\partial_r \hat p(t,r)=0,\\
			&{\color{black}\partial_t \hat p(t,r)+\partial_r \hat e(t,r)=4\sinh^2(\Lambda)\left(\phi_r^2-\phi_t^2\right)+\Lambda_r^2-\Lambda_t^2.
}
		\end{aligned}
	\end{equation}
\end{lem}

The proof of this result is obtained in a similar way to the previous subsection. An important remark obtained from \eqref{CEspacelike} is the following: in this set of identities, the role of energy is played by the momentum, and vice versa. This somehow harmless condition destroys possible computations of decay by showing that the quantity that decays has no particular positivity. However, we expect to consider Lemma \ref{EC:spacelike} in forthcoming works.
	
\section{Virial Estimates for Cosmological-type Solutions}\label{Sect:5b}

Let us come back to the setting already worked in Subsection \ref{CTS}. In what follows, let us consider $(\Lambda,\phi,\alpha)$ globally defined in time and continuous such that
\begin{equation}\label{p dom e}
\hbox{\eqref{cosmological type} and \eqref{uniform} are satisfied.}
\end{equation}
Note that \eqref{alfax_alfat} is a consequence of assuming \eqref{cosmological type} in Theorem \ref{MT2}. 
Finally,
\[
E[\Lambda,\phi; \alpha] := \int_{\mathbb{R}} \hat e(t,x)dx 
\]
 is well-defined for all time and bounded: 
\begin{equation}\label{finite_energy}
0\leq  E[\Lambda,\phi; \alpha]  \leq \sup_{t\in\mathbb R} E[\Lambda,\phi;\alpha] < +\infty.
 \end{equation}
Notice that this time $E[\Lambda,\phi; \alpha] $ is not conserved (see \eqref{CEtimelike}). 

\begin{rem}
Condition \eqref{finite_energy} is not empty, for instance if the data is given as in Theorem \ref{GLOBAL0} and $\alpha$ satisfies the time-like condition \eqref{condicion:timelike} and \eqref{dt_alpha_p}, then   \eqref{supuesto}-\eqref{supuesto1} ensures that the energy is bounded in time as in \eqref{finite_energy}. See Lemma \ref{EdCo} for a proof.
\end{rem}

We introduce a Virial identity for the Einstein field equation \eqref{sistema1}. Indeed, let $\rho$ be a smooth bounded function with $L^1\cap L^\infty$ integrable derivative. Let $\omega(t)$ be a smooth positive function to be chosen later, not necessarily varying in time. Finally,  for $v\in (-1,1)$ let
\begin{equation} \label{virial}
    \begin{aligned}
      \mathcal{I}(t):= &~{} -\int \rho\left( \frac{x-vt}{\omega(t)}\right)\kappa\partial_x \alpha \left(4\sinh^2(\Lambda)((\partial_t \phi)^2+(\partial_x\phi)^2)+
         (\partial_t\Lambda)^2+(\partial_x\Lambda)^2 \right)dx\\
         &~{} \quad  +\int \rho \left( \frac{x-vt}{\omega(t)}\right) \kappa \partial_t \alpha \left( 2\partial_x\Lambda\partial_t\Lambda+ 8\partial_x\phi \partial_t\phi \sinh^2(\Lambda) \right)dx\\
         =:  &~{} -\int  \rho\left( \frac{x-vt}{\omega(t)}\right) \hat p (t,x)dx.
  \end{aligned}
\end{equation}
A time-dependent weight $\omega(t)$ was already considered in \cite{Alejo2018, JT}, but $\omega(t)=const.$ is also perfectly possible. The choice of $\mathcal{I}(t)$ is motivated by the momentum and energy densities.

\begin{lem}[Virial identity]\label{Virial2} One has $\mathcal I(t)$ well-defined and bounded in time, and 
\begin{equation}\label{dt I}
 \begin{aligned}
    \frac{d}{dt}\mathcal{I}(t)=&~{} -\frac{\omega'(t)}{\omega (t)}\int \frac{x-vt}{\omega(t)} \rho'\left( \frac{x-vt}{\omega(t)}\right)\hat p(t,x)\\
                                         &~{} +\frac{v}{\omega(t)}\int \rho'\left( \frac{x-vt}{\omega(t)}\right)\hat p(t,x)\\
                                         &~{} + \frac{1}{\omega(t)}\int \rho'\left( \frac{x-vt}{\omega(t)}\right)\hat e(t,x).
 \end{aligned}
\end{equation}
\end{lem}
\begin{proof} The proof of \eqref{dt I} follows immediately from the Lemma \refeq{EC:timelike}. The proof of boundedness of $\mathcal I(t)$ goes as follows: from \eqref{virial}, the boundedness of $\rho$ and \eqref{p dom e},
\[
\left| \mathcal I(t) \right| \leq \int  |\rho| \left( \frac{x-vt}{\omega(t)}\right) |\hat p| (t,x)dx \lesssim \int  |\hat p| (t,x)dx \leq \int \hat e (t,x)dx,
\]
therefore from \eqref{finite_energy} we obtain $\sup_{t\geq 0} |\mathcal I(t)| <+\infty.$
\end{proof}

\subsection{Virial estimates} Now we are ready to use the previous identities. 

\medskip

We choose $\omega$ and $\rho$. Let $\omega(t)= const.$ or 
\begin{equation}\label{lambda}
\omega(t):= \frac{t}{\log^2 t}, \quad \quad \frac{ \omega'(t)}{\omega (t)} = \frac{1}{t}\left(1-\frac{2}{\log t} \right).
\end{equation}
and
\begin{equation}\label{rho_def}
\rho:= \tanh, \quad \rho' = \sech^2.
\end{equation}

\begin{thm}\label{estimacion}
	Let $\omega$ and $\rho$ be given as in \eqref{lambda}-\eqref{rho_def}. Assume that the solution $(\Lambda, \phi, \alpha)(t)$ of the system \eqref{sistema1} is such that $\alpha$ satisfies \eqref{p dom e} and the finite energy condition \eqref{finite_energy} is satisfied.
Then  we have the averaged estimate
	\begin{equation}\label{desg1}
		\int_2^{\infty} \frac{1}{\omega(t)}\int\sech^2\left(\frac{x-vt}{\omega(t)}\right)\hat e(t,x)dxdt \lesssim  1,
	\end{equation}
	Moreover, there exists an increasing sequence $t_n\to +\infty$ such that 
	\begin{equation}\label{desg2}		
	\lim_{n\longrightarrow +\infty} \int\sech^{2}\left(\frac{x-vt_n}{\omega(t_n)}\right)\hat e(t_n,x)dx=0.
	\end{equation}
\end{thm}
In order to show Theorem \ref{estimacion}, we use the new Virial identity for \eqref{virial} presented
for the Einstein field equation \eqref{sistema1}.

\begin{proof}
On the other hand, recall that we are considering that $\alpha$ is a positive solution of the one-dimensional wave equation, with time-like gradient and with positive time derivative in all spacetime, therefore, we can use the Lemma \ref{LEM} and the Lemma \ref{Virial2}, then, we get
\[
  \frac{d}{dt}\mathcal{I}(t) =: \mathcal{J}_1+ \mathcal{J}_2+\mathcal{J}_3.
\]
First of all, we consider $\mathcal{J}_1$. If $\omega(t)$ is constant, there is nothing to prove. Assume now $\omega(t)$ given as in \eqref{lambda}. We have
\begin{equation*}
|\mathcal{J}_1|\leq \frac{|\omega'(t)|}{\omega (t)}\int \frac{|x-vt|}{\omega(t)} \rho'\left( \frac{x-vt}{\omega(t)}\right)|\hat p(t,x)|dx \leq \frac{|\omega'(t)|}{\omega (t)} \int \frac{|x-vt|}{\omega(t)} \rho'\left( \frac{x-vt}{\omega(t)}\right)\hat e(t,x)dx.
\end{equation*}
From the definition of $\omega(t)$ and using Cauchy's inequality for $\delta> 0$ small, we have: 
\begin{equation*}
\begin{aligned}
|\mathcal{J}_1| \leq &~{}  \frac{C_{\delta}\omega(t)}{t^2} \sup_{x \in \mathbb{R}}\left( \frac{(x-vt)^2}{\omega^2(t)}|\rho'|\left( \frac{x-vt}{\omega(t)}\right)\right)\int \hat e(t,x)dx\\
&~{}+{\color{black}\frac{\delta}{\omega(t)}\int \rho'\left( \frac{x-vt}{\omega(t)}\right) \hat e(t,x)dx}\\
\leq &~{} \dfrac{C}{t \log^2 t} + {\color{black}\frac{\delta}{\omega(t)}\int \rho'\left( \frac{x-vt}{\omega(t)}\right) \hat e(t,x)}dx.
\end{aligned}
\end{equation*}
Now,
\[
\left| \mathcal J_2(t)\right| \leq \frac{|v|}{\omega(t)}\int \rho'\left( \frac{x-vt}{\omega(t)}\right) |\hat p(t,x)| \leq \frac{|v|}{\omega(t)}\int \rho'\left( \frac{x-vt}{\omega(t)}\right) \hat e(t,x).
\]
Finally, $\mathcal J_3(t)$ does not need any bound at all. In any case, $\omega(t)= const.$ or $\omega(t)$ as in \eqref{lambda}, one has the following: if $\delta>0$ is small:
\begin{equation}\label{variacionI}
\begin{aligned}
\frac{d}{dt}\mathcal{I}(t) \geq &~{} \frac{1-|v|-\delta}{\omega(t)}\int  \rho'\left( \frac{x-vt}{\omega(t)}\right)\hat e(t,x) -\frac{C_{\delta}}{t\log^2 t}.
\end{aligned}
\end{equation} 
After integration in time  in \eqref{variacionI} and since the term $\frac{C}{t \log^2 t}$ integrates finite, we get \eqref{desg1} Finally,  \eqref{desg2} is obtain from \eqref{desg1} and the fact that $\omega^{-1}(t)$ is not integrable in $[2,\infty).$
\end{proof}

\subsection{Proof of the Theorem \ref{MT2} }
First of all, notice that the RHS in \eqref{CEtimelike} satisfies (with $h_1$ given in \eqref{h1})
\[
\left| 4\sinh^2(\Lambda)\left(\phi_x^2-\phi_t^2\right)+\Lambda_x^2-\Lambda_t^2 \right| \leq h_1.
\]
Using the the Lemma \refeq{LEM}, Lemma \refeq{EC:timelike}, \eqref{New_Estimation} and integration by part we have
\begin{equation*}
\begin{aligned}
& \left|\frac{d}{dt} \int \sech^4\left( \frac{x-vt}{\omega(t)}\right)\hat e(t,x)dx\right|\\
 &~{}  \leq \frac{|\omega'(t)|}{\omega(t)} \int   \frac{x-vt}{\omega(t)} |(\sech^4)'|\left( \frac{x-vt}{\omega(t)}\right)\hat e(t,x)dx\ + \frac{4|v|}{\omega(t)} \int \sech^4 \left( \frac{x-vt}{\omega(t)} \right)\hat e(t,x)dx \\
 &~{} \quad + \int \frac{1}{\omega(t)} \sech^4\left( \frac{x-vt}{\omega(t)}\right)  |\hat p| dx + \int \sech^4\left( \frac{x-vt}{\omega(t)}\right) h_1 dx\\
  &~{}\leq \frac{2|v| +1 +|\omega'(t)|}{\omega(t)}\int \sech^2\left( \frac{x-vt}{\omega(t)}\right)  \hat e(t,x)dx +\frac{2}{\omega(t)} \int \sech^4\left( \frac{x-vt}{\omega(t)}\right) \omega(t) \frac{\partial_t \alpha}{\alpha}\hat e(t,x) dx.
\end{aligned}
\end{equation*}
Finally, notice that from \eqref{eqns:alpha}, \eqref{u bar u} and \eqref{lambda},
\[
\omega(t)  \sech \left( \frac{x-vt}{\omega(t)}\right)  \frac{\partial_t \alpha}{\alpha}  \lesssim 1. 
\]
This estimation is possible since $\alpha_0'$ and $\alpha_1$ are compactly supported and Schwartz, respectively. Therefore, for every $n\geq 0$,
\[
\alpha \geq \frac12, \quad   \partial_t \alpha \lesssim_n \frac{1}{\varphi^{n}(u)} +\frac{1}{\varphi^{n}(\underline u)}, \quad u,\; \underline u \hbox{ as in \eqref{u bar u}}.
\]
Consequently, for $n$ sufficiently large but fixed,
\[
\begin{aligned}
\omega(t)  \sech \left( \frac{x-vt}{\omega(t)}\right)  \frac{\partial_t \alpha}{\alpha}  \lesssim  &~{} \omega(t)  \sech \left( \frac{x-vt}{\omega(t)}\right) \left( \frac{1}{\varphi^{n}(u)} +\frac{1}{\varphi^{n}(\underline u)} \right)\\
\lesssim  &~{} \frac{\omega(t) }{\varphi^n ((1-|v|)|t|)} +  \omega(t)  \sech \left( (1-|v|)\frac{|t|}{\omega(t)}\right)\\
  \lesssim &~{} 1+ \omega(t)  \sech \left( (1-|v|)\log^2 t\right) \lesssim \omega(t)  t^{-(1-|v|)\log t } \lesssim 1.
\end{aligned}
\]
We conclude that
\begin{equation*}
\begin{aligned}
 \left|\frac{d}{dt} \int \sech^4\left( \frac{x-vt}{\omega(t)}\right)\hat e(t,x)dx\right|  \lesssim  &~{} \frac{1}{\omega(t)}\int \sech^2\left( \frac{x-vt}{\omega(t)}\right)  \hat e(t,x)dx .
\end{aligned}
\end{equation*}
Then integrating in time for $t< t_n$, we have
\begin{equation*}
\begin{aligned}
&~{} \left|\int \sech^4\left( \frac{x-vt_n}{\omega(t)}\right)\hat e(t_n,x)dx- \int \sech^4\left( \frac{x-vt}{\omega(t)}\right)\hat e(t,x)dx\right| \\
&~{}\quad \quad \quad  \leq  \int_{t}^{t_n} \frac{1}{\omega(t)}\int \sech^2\left( \frac{x-vs}{\omega(t)}\right)\hat e(s,x)dxds.
\end{aligned}
\end{equation*}
sending $n \longrightarrow \infty$ and using \eqref{desg2}, we get
\begin{equation*}
\begin{aligned}
\left| \int \sech^4\left( \frac{x-vt}{\omega(t)}\right)\hat e(t,x)dx\right|\leq &~{} \int_{t}^{\infty} \frac{1}{\omega(t)}\int \sech^2\left( \frac{x-vs}{\omega(t)}\right)\hat e(s,x)dxds.
\end{aligned}
\end{equation*}
Now, sending $t\longrightarrow \infty$, we get
\[
 \lim_{t\longrightarrow \infty}\int \sech^4\left( \frac{x-vt}{\omega(t)}\right)\hat e(t,x)dx =0.
\]
Using again the definition of the $h_1$ and $h_2$ given in \eqref{h1}-\eqref{h2}, and the condition \eqref{condicion:timelike} that ensures $2|h_2| \leq h_1$, one has
\begin{equation}\label{decay22}
\hat e = -\kappa \partial_t \alpha h_1 +2\kappa \partial_x \alpha h_2 \geq (\partial_t \alpha -|\partial_x \alpha|)|\kappa | h_1.
\end{equation}
In addition, using that $\alpha>0$,  $\partial_t \alpha > 0$ and Lemma \refeq{condicionSigno0}, we estimate $|\kappa|(\partial_t \alpha-|\partial_x \alpha|)$ as follows
\begin{equation*}
\begin{aligned}
|\kappa|(\partial_t \alpha-|\partial_x \alpha|)= &~{} \frac{\alpha}{|(\partial_x \alpha)^2-(\partial_t \alpha)^2|}(\partial_t \alpha - |\partial_x \alpha|) \\
= &~{} \frac{\alpha}{|\partial_x \alpha| +\partial_t \alpha} \geq  \frac{1}{2}\frac{\alpha}{\partial_t \alpha}.
\end{aligned}
\end{equation*}
Therefore,
\begin{equation}\label{New_Estimation}
 \frac{1}{2}\frac{\alpha}{\partial_t \alpha} h_1 \leq  \hat e.
\end{equation}
Now, for $\delta$ sufficiently small and $|v|<1$, the first term in the right side of the equation \eqref{variacionI}  in the Theorem \refeq{estimacion}, can be estimated as follows
\begin{equation*}
\begin{aligned}
\frac{1-|v|-\delta}{\omega(t)}\int  \rho'\left( \frac{x-vt}{\omega(t)}\right)\hat e(t,x) 
&~{}  \gtrsim \frac{1}{\omega(t)} \int  \rho'\left( \frac{x-vt}{\omega(t)}\right)\frac{\alpha}{2\partial_t \alpha}h_1(t,x).
\end{aligned}
\end{equation*} 
Finally, from the hypothesis \eqref{uniform} we have $(\partial_t\alpha)^{-1}>c_0>0$, and from the inequalities \eqref{decay22} and \eqref{New_Estimation} we get the lower bound
\begin{equation*}
\begin{aligned}
 &~{}\int \sech^2\left( \frac{x-vt}{\omega(t)}\right)\left( (\partial_x \Lambda)^2+(\partial_t \Lambda)^2 +\sinh^2(\Lambda)((\partial_x \phi)^2+(\partial_t \phi)^2) \right)(t,x)dx\\
&~{}\quad \quad \quad \quad \quad \quad  \lesssim \int \sech^2\left( \frac{x-vt}{\omega(t)}\right)\hat e(t,x)dx,
\end{aligned}
\end{equation*}
which finally shows the validity of Theorem \ref{MT2} and the proof of \eqref{desg200}.

\begin{rem}
Notice that from \eqref{eqns:alpha} we have $\partial_t\alpha$ uniformly in the Schwartz class, and one has the lower bound $(\partial_t\alpha)^{-1}>c_0>0$. Indeed, 
\[
\frac{\alpha}{\partial_t \alpha} \gtrsim \frac{1}{\alpha_1 - |\alpha_0'|}\gtrsim 1.
\]
consequently, the solutions from Theorem \ref{GLOBAL0} satisfy Theorem \ref{MT2} as well.
\end{rem}

\section{Applications to gravitational solitons}\label{Sect:5}

 The purpose of this section is to analyze the dynamics of certain exact solutions to the Einstein field equations that can be derived from the Belinski-Zakaharov transform. 

\subsection{Generalized Kasner metric background}
We begin our analysis by considering vacuum cosmologies described by the Kasner-type model. The Kasner metric, being one of the first known exact solutions in relativistic cosmology, remains one of the most important exact solutions in GR. The generalized Kasner metric can be written in the diagonal form as follows:
\begin{equation}\label{stinterval2}
	ds^2=f_0(t,x)(dx^2-dt^2)+\alpha e^{u_0}dy^2+\alpha e^{-u_0}dz^2,
\end{equation}
where the function $u_0$ is given by
\begin{equation}\label{alfaalfa}
	u_0(t,x)=d \ln \alpha,
\end{equation}
and $d$ is an arbitrary parameter, the \textit{Kasner parameter}. It can be chosen either positive or negative, for instance $d=\pm 1$ corresponds to a region of Minkowski, $d=0$ is an LRS space with Petrov type metric D. The $x$ axis expands as time evolves if $|d|>1$ and contracts if $|d|<1$ \cite{belinski2001gravitational}. The original Kasner metric \cite{Kasner1921} is obtained by taking $\alpha=t$ (timelike) and describes an anisotropic universe without matter. The original Kasner's choice does not fit into the assumptions of Theorem \ref{GLOBAL0}, and will be studied elsewhere.

\medskip

In this work, we will assume that $d\geq 1$ to ensure the correct finite energy condition. Naturally one has
\[
\det g = \alpha^2, \quad g= \alpha\, \hbox{diag} \left( e^{u_0}, e^{-u_0} \right).
\]
As mentioned in the previous section, in order to identify the spacetime \eqref{stinterval2} with a cosmological model, the function $\alpha(t,x)$ must be globally timelike. If one compares (\refeq{stinterval2}) with (\refeq{diag1}),  we have that $\Lambda$ and $\phi$ should be given by 
\begin{equation}\label{lambKasner}
	\Lambda^{(0)}(t,x)=u_0, \quad \hbox{and} \quad  \phi^{(0)}=n\pi, \quad n\in \mathbb Z.
\end{equation}


{\color{blue}  \subsection{Proof of the Theorem \ref{MT3}}\label{prub:MT3} In the following, we will present a set of results that will allow us to conclude the proof of the Theorem \ref{MT3}.
First of all, note that the Kasner-type seed metric satisfies the finite energy condition, provided that $\alpha$ satisfies the smallness conditions obtained from Theorem \ref{GLOBAL0}. Therefore, the result holds for this particular case:}

\begin{lem}
If the function $\alpha(t,x)$   satisfies the hypotheses of the Theorem \ref{GLOBAL0} with $|\tilde \alpha_0| < \alpha_1$ and $\partial_t \alpha > 0$, then, the Kasner-type seed solution $(\Lambda^{(0)},\phi^{(0)})$ of the \eqref{sistema1} has finite nonnegative energy.
\end{lem}
\begin{proof}
 The energy density  proposed in \eqref{e_timelike_0}, in this case, has the following structure:
\begin{equation*}
	\begin{aligned}\label{energykasner}
		\hat e_0=&-\frac{\alpha \partial_t \alpha}{(\partial_x\alpha)^2-(\partial_t\alpha)^2}\left[(\partial_x \Lambda^{(0)})^2+(\partial_t \Lambda^{(0)})^2 \right]+\frac{2\alpha \partial_x \alpha}{(\partial_x \alpha)^2-(\partial_t \alpha)^2}\partial_x\Lambda^{(0)}\partial_t\Lambda^{(0)}.
	\end{aligned}
\end{equation*}
 Using (\refeq{lambKasner}) we get:
\begin{equation*}
	\begin{cases*}
		\partial_x \Lambda^{(0)}= \dfrac{d\partial_x \alpha}{\alpha},\\
		\partial_t \Lambda^{(0)}= \dfrac{d\partial_t \alpha}{\alpha}.
	\end{cases*}
\end{equation*}
Then, since $(\partial_x \alpha)^2 -(\partial_t \alpha)^2  <0$ we can simplify the expression \eqref{energykasner} as:
\begin{equation}
	\begin{aligned}\label{kasner0}
		\hat e_0&= \frac{d^2\alpha\partial_t \alpha}{(\partial_x \alpha)^2-(\partial_t \alpha)^2}\left( \frac{(\partial_x \alpha)^2-(\partial_t \alpha)^2}{\alpha^2}\right)=d^2\partial_t(\ln \alpha).
	\end{aligned}
\end{equation}
Notice that $\hat e_0$ is nonnegative and well-defined thanks to the timelike condition on $\alpha$.  And a similar way, we have the following  momentum density:
\begin{equation*}
	\hat p_0=d^2\partial_x(\ln \alpha).
\end{equation*}
Given the definition of $\partial_t \alpha$ in terms of initial conditions $\tilde \alpha_0, \alpha_1,$  and using \eqref{Llogalpha},  we can conclude that the energy \eqref{kasner0} corresponding to this background metric is finite, i.e. 
\[
E[\Lambda^{(0)},\phi^{(0)}; \alpha](t)=\int \hat e_0dx=\int \frac{1}{2}[L(\ln \alpha)+\underline{L}(\ln \alpha)]dx <\infty.
\]
The proof is complete.
\end{proof}
Theorems \ref{GLOBAL0} and \ref{MT2} imply in this case that  for any $|v|<1$ and $\omega(t)=t (\log t)^{-2}$, 
\[
\lim_{t\to +\infty }\int_{|x-vt| \leq \omega(t)}  \frac1{\alpha^2}\left(\alpha_t^2 + \alpha_x^2\right) (t,x) dx =0, 
\]
as naturally expected for solutions of the 1D wave equation. What is more interesting is the case of 1-soliton solutions.

\begin{rem}
As we can notice, until now it has been enough to impose certain constraints on the function $\alpha$, to understand how we must define the energy in each spacetime and to understand how the solution of the system behaves in long time. 
At this point, it is important to emphasize that, in the framework of the inverse scattering theory for the Einstein equation, proposed by Belinski and Zakharov, in addition to the function $\alpha(t,x)$, it is necessary to introduce its conjugate derivative $\beta(t,x)$, related to $\alpha(t,x)$ by the identities 
\[
\partial_t\beta=\partial_x\alpha, \quad \mbox{and} \quad  \partial_x\beta=\partial_t\alpha,
\]
$\beta$ is introduced with the aim of describing the 1-soliton solution. In the setting of Theorem \ref{GLOBAL0}, from \eqref{eqn:alpha_DA} one sees that this function is given by
\begin{equation}\label{betaeq}
\beta(t,x):= C+\tilde\alpha_0(2u)-\tilde \alpha_0(-2\underline{u})+\int_{0}^{2u}\alpha_1(s)ds+\int_{0}^{-2\underline{u}}\alpha_1(s)ds, \quad C\in\mathbb R.
\end{equation}
$\beta(t,x)$ is a second independent solution of the one-dimensional wave equation, and will be automatically spacelike in our setting. Indeed, from \eqref{u bar u},
\[
\beta_t = \tilde\alpha_0'(2u) + \tilde \alpha_0'(-2\underline{u})+\alpha_1(2u) - \alpha_1(-2\underline{u}),
\]
and 
\[
\beta_x =   \tilde\alpha_0'(2u) - \tilde \alpha_0'(-2\underline{u}) + \alpha_1(2u) + \alpha_1(-2\underline{u}) >0.
\]
Consequently,
\[
\beta_x - \beta_t = -2\tilde\alpha_0'(-2\underline u)  + 2\alpha_1(-2\underline u) >0, 
\]
and
\[
\beta_x + \beta_t =  2\tilde \alpha_0'(2u) + 2\alpha_1(2u ) >0. 
\]

Belinski and Zakharov postulate that there is a smooth, one-to-one, surjective mapping between $t,x$ and $\alpha, \beta$, see \cite{Belinsky}. 
\end{rem}

In the setting of Theorem \ref{GLOBAL0}, it is clear that $\beta$ defines a bounded function in spacetime. Additionally,
\[
\lim_{t\to +\infty }\int_{|x-vt| \leq \omega(t)}  \frac1{\beta^2}\left(\beta_t^2 + \beta_x^2\right) (t,x) dx =0, 
\]
\subsubsection{One Soliton Solution}
Belinski and Verdaguer \cite[p. 47]{belinski2001gravitational} introduced the one soliton solution with \textit{Kasner background}. Let $w \in\mathbb R$ be a fixed parameter. Let $\mu$ be
\begin{equation}\label{mu}
	\mu:=w-\beta -\sqrt{(w-\beta)^2-\alpha^2},
\end{equation}
where $\beta$ solves \eqref{betaeq}, namely $\partial_t\beta=\partial_x\alpha$. Then the 1-soliton with Kasner background is given by
\begin{equation}\label{g1}
	g^{(1)}=\frac{1}{\mu \cosh(\rho)}\left[\begin{array}{cc}
		e^{\Lambda^{(0)}}(\mu^2e^{\rho}+\alpha^2e^{-\rho}) & \alpha^2-\mu^2\\
		\alpha^2-\mu^2 & e^{-\Lambda^{(0)}}(\alpha^2e^{\rho}+\mu^2e^{-\rho})
	\end{array} \right],
\end{equation}
where
\begin{equation*}
\begin{aligned}
	& \rho=d\ln\left(\frac{\mu}{\alpha}\right)+C, \quad C\in \mathbb{R},\\
	& f=f^{(0)}\alpha^{1/2}\mu\cosh(\rho)(\alpha^2-\mu^2)^{-1}.
\end{aligned}	
\end{equation*}
Some important remarks are in order. First, from \eqref{mu} one can see that if $w$ is sufficiently large, $\mu$ is real-valued. Assume $w>0$ sufficiently large such that $\mu$ is real valued and positive. For the purposes of this work, we need a further simplification of \eqref{g1}. Assuming for simplicity $C=0$ in $\rho$, after some computations we get
\begin{equation}\label{eq:ref_1}
{\color{blue}e^{\Lambda^{(0)}} = \alpha^d}, \quad e^\rho = \left(\frac{\mu}{\alpha}\right)^d, \quad e^{-\rho} =  \left(\frac{\mu}{\alpha}\right)^{-d}, \quad  \cosh \rho = \frac12\left(  \left(\frac{\mu}{\alpha}\right)^d + \left(\frac{\mu}{\alpha}\right)^{-d} \right),
\end{equation}
and for 
\begin{equation}\label{m}
m:= \frac{\mu}{\alpha} = \frac1{\alpha}( w-\beta -\sqrt{(w-\beta)^2-\alpha^2}),
\end{equation}
one obtains
\begin{equation}\label{g1_new}
\begin{aligned}
	g^{(1)}= &~{} \frac{2}{\mu \left(  \left(\frac{\mu}{\alpha}\right)^d + \left(\frac{\mu}{\alpha}\right)^{-d} \right)}\left[\begin{array}{cc}
		\alpha^d \left(\mu^2\left(\frac{\mu}{\alpha}\right)^d+\alpha^2\left(\frac{\mu}{\alpha}\right)^{-d} \right) & \alpha^2-\mu^2\\
		\alpha^2-\mu^2 & \alpha^{-d}\left(\alpha^2\left(\frac{\mu}{\alpha}\right)^d+\mu^2\left(\frac{\mu}{\alpha}\right)^{-d}\right)
	\end{array} \right]\\
	= &~{} \frac{2 \alpha }{  m^{d} +m^{-d}}\left[\begin{array}{cc}
		\alpha^d  \left(m^{d+1}+m^{-d-1} \right) & \frac1m -m \\
		\frac1m -m  & \alpha^{-d}\left(m^{d - 1} + m^{-d + 1}\right)
	\end{array} \right].
\end{aligned}	
\end{equation}
A first glimpse of the $g^{(1)}$ reveals that it will behave closely to the functions $\alpha$ and $\beta$. In this sense, we can say that the associated propagation speed must coincide with a support on the light cone. Comparing \eqref{g1_new} with (\refeq{diag1}), we find that
\[
\begin{aligned}
\cosh{\Lambda} +\cos 2\phi \sinh{\Lambda}  = & ~{}\frac{2\alpha^d  \left(m^{d+1}+m^{-d-1} \right)}{m^{d} +m^{-d}} \\
 \cosh{\Lambda}-\cos 2\phi \sinh{\Lambda} = & ~{} \frac{2\alpha^{-d}  \left(m^{d-1}+m^{-d+1} \right)}{m^{d} +m^{-d}}\\
 \sin 2\phi \sinh{\Lambda} =&~{}\frac{2 \left(m^{-1} -m \right)}{m^{d} +m^{-d}}.
 \end{aligned}
\]

and therefore
\begin{equation}\label{Kasner1solitonLambda}
\cosh \Lambda = \frac{\alpha^d  \left(m^{d+1}+m^{-d-1} \right)+ \alpha^{-d}  \left(m^{d-1}+m^{-d+1} \right)}{m^{d} +m^{-d}},
\end{equation}
\begin{equation}\label{Kasner1solitonPhi}
\tanh 2\phi = \frac{2 \left(m^{-1} -m \right)}{\alpha^d  \left(m^{d+1}+m^{-d-1} \right)-\alpha^{-d}  \left(m^{d-1}+m^{-d+1} \right)}.
\end{equation}
Since $\alpha,m>0$ by hypothesis and $a+\frac1a\geq 1$ for $a>0$, we get
\[
\begin{aligned}
{\color{blue}\cosh(\Lambda)}&~{} =\frac{\alpha^d  \left(m^{d+1}+m^{-d-1} \right)+ \alpha^{-d}  \left(m^{d-1}+m^{-d+1} \right)}{m^{d} +m^{-d}} \\
&~{} =  \frac{(\alpha^{d}m + \alpha^{-d}m^{-1})m^d  +( \alpha^{d}m^{-1} + \alpha^{-d}m )m^{-d} }{m^{d} +m^{-d}} \geq 1.
\end{aligned}
\]

Using that $\sinh(\hbox{arccosh} x)=\sqrt{x^2-1}$ for $|x|\geq 1$, we get
\begin{equation}\label{SH}
\sinh^2 \Lambda = \left(  \frac{(\alpha^{d}m + \alpha^{-d}m^{-1})m^d  +( \alpha^{d}m^{-1} + \alpha^{-d}m )m^{-d} }{m^{d} +m^{-d}} \right)^2-1.
\end{equation}
As a first application, we use Theorem \ref{GLOBAL0} to provide the following global existence result:

\begin{cor}\label{cor5p1}
Under the smallness hypotheses on $\alpha$ from Theorem \ref{GLOBAL0}, suitable perturbations of the 1-soliton with Kasner metric background \eqref{g1_new} are globally defined.
\end{cor}

Additionally, it is not difficult to realize that $(\Lambda,\phi,\alpha)$ {\color{blue} induce globally} defined finite energy solutions. Consequently, Theorem \ref{MT2} allows us to conclude

\begin{cor}\label{cor5p2}
Under the hypothesis on $\alpha$ obtained from Theorem \ref{GLOBAL0}, $g^{(1)}$ in the form $(\Lambda, \phi,\alpha)$ satisfies the assumptions in Theorem \ref{MT2}, and consequently 
\[
\lim_{t\to +\infty }\int_{|x-vt| \leq \omega(t)}  \left(\Lambda_t^2 +\Lambda_x^2 + \sinh^2(\Lambda) (\phi_t^2 + \phi_x^2)\right) (t,x) dx =0.
\]
\end{cor}

Both corollaries complete the proof of Theorem \ref{MT3}. 

\begin{proof}[Proof of Corollary \ref{cor5p1}]
We have to verify the hypotheses in Theorem \ref{GLOBAL0}. Indeed, notice that {\color{blue} under the smallness assumptions on $\alpha$ and} from \eqref{betaeq} (by choosing $C=1$)
\[
\alpha = 1 + \tilde \alpha_0, \quad \beta= 1 + \tilde \beta_0.
\]
{\color{blue} Similarly,  taking into account the identities given in \eqref{eq:ref_1} and the definition of $m$ given by \eqref{m},  it follows that the dynamics of $m$ are completely determined by $\alpha$   (in terms of $\tilde \alpha_0, \alpha_1$), then,  we get that $m$ has the same asymptotic behavior, converging to a constant as space tends to infinity}. 
{\color{teal}  Indeed, we can rewritte $m$ in \eqref{m} as follows: 
\[
m= \frac{\alpha}{s+\sqrt{s^2-\alpha^2}}, \qquad \mbox{with} \qquad s(t,x):= w- \beta(t,x),
\]
then, since $\beta$ is a bounded function, constant after some large distance, there exists $s_0>0$ such that 
\begin{equation*}
m(t,x)=\frac{1+\tilde{\alpha}}{s+\sqrt{s^2-\alpha^2}}\to \frac{1+\tilde \alpha}{s_0}, \quad \hbox{as} \quad |x|\to +\infty.
\end{equation*}
When $t=0$ we have $m(0,x)\leq C\frac{1+\tilde \alpha_0}{s_0},$ then $m(0,x)=\mathcal{O}(1/s_0)$. Let us see that $\Lambda$ and $\phi$ satisfy the conditions of smallness required by Theorem \ref{GLOBAL0}. First we define
\[\mathcal{G}(\alpha,m):=  \frac{(\alpha^{d}m + \alpha^{-d}m^{-1})m^d  +( \alpha^{d}m^{-1} + \alpha^{-d}m )m^{-d} }{m^{d} +m^{-d}}, \]
then using the Taylor expansion  at the point $(1,1/s_0)$, we can write 
\begin{equation}\label{T1}
\mathcal{G}(\alpha,m)=\mathcal{G}(1,1/s_0)+\partial_{\alpha}\mathcal{G}(1,1/s_0)\tilde \alpha+\partial_{m}\mathcal{G}(1,1/s_0)(m-1/s_{0})+\mathcal{O}((\tilde \alpha)^2+(\delta m)^2),
\end{equation}
with $\delta m =m-1/s_0$. Next, from \eqref{Kasner1solitonLambda} and using the Taylor expansion  at the point $\varrho:= \mathcal{G}(1,1/s_0)$, we can rewrite $\Lambda$ as:
\[
\Lambda = \cosh^{-1}(\varrho)+\frac{1}{\sqrt{\varrho^2-1}}(\mathcal{G}-\varrho)+\mathcal{O}((\mathcal{G}-\varrho)^2).
\]
Using \eqref{T1}-\eqref{SH}, we obtain 
\begin{equation}\label{L:tilde}
\Lambda= \cosh^{-1}(\varrho)+\frac{\partial_{\alpha}\mathcal{G}(1,1/s_0)\tilde \alpha+\partial_{m}\mathcal{G}(1,1/s_0)\delta m}{\sinh(\hbox{arccosh}  (\varrho))}+\mathcal{O}(\delta^2).
\end{equation}
Therefore, if $\lambda=  \cosh^{-1}(\varrho)$ and 
\[
\tilde \Lambda (t,x):=  \frac{\partial_{\alpha}\mathcal{G}(1,1/s_0)\tilde \alpha+\partial_{m}\mathcal{G}(1,1/s_0)\delta m}{\sinh(\hbox{arccosh}  (\varrho))}+\mathcal{O}(\delta^2),
\]
it is then revealed that $\Lambda=\lambda +\widetilde{\Lambda}$  in \eqref{Kasner1solitonLambda} 
can be made arbitrarily small at $t=0$, depending on a parameter $\varepsilon.$ Analogously, the conditions for the function $\phi$ and for the corresponding derivatives are obtained. The remaining hypotheses are standard and satisfied in the standard way.} 
\end{proof}

\begin{proof}[Proof of Corollary \ref{cor5p2}]
Assume \eqref{cosmological type} and \eqref{uniform}. In order to apply Theorems \ref{MT20} and \ref{MT2}, we only need to check the finite energy condition for all time. This is clear from the form of $\Lambda=\lambda +\widetilde{\Lambda}$ in \eqref{L:tilde} and $\phi$ in \eqref{Kasner1solitonPhi}: every squared time and space derivative will involve squared derivatives on $\alpha$, $\mu$ and $\beta$, which have bounded in time finite energy. Finally, \eqref{SH} ensures the last part of the energy condition.
\end{proof}

\subsection{The Einstein-Rosen Metric} We study now a metric with cylindrical symmetry where our decay results do not apply. We will choose $\alpha=r >0$ as solution to 1D waves, such that 
\[
\alpha_t =0 ,\quad \alpha_r^2 -\alpha_t^2 =1>0, \quad \frac{\alpha_t}{\alpha} =0. 
\]
The cylindrical coordinates are $x^{\mu}={t,r}$ and $x^a={\varphi, z}.$ The metric will also be diagonal ($\varphi=0$). We then have the following spacetime interval \cite{Einstein1937}:
\begin{equation}\label{stinterval}
	ds^2:= f^{(0)}(-dt^2+dr^2)+e^{u_0}(rd\phi)^2+e^{-u_0}dz^2,
\end{equation}
where $f^{(0)}>0$ and $u_0$ are functions of $t,r$ and  $u_0(t,r)$ satisfies the ``cylindrical'' wave equation
\begin{equation}\label{stinterval1}
	\partial^{2}_tu_0=\frac{1}{r}\partial_{r}(r\partial_ru_0).
\end{equation}
This is the  \textit{Einstein-Rosen diagonal form}. As in the previous case, the Belinski-Zakharov setting is
\[ 
g=\alpha  \, \diag(e^{\Lambda^{(0)}}, e^{-\Lambda^{(0)}}), \quad \alpha =r, \quad u_0(t,r) = \Lambda^{(0)} -\ln r  .
\]
Then $\Lambda^{(0)}$ is as \eqref{diag1} if $\phi^{(0)}= n\pi$. It satisfies the equation 
\begin{equation*}
	\partial^{2}_t\Lambda^{(0)}=\frac{1}{r}\partial_{r}(r\partial_r\Lambda^{(0)}).
\end{equation*}

A particular choice for $\Lambda^{(0)}$ is given by 
\begin{equation}\label{L0}
	\Lambda^{(0)}= J_0(r)\sin(t),
\end{equation}
where $J_0$ denotes the  $0$-th order Bessel function. From \eqref{L0} clearly $\Lambda^{(0)}$ does not decay in time. For this case, the densities are given as follows:
\begin{equation*}
	\begin{aligned}
		e_0&= r((\partial_t \Lambda^{(0)})^2+(\partial_r \Lambda^{(0)})^2)\\
		p_0&= -2r\partial_t \Lambda^{(0)} \partial_r \Lambda^{(0)}.
	\end{aligned}
\end{equation*}
For completeness, the one soliton solution in this case was studied by Hadad in \cite{yaronhadad_2013}, and it is given as
\begin{equation*}
	g^{(1)}=\frac{1}{\mu \cosh(\gamma)}\left[\begin{array}{cc}
		r^2e^{u_0}\cosh(\gamma+\tilde{\gamma}) & \frac{r^2-\mu^2}{2\mu}\\
		\frac{r^2-\mu^2}{2\mu} & e^{-u_0}\cosh(\gamma-\tilde{\gamma})
	\end{array} \right],
\end{equation*}
where $\omega\in\mathbb R$,
\begin{equation*}
	\begin{cases*}
		\mu=\omega-t\pm\sqrt{(\omega-t)^2-r^2}\\
		\tilde{\gamma}=\ln\left(\frac{r}{|\mu|}\right)\mp \cosh^{-1}\left( \frac{\omega-t}{r}\right)\\
		\gamma= K+u_0+2\rho+\tilde{\gamma} \qquad K=\ln(C), \quad C>0\\
		\partial_t\rho=\frac{r}{\mu^2-r^2}(r\partial_t u_0+\mu \partial_r u_0)\\
		\partial_r\rho=\frac{r}{\mu^2-r^2}(r\partial_r u_0+\mu \partial_t u_0)\\
		f=C_0\sqrt{r}\frac{\mu}{\mu^2-r^2}\cosh(\gamma)f^0.
	\end{cases*}
\end{equation*}
Then, the fields $\Lambda$ and $\phi$ are given as:
\begin{equation*}
	\begin{aligned}
		\Lambda&=\cosh^{-1}\left(\frac{r}{2}e^{u_0}\cosh(\gamma+\tilde{\gamma})+\frac{e^{-u_0}}{2r}\cosh(\gamma-\tilde{\gamma})\right)\\
		\phi&= \frac{1}{2}\tan^{-1}\left(\frac{r^2-\mu^2}{4\mu\cosh(\gamma)h}\right),\\
		\mbox{where}\\
		h&= \frac{r}{2}e^{u_0}\cosh(\gamma+\tilde{\gamma})-\frac{e^{-u_0}}{2r}\cosh(\gamma-\tilde{\gamma}).
	\end{aligned}
\end{equation*}
A further study of this metric with other techniques will be done elsewhere.
\appendix

\section{Proof of the Lemma \ref{eqContinuidad0}}\label{App:B}

In this section, we prove, for completeness, the modified continuity equations \eqref{eqContinuidad}. First, let us start by writing the derivatives of the energy and momentum densities (respectively). Recall that, for this case, we have a full form for $h_1$ and $h_2$ introduced in \eqref{h1} and \eqref{h2},  as  follows:
\[
 h_1= {\color{black}\dfrac{(\partial_x \alpha)^2+(\partial_t \alpha)^2}{\alpha^2}}+ {\color{black}4\sinh^2(\Lambda)\Big((\partial_t \phi)^2+(\partial_x\phi)^2\Big)}+{\color{black}(\partial_t\Lambda)^2+(\partial_x\Lambda)^2}
  \]
and  
\[
h_2={\color{black}\dfrac{\partial_x\alpha \partial_t \alpha}{\alpha^2}}+{\color{black}\partial_x\Lambda \partial_t\Lambda} +{\color{black}4\partial_x \phi \partial_t \phi \sinh^2(\Lambda)}.
\]
Then, we can write the energy and momentum density as:
\begin{equation*}
\begin{aligned}
&e(t,x)=\kappa \partial_t \alpha h_1-2\kappa\partial_x \alpha h_2,\\
& p(t,x)=\kappa \partial_x \alpha h_1-2\kappa\partial_t \alpha h_2,
\end{aligned}
\end{equation*}
where
\begin{equation*}
\begin{aligned}
\kappa:=&\dfrac{\alpha}{\alpha_x^2-\alpha_t^2},\quad  \mbox{and}\\
\partial_t \kappa=& \frac{\alpha_t\alpha_x^2-\alpha_t\alpha_t^2-2\alpha\alpha_x\alpha_{xt}+2\alpha\alpha_t\alpha_{tt}}{(\alpha_x^2-\alpha_t^2)^2},\\
\partial_x \kappa=&\frac{\alpha_x\alpha_x^2-\alpha_x\alpha_t^2-2\alpha\alpha_x\alpha_{xx}+2\alpha\alpha_t\alpha_{tx}}{(\alpha_x^2-\alpha_t^2)^2}.
\end{aligned}
\end{equation*}
For the derivatives of  $h_1$ and $h_2$ we have:
\begin{equation*}
\begin{aligned}
\partial_t h_1=&{\color{black}2 \dfrac{(\alpha_x\alpha_{xt}+\alpha_t\alpha_{tt})\alpha^2-\alpha\alpha_t(\alpha_x^2+\alpha_t^2)}{\alpha^4}}+{\color{black}4\sinh(2\Lambda)\Lambda_t(\phi_x^2+\phi_t^2)}\\
&{\color{black}+ 8\sinh^2(\Lambda)(\phi-x\phi_{xt}+\phi_t\phi_{tt})}+{\color{black}2\Lambda_x\Lambda_{xt}+2\Lambda_t\Lambda_{tt}}\\
\partial_x h_1=& 2 \dfrac{(\alpha_x\alpha_{xx}+\alpha_t\alpha_{tx})\alpha^2-\alpha\alpha_x(\alpha_x^2+\alpha_t^2)}{\alpha^4}+4\sinh(2\Lambda)\Lambda_x(\phi_x^2+\phi_t^2)\\
&+ 8\sinh^2(\Lambda)(\phi_x\phi_{xx}+\phi_t\phi_{tx})+2\Lambda_x\Lambda_{xx}+2\Lambda_t\Lambda_{tx}\\
\partial_th_2= &\dfrac{(\alpha_{tt}\alpha_x+\alpha_t\alpha_{xt})\alpha^2-2\alpha \alpha_t^2\alpha_x}{\alpha^2}+\Lambda_{xt}\Lambda_t+\Lambda_x\Lambda_{tt}\\
&+4\sinh(2\Lambda)\Lambda_t\phi_t\phi_x+4\sinh^2(\Lambda)(\phi_{xt}\phi_t+\phi_x\phi_{tt})\\
\partial_xh_2= &\dfrac{(\alpha_{xx}\alpha_t+\alpha_x\alpha_{xt})\alpha^2-2\alpha \alpha_x^2\alpha_t}{\alpha^2}+\Lambda_{xx}\Lambda_t+\Lambda_x\Lambda_{tx}\\
&+4\sinh(2\Lambda)\Lambda_x\phi_t\phi_x+4\sinh^2(\Lambda)(\phi_{xx}\phi_t+\phi_x\phi_{tx})
\end{aligned}
\end{equation*}
The first step will be to prove the first equation in \eqref{eqContinuidad}, taking derivative in $x$ for energy density and derivative in $t$ for the momentum density, we have
\begin{equation*}
\begin{aligned}
\partial_x e(t,x)=&{\color{black}\kappa_x \alpha_t h_1-2\kappa_x\alpha_xh_2}+{\color{black}\kappa\left( \alpha_{tx}h_1+\alpha_t\partial_x h_1-2\alpha_{xx}h_2-2\alpha_x\partial_xh_2\right)}\\
=&{\color{black}Te_{\alpha}}+{\color{black}Te_{\Lambda}}+{\color{black}Te_{\phi}},\\
\partial_tp(t,x)=&{\color{black}\kappa_t \alpha_x h_1-2\kappa_t\alpha_t h_2}+{\color{black}\kappa\left( \alpha_{tx}h_1+\alpha_x\partial_t h_1-2\alpha_{tt}h_2-2\alpha_t\partial_th_2\right)}\\
=&{\color{black}Tp_{\alpha}}+{\color{black}Tp_{\Lambda}}+{\color{black}Tp_{\phi}}.
\end{aligned}
\end{equation*}
Where the terms, for example, $Te_{\alpha}, Te_{\Lambda},  Te_{\phi} $ represent the terms in $\partial_x e$ that are related with $\alpha, \Lambda$, and $\phi$ respectively,  as follows:
\begin{equation*}
\begin{aligned}
Te_{\alpha}&:=\partial_x \left( \kappa \partial_t \alpha \frac{(\partial_x \alpha)^2+(\partial_t \alpha)^2}{\alpha^2} -2\kappa\partial_x \alpha \frac{\partial_x\alpha \partial_t \alpha}{\alpha^2}\right)\\
Te_{\phi}&:=  \partial_x \left(4 \kappa \partial_t \alpha \sinh^2(\Lambda)\Big((\partial_t \phi)^2+(\partial_x\phi)^2\Big)-8\kappa\partial_x \alpha \partial_x \phi \partial_t \phi \sinh^2(\Lambda)\right)\\
Te_{\Lambda}&:= \partial_x \left( \kappa \partial_t \alpha ((\partial_t\Lambda)^2+(\partial_x\Lambda)^2) -2\kappa\partial_x \alpha \partial_x\Lambda \partial_t\Lambda \right).
\end{aligned}
\end{equation*}
 In the same way for  $Tp_{\alpha}, Tp_{\phi}$, $Tp_{\Lambda}$, this time respect to $\partial_t p$:
 \begin{equation*}
\begin{aligned}
Tp_{\alpha}&:=\partial_t \left( \kappa \partial_x \alpha \frac{(\partial_x \alpha)^2+(\partial_t \alpha)^2}{\alpha^2} -2\kappa\partial_t \alpha \frac{\partial_x\alpha \partial_t \alpha}{\alpha^2}\right)\\
Tp_{\phi}&:=  \partial_t \left(4 \kappa \partial_x \alpha \sinh^2(\Lambda)\Big((\partial_t \phi)^2+(\partial_x\phi)^2\Big)-8\kappa\partial_t \alpha \partial_x \phi \partial_t \phi \sinh^2(\Lambda)\right)\\
Tp_{\Lambda}&:= \partial_t \left( \kappa \partial_x\alpha ((\partial_t\Lambda)^2+(\partial_x\Lambda)^2) -2\kappa\partial_x \alpha \partial_t\Lambda \partial_t\Lambda \right).
\end{aligned}
\end{equation*}
Now, we are going to compute the sum of these two expressions, term by term, taking into account the structure of each term, starting by $Te_{\alpha}, Tp_{\alpha}$:
\begin{equation*}
\begin{aligned}
{\color{black}Tp_{\alpha}}=&{\color{black}\dfrac{\alpha_x\alpha_x^2\alpha_t}{(\alpha_x^2-\alpha_t^2)^2\alpha^2}}+{\color{black}\dfrac{\alpha_x\alpha_t\alpha_t^2}{(\alpha_x^2-\alpha_t^2)^2\alpha^2}}-\dfrac{2\alpha\alpha_x^2\alpha_x^2\alpha_{xt}}{(\alpha_x^2-\alpha_t^2)^2\alpha^2}-{\color{black}\dfrac{2\alpha\alpha_x^2\alpha_{xt}\alpha_t^2}{(\alpha_x^2-\alpha_t^2)^2\alpha^2}}+\dfrac{2\alpha\alpha_t\alpha_x\alpha_x^2\alpha_{tt}}{(\alpha_x^2-\alpha_t^2)^2\alpha^2}\\
&+{\color{black}\dfrac{2\alpha\alpha_t\alpha_x\alpha_t^2\alpha_{tt}}{(\alpha_x^2-\alpha_t^2)^2\alpha^2}} -{\color{black}\dfrac{2\alpha_t^2\alpha_x\alpha_t}{(\alpha_x^2-\alpha_t^2)\alpha^2}}+{\color{black}\dfrac{4\alpha\alpha_x^2\alpha_t^2\alpha_{xt}}{(\alpha_x^2-\alpha_t^2)^2\alpha^2}}-{\color{black}\dfrac{4\alpha\alpha_t\alpha_x\alpha_t^2\alpha_{tt}}{(\alpha_x^2-\alpha_t^2)^2\alpha^2}}+\dfrac{2\alpha\alpha_x^2\alpha_{xt}}{(\alpha_x^2-\alpha_t^2)\alpha^2}\\
&+\dfrac{2\alpha\alpha_x\alpha_t\alpha_{tt}}{(\alpha_x^2-\alpha_t^2)\alpha^2}-{\color{black}\dfrac{2\alpha_x\alpha_t\alpha_x^2}{(\alpha_x^2-\alpha_t^2)\alpha^2}} -{\color{black}\dfrac{2\alpha_t^2\alpha_x\alpha_t}{(\alpha_x^2-\alpha_t^2)\alpha^2}}-\dfrac{2\alpha_{tt}\alpha\alpha_x\alpha_t}{(\alpha_x^2-\alpha_t^2)\alpha^2}-\dfrac{2\alpha_t\alpha\alpha_{tt}\alpha_x}{(\alpha_x^2-\alpha_t^2)\alpha^2} \\
& -{\color{black}\dfrac{2\alpha_t^2\alpha\alpha_{xt}}{(\alpha_x^2-\alpha_t^2)\alpha^2}}+{\color{black}\dfrac{4\alpha_t\alpha_t^2\alpha_x}{(\alpha_x^2-\alpha_t^2)\alpha^2}}+{\color{black}\dfrac{\alpha\alpha_{tt}\alpha_x^2}{(\alpha_x^2-\alpha_t^2)\alpha^2}}+{\color{black}\dfrac{\alpha\alpha_{xt}\alpha_t^2}{(\alpha_x^2-\alpha_t^2)\alpha^2}}.
\end{aligned}
\end{equation*}
Arranging terms,
\begin{equation*}
\begin{aligned}
Tp_{\alpha}=& \dfrac{\alpha_x^2}{\alpha_x^2-\alpha_t^2} \left( \dfrac{\alpha\alpha_{xt}-\alpha_x\alpha_t}{\alpha^2}\right)-\dfrac{\alpha_t^2}{\alpha_x^2-\alpha_t^2}\left( \dfrac{\alpha\alpha_{xt}-\alpha_t\alpha_x}{\alpha^2}\right)+\dfrac{2\alpha\alpha_x^2\alpha_t^2\alpha_{xt}}{(\alpha_x^2-\alpha_t^2)^2\alpha^2}-\dfrac{2\alpha\alpha_t\alpha_x\alpha_{tt}\alpha_t^2}{(\alpha_x^2-\alpha_t^2)^2\alpha^2}\\
&-\dfrac{2\alpha\alpha_x^2\alpha_x^2\alpha_{xt}}{(\alpha_x^2-\alpha_t^2)^2\alpha^2}+\dfrac{2\alpha\alpha_t\alpha_x\alpha_x^2\alpha_{tt}}{(\alpha_x^2-\alpha_t^2)^2\alpha^2}+\dfrac{2\alpha\alpha_x^2\alpha_{xt}}{(\alpha_x^2-\alpha_t^2)\alpha^2}-\dfrac{2\alpha_t\alpha\alpha_{tt}\alpha_x}{(\alpha_x^2-\alpha_t^2)\alpha^2}\\
=& \dfrac{\alpha_x^2}{\alpha_x^2-\alpha_t^2}\partial_x\left( \dfrac{\alpha_t}{\alpha}\right)-\dfrac{\alpha_t^2}{\alpha_x^2-\alpha_t^2}\partial_x\left( \dfrac{\alpha_t}{\alpha}\right)-\dfrac{2\alpha_x^2\alpha\alpha_{xt}}{(\alpha_x^2-\alpha_t^2)\alpha^2}+\dfrac{2\alpha\alpha_t\alpha_x\alpha_{tt}}{(\alpha_x^2-\alpha_t^2)\alpha^2}\\
&+\dfrac{2\alpha_x^2\alpha\alpha_{xt}}{(\alpha_x^2-\alpha_t^2)\alpha^2}-\dfrac{2\alpha\alpha_t\alpha_x\alpha_{tt}}{(\alpha_x^2-\alpha_t^2)\alpha^2}
={\color{black}\partial_x\left(\dfrac{\alpha_t}{\alpha} \right)}.
\end{aligned}
\end{equation*}
Similarly,
\begin{equation*}
\begin{aligned}
{\color{black}Te_{\alpha}}= &\dfrac{\alpha_x^2}{\alpha_x^2-\alpha_t^2}\left( \dfrac{\alpha_x\alpha_t-\alpha \alpha_{tx}}{\alpha^2}\right)-\dfrac{\alpha_x\alpha_t\alpha_t^2}{(\alpha_x^2-\alpha_t^2)\alpha^2}+\dfrac{3\alpha_{tx}\alpha\alpha_t^2}{(\alpha_x^2-\alpha_t^2)\alpha^2}-{\color{black}\dfrac{2\alpha \alpha_x^2\alpha_t^2\alpha_{tx}}{(\alpha_x^2-\alpha_t^2)\alpha^2}}\\
&+{\color{black}\dfrac{2\alpha\alpha_x^2\alpha_x\alpha_t\alpha_{xx}}{(\alpha_x^2-\alpha_t^2)^2\alpha^2}}-{\color{black}\dfrac{2\alpha\alpha_x\alpha_{xx}\alpha_t\alpha_t^2}{(\alpha_x^2-\alpha_t^2)^2\alpha^2}}+{\color{black}\dfrac{2\alpha\alpha_t^2\alpha_{xt}\alpha_t^2}{(\alpha_x^2-\alpha_t^2)^2\alpha^2}}-\dfrac{2\alpha\alpha_{xx}\alpha_x\alpha_t}{s\alpha^2}\\
=&-\dfrac{\alpha_x^2}{\alpha_x^2-\alpha_t^2}\partial_x\left( \dfrac{\alpha_t}{\alpha}\right)-\dfrac{\alpha_x\alpha_t\alpha_t^2}{(\alpha_x^2-\alpha_t^2)\alpha^2}+\dfrac{3\alpha_{tx}\alpha\alpha_t^2}{(\alpha_x^2-\alpha_t^2)\alpha^2}-\dfrac{2\alpha_t^2\alpha_{tx}\alpha}{(\alpha_x^2-\alpha_t^2)\alpha^2}\\
&+\dfrac{2\alpha\alpha_{xx}\alpha_t\alpha_x}{(\alpha_x^2-\alpha_t^2)\alpha^2}-\dfrac{2\alpha\alpha_{xx}\alpha_x\alpha_t}{(\alpha_x^2-\alpha_t^2)\alpha^2}\\
=&-\dfrac{\alpha_x^2}{\alpha_x^2-\alpha_t^2}\partial_x\left( \dfrac{\alpha_t}{\alpha}\right)+\dfrac{\alpha_t^2}{\alpha_x^2-\alpha_t^2}\left( \dfrac{\alpha_{tx}\alpha-\alpha_x\alpha_t}{\alpha^2}\right)\\
=& -\dfrac{\alpha_x^2}{\alpha_x^2-\alpha_t^2}\partial_x\left( \dfrac{\alpha_t}{\alpha}\right)+\dfrac{\alpha_t^2}{\alpha_x^2-\alpha_t^2}\partial_x\left( \dfrac{\alpha_t}{\alpha}\right)\\
=&-\partial_x\left( \dfrac{\alpha_t}{\alpha}\right)\dfrac{\alpha_x^2-\alpha_t^2}{\alpha_x^2-\alpha_t^2}={\color{black}-\partial_x\left(  \dfrac{\alpha_t}{\alpha}\right)},
\end{aligned}
\end{equation*}
therefore $Te_{\alpha}+Tp_{\alpha}=0.$  We continue with the terms that depend mainly  on $\Lambda$ and its derivatives:
\begin{equation*}
\begin{aligned}
{\color{black}Te_{\Lambda}}=& {\color{black}\dfrac{\alpha_x\alpha_t\Lambda_t^2}{\alpha_x^2-\alpha_t^2}}+{\color{black}\dfrac{\alpha_x\alpha_t\Lambda_x^2}{\alpha_x^2-\alpha_t^2}}-{\color{black}\dfrac{2\alpha_x^2\Lambda_x\Lambda_t}{\alpha_x^2-\alpha_t^2}}-\dfrac{2\alpha\alpha_x\alpha_{xx}\alpha_t\Lambda_t^2}{(\alpha_x^2-\alpha_t^2)^2}-\dfrac{2\alpha\alpha_x\alpha_{xx}\alpha_t\Lambda_x^2}{(\alpha_x^2-\alpha_t^2)^2}+\dfrac{4\alpha\alpha_x^2\alpha_{xx}\Lambda_x\Lambda_t}{(\alpha_x^2-\alpha_t^2)^2}\\
&+\dfrac{\alpha\alpha_{xt}\Lambda_t^2}{\alpha_x^2-\alpha_t^2}+\dfrac{\alpha\alpha_{xt}\Lambda_x^2}{\alpha_x^2-\alpha_t^2}+{\color{black}\dfrac{2\alpha\alpha_t\Lambda_x\Lambda_{xx}}{\alpha_x^2-\alpha_t^2}}+\dfrac{2\alpha\alpha_t\Lambda_t\Lambda_{tx}}{\alpha_x^2-\alpha_t^2}-\dfrac{2\alpha\alpha_{xx}\Lambda_x\Lambda_t}{\alpha_x^2-\alpha_t^2}-{\color{black}\dfrac{2\alpha \alpha_x\Lambda_{xx}\Lambda_t}{\alpha_x^2-\alpha_t^2}}\\
&-\dfrac{2\alpha\alpha_x\Lambda_x\Lambda_{tx}}{\alpha_x^2-\alpha_t^2}+\dfrac{2\alpha\alpha_t^2\alpha_{tx}\Lambda_t^2}{(\alpha_x^2-\alpha_t^2)^2}+\dfrac{2\alpha\alpha_t^2\alpha_{tx}\Lambda_x^2}{(\alpha_x^2-\alpha_t^2)^2}-\dfrac{4\alpha\alpha_t\alpha_{tx}\alpha_x\Lambda_x\Lambda_t}{(\alpha_x^2-\alpha_t^2)^2},
\end{aligned}
\end{equation*}
and for $p$ we have:
\begin{equation*}
\begin{aligned}
{\color{black}Tp_{\Lambda}}=& {\color{black}\dfrac{\alpha_x\alpha_t\Lambda_t^2}{\alpha_x^2-\alpha_t^2}}+{\color{black}\dfrac{\alpha_x\alpha_t\Lambda_x^2}{\alpha_x^2-\alpha_t^2}}-{\color{black}\dfrac{2\alpha_x^2\Lambda_x\Lambda_t}{\alpha_x^2-\alpha_t^2}}-\dfrac{2\alpha\alpha_x^2\alpha_{xt}\alpha_t\Lambda_t^2}{(\alpha_x^2-\alpha_t^2)^2}-\dfrac{2\alpha\alpha_x^2\alpha_{xt}\alpha_t\Lambda_x^2}{(\alpha_x^2-\alpha_t^2)^2}+\dfrac{4\alpha\alpha_x\alpha_t\alpha_{xt}\Lambda_x\Lambda_t}{(\alpha_x^2-\alpha_t^2)^2}\\
&+\dfrac{\alpha\alpha_{xt}\Lambda_t^2}{\alpha_x^2-\alpha_t^2}+\dfrac{\alpha\alpha_{xt}\Lambda_x^2}{\alpha_x^2-\alpha_t^2}+{\color{black}\dfrac{2\alpha\alpha_x\Lambda_t\Lambda_{tt}}{\alpha_x^2-\alpha_t^2}}-\dfrac{2\alpha\alpha_t\Lambda_t\Lambda_{tx}}{\alpha_x^2-\alpha_t^2}-\dfrac{2\alpha\alpha_{xx}\Lambda_x\Lambda_t}{\alpha_x^2-\alpha_t^2}-{\color{black}\dfrac{2\alpha \alpha_t\Lambda_{tt}\Lambda_t}{\alpha_x^2-\alpha_t^2}}\\
&+\dfrac{2\alpha\alpha_x\Lambda_x\Lambda_{tx}}{\alpha_x^2-\alpha_t^2}+\dfrac{2\alpha\alpha_t\alpha_x\alpha_{tt}\Lambda_t^2}{(\alpha_x^2-\alpha_t^2)^2}+\dfrac{2\alpha\alpha_t\alpha_x\alpha_{tt}\Lambda_x^2}{(\alpha_x^2-\alpha_t^2)^2}-\dfrac{4\alpha\alpha_t^2\alpha_{tt}\Lambda_x\Lambda_t}{(\alpha_x^2-\alpha_t^2)^2}.
\end{aligned}
\end{equation*}
If we sum these two terms and using the first equation in the \eqref{sistema1}, we get 
\begin{equation*}
\begin{aligned}
&Te_{\Lambda}+Tp_{\Lambda}=\\
&= \frac{2(\alpha_x\Lambda_t-\alpha_t\Lambda_x)[\alpha\Lambda_{tt}-\alpha\Lambda_{xx}]}{\alpha_x^2-\alpha_t^2}+\frac{2\alpha_t\Lambda_t(\alpha_x\Lambda_t-\alpha_t\Lambda_x)}{\alpha_x^2-\alpha_t^2}- \frac{2\alpha_x\Lambda_x(\alpha_x\Lambda_t-\alpha_t\Lambda_x)}{\alpha_x^2-\alpha_t^2}+{\color{black}R}\\
&=\dfrac{2}{\alpha_x^2-\alpha_t^2}\left(\alpha_x\Lambda_t-\alpha_t\Lambda_x\right)\left[\partial_t(\alpha \Lambda_t)-\partial_x(\alpha\Lambda_x) \right]+ {\color{black}R}\\
&= \dfrac{2}{\alpha_x^2-\alpha_t^2}\left(\alpha_x\Lambda_t-\alpha_t\Lambda_x\right)\left[ 2\alpha\phi_t^2\sinh(2\Lambda)-2\alpha\phi_x^2\sinh(2\Lambda) \right] +{\color{black}R},
\end{aligned}
\end{equation*}
where ${\color{black}R}$ represents the remainder of the terms in the sum above. After simplification, we have that $R$ is actually equal to zero, indeed
\begin{equation*}
\begin{aligned}
R=& ~{}\dfrac{2\alpha_{xt}\alpha\Lambda_t^2}{(\alpha_x^2-\alpha_t^2)^2}(\alpha_t^2-\alpha_x^2)+\dfrac{2\alpha_{xt}\alpha\Lambda_x^2}{(\alpha_x^2-\alpha_t^2)^2}(\alpha_t^2-\alpha_x^2)+\dfrac{4\alpha_{tt}\alpha\Lambda_t\Lambda_x}{(\alpha_x^2-\alpha_t^2)^2}(\alpha_x^2-\alpha_t^2)+\dfrac{2\alpha_{xt}\alpha\Lambda_t^2}{\alpha_x^2-\alpha_t^2}\\
&+\dfrac{2\alpha_{xt}\alpha\Lambda_x^2}{(\alpha_x^2-\alpha_t^2)^2}-\dfrac{4\alpha_{tt}\alpha\Lambda_t\Lambda_x}{\alpha_x^2-\alpha_t^2}=0.
\end{aligned}
\end{equation*}
The last terms to simplify are the terms that depend mainly on $\phi$, first, let us start with the terms related to $\phi$ in momentum density derivatives:
\begin{equation*}
\begin{aligned}
{\color{black}Tp_{\phi}}=&~{} \dfrac{4\alpha_t\alpha_x\sinh^2(\Lambda)\phi_t^2}{\alpha_x^2-\alpha_t^2}+\dfrac{4\alpha_t\alpha_x\sinh^2(\Lambda)\phi_x^2}{\alpha_x^2-\alpha_t^2}-\dfrac{8\alpha_t^2\sinh^2(\Lambda)\phi_x\phi_t}{\alpha_x^2-\alpha_t^2}-\dfrac{8\alpha\alpha_x^2\alpha_{tx}\sinh^2(\Lambda)\phi_x^2}{(\alpha_x^2-\alpha_t^2)^2}\\
&-\dfrac{8\alpha\alpha_x^2\alpha_{tx}\sinh^2(\Lambda)\phi_t^2}{(\alpha_x^2-\alpha_t^2)^2} +\dfrac{16\alpha\alpha_x\alpha_{xt}\alpha_t\sinh^2(\Lambda)\phi_x\phi_t}{(\alpha_x^2-\alpha_t^2)^2}+\dfrac{8\alpha\alpha_t\alpha_{tt}\alpha_x\sinh^2(\Lambda)\phi_t^2}{(\alpha_x^2-\alpha_t^2)^2}\\
&+\dfrac{8\alpha\alpha_t\alpha_{tt}\alpha_x\sinh^2(\Lambda)\phi_x^2}{(\alpha_x^2-\alpha_t^2)^2}-\dfrac{16\alpha\alpha_t^2\alpha_{tt}\sinh^2(\Lambda)\phi_t\phi_x}{(\alpha_x^2-\alpha_t^2)^2}+\dfrac{4\alpha\alpha_{xt}\sinh^2(\Lambda)\phi_t^2}{\alpha_x^2-\alpha_t^2}+\dfrac{4\alpha\alpha_{xt}\sinh^2(\Lambda)\phi_x^2}{\alpha_x^2-\alpha_t^2}\\
& +\dfrac{4\alpha\alpha_x\sinh(2\Lambda)\Lambda_t\phi_x^2}{\alpha_x^2-\alpha_t^2}+\dfrac{4\alpha\alpha_x\sinh(2\Lambda)\Lambda_t\phi_t^2}{\alpha_x^2-\alpha_t^2}+\dfrac{8\alpha\alpha_x\sinh^2(\Lambda)\phi_x\phi_{xt}}{\alpha_x^2-\alpha_t^2}+\dfrac{8\alpha\alpha_x\sinh^2(\Lambda)\phi_x\phi_{tt}}{\alpha_x^2-\alpha_t^2}\\
&-\dfrac{8\alpha\alpha_{tt}\sinh^2(\Lambda)\phi_x\phi_t}{\alpha_x^2-\alpha_t^2}-\dfrac{8\alpha\alpha_t\sinh(2\Lambda)\Lambda_t\phi_x\phi_t}{\alpha_x^2-\alpha_t^2}-\dfrac{8\alpha\alpha_t\sinh^2(\Lambda)\phi_{xt}\phi_t}{\alpha_x^2-\alpha_t^2}-\dfrac{8\alpha\alpha_t\sinh^2(\Lambda)\phi_{tt}\phi_x}{\alpha_x^2-\alpha_t^2}
\end{aligned}
\end{equation*}
Now, let us pass to the terms in the derivative of the energy density
\begin{equation*}
\begin{aligned}
{\color{black}Te_{\phi}}=& \dfrac{4\alpha_t\alpha_x\sinh^2(\Lambda)\phi_t^2}{\alpha_x^2-\alpha_t^2}+\dfrac{4\alpha_t\alpha_x\sinh^2(\Lambda)\phi_x^2}{\alpha_x^2-\alpha_t^2}-\dfrac{8\alpha_x^2\sinh^2(\Lambda)\phi_x\phi_t}{\alpha_x^2-\alpha_t^2}-\dfrac{8\alpha\alpha_x\alpha_t\alpha_{xx}\sinh^2(\Lambda)\phi_t^2}{(\alpha_x^2-\alpha_t^2)^2}\\
&-\dfrac{8\alpha\alpha_x\alpha_t\alpha_{xx}\sinh^2(\Lambda)\phi_x^2}{(\alpha_x^2-\alpha_t^2)^2} +\dfrac{16\alpha\alpha_x^2\alpha_{xx}\sinh^2(\Lambda)\phi_x\phi_t}{(\alpha_x^2-\alpha_t^2)^2}+\dfrac{8\alpha\alpha_t^2\alpha_{tx}\sinh^2(\Lambda)\phi_t^2}{(\alpha_x^2-\alpha_t^2)^2}\\
&+\dfrac{8\alpha\alpha_t^2\alpha_{tx}\sinh^2(\Lambda)\phi_x^2}{(\alpha_x^2-\alpha_t^2)^2}-\dfrac{16\alpha\alpha_t\alpha_{tx}\alpha_x\sinh^2(\Lambda)\phi_t\phi_x}{(\alpha_x^2-\alpha_t^2)^2}+\dfrac{4\alpha\alpha_{xt}\sinh^2(\Lambda)\phi_t^2}{\alpha_x^2-\alpha_t^2}+\dfrac{4\alpha\alpha_{xt}\sinh^2(\Lambda)\phi_x^2}{\alpha_x^2-\alpha_t^2}\\
& +\dfrac{4\alpha\alpha_t\sinh(2\Lambda)\Lambda_x\phi_x^2}{\alpha_x^2-\alpha_t^2}+\dfrac{4\alpha\alpha_t\sinh(2\Lambda)\Lambda_x\phi_t^2}{\alpha_x^2-\alpha_t^2}+\dfrac{8\alpha\alpha_t\sinh^2(\Lambda)\phi_x\phi_{xx}}{\alpha_x^2-\alpha_t^2}+\dfrac{8\alpha\alpha_t\sinh^2(\Lambda)\phi_t\phi_{tx}}{\alpha_x^2-\alpha_t^2}\\
&-\dfrac{8\alpha\alpha_{tt}\sinh^2(\Lambda)\phi_x\phi_t}{\alpha_x^2-\alpha_t^2}-\dfrac{8\alpha\alpha_x\sinh(2\Lambda)\Lambda_x\phi_x\phi_t}{\alpha_x^2-\alpha_t^2}-\dfrac{8\alpha\alpha_{x}\sinh^2(\Lambda)\phi_{xx}\phi_t}{\alpha_x^2-\alpha_t^2}-\dfrac{8\alpha\alpha_x\sinh^2(\Lambda)\phi_{x}\phi_{tx}}{\alpha_x^2-\alpha_t^2}.
\end{aligned}
\end{equation*}
In the next step we will sum $Te_{\phi}+Tp_{\phi}+Te_{\Lambda}+Tp_{\Lambda}$, then, simplify the similar terms and cancel the corresponding terms, then we obtain the following expression:
\begin{equation*}
\begin{aligned}
& Te_{\phi}+Tp_{\phi}+Te_{\Lambda}+Tp_{\Lambda}\\
& =\dfrac{8\alpha_x\alpha_t\sinh^2(\Lambda)\phi_t^2}{\alpha_x^2-\alpha_t^2}+\dfrac{8\alpha_x\alpha_t\sinh^2(\Lambda)\phi_x^2}{\alpha_x^2-\alpha_t^2}-\dfrac{8\alpha_x^2\sinh^2(\Lambda)\phi_t\phi_x}{\alpha_x^2-\alpha_t^2}-\dfrac{8\alpha_t^2\sinh^2(\Lambda)\phi_t\phi_x}{\alpha_x^2-\alpha_t^2}\\
&+\dfrac{8\alpha\alpha_{tx}\sinh^2(\Lambda)\phi_t^2}{(\alpha_x^2-\alpha_t^2)^2}(\alpha_t^2-\alpha_x^2)+\dfrac{8\alpha\alpha_{tx}\sinh^2(\Lambda)\phi_x^2}{(\alpha_x^2-\alpha_t^2)^2}(\alpha_t^2-\alpha_x^2)\\
&+\dfrac{16\alpha\alpha_{xx}\sinh^2(\Lambda)\phi_t\phi_x}{(\alpha_x^2-\alpha_t^2)^2}(\alpha_x^2-\alpha_t^2)+\dfrac{8\alpha\alpha_{tx}\sinh^2(\Lambda)\phi_t^2}{\alpha_x^2-\alpha_t^2}+\dfrac{8\alpha\alpha_{tx}\sinh^2(\Lambda)\phi_x^2}{\alpha_x^2-\alpha_t^2}\\
&+\dfrac{4\alpha\alpha_t\sinh(2\Lambda)\Lambda_x\phi_x^2}{\alpha_x^2-\alpha_t^2}+\dfrac{4\alpha\alpha_t\sinh(2\Lambda)\Lambda_x\phi_t^2}{\alpha_x^2-\alpha_t^2}+\dfrac{4\alpha\alpha_x\sinh(2\Lambda)\Lambda_t\phi_x^2}{\alpha_x^2-\alpha_t^2}\\
&+\dfrac{4\alpha\alpha_x\sinh(2\Lambda)\Lambda_t\phi_t^2}{\alpha_x^2-\alpha_t^2}-\dfrac{16\alpha\alpha_{tt}\sinh^2(\Lambda)\phi_x\phi_t}{\alpha_x^2-\alpha_t^2}-\dfrac{8\alpha\alpha_x\sinh(2\Lambda)\Lambda_x\phi_x\phi_t}{\alpha_x^2-\alpha_t^2}\\
&-\dfrac{8\alpha\alpha_x\sinh^2(\Lambda)\phi_{xx}\phi_t}{\alpha_x^2-\alpha_t^2}+\dfrac{8\alpha\alpha_x\sinh^2(\Lambda)\phi_t\phi_{tt}}{\alpha_x^2-\alpha_t^2}-\dfrac{8\alpha\alpha_t\sinh(2\Lambda)\Lambda_t\phi_x\phi_t}{\alpha_x^2-\alpha_t^2}\\
&-\dfrac{8\alpha\alpha_t\sinh^2(\Lambda)\phi_x\phi_{tt}}{\alpha_x^2-\alpha_t^2}+\dfrac{4\alpha\alpha_x\sinh(2\Lambda)\Lambda_t\phi_t^2}{\alpha_x^2-\alpha_t^2}-\dfrac{4\alpha\alpha_x\sinh(2\Lambda)\Lambda_t\phi_x^2}{\alpha_x^2-\alpha_t^2}\\
&-\dfrac{4\alpha\alpha_t\sinh(2\Lambda)\Lambda_x\phi_t^2}{\alpha_x^2-\alpha_t^2}+\dfrac{4\alpha\alpha_t\sinh(2\Lambda)\Lambda_x\phi_x^2}{\alpha_x^2-\alpha_t^2}+\dfrac{8\alpha\alpha_t\sinh^2(\Lambda)\phi_x\phi_{xx}}{\alpha_x^2-\alpha_t^2}.
\end{aligned}
\end{equation*}
In this expression we have several terms that will cancel out. They can be gathered in such a way that we can use the second equation in the system \eqref{sistema1}. We have
\begin{equation*}
\begin{aligned}
Te_{\phi}+Tp_{\phi}+Te_{\Lambda}+Tp_{\Lambda}=&\sinh^2(\Lambda)(\alpha_t\phi_t-\alpha_x\phi_x+\alpha\phi_{tt}-\alpha\phi_{tt})\left( \frac{8\alpha_x\phi_t-8\alpha_t\phi_x}{\alpha_x^2-\alpha_t^2}\right)\\
&+\sinh(2\Lambda)(\alpha\phi_t\Lambda_t-\alpha\phi_x\Lambda_x)\left( \frac{8\alpha_x\phi_t-8\alpha_t\phi_x}{\alpha_x^2-\alpha_t^2}\right)\\
=&(\partial_t(\alpha \sinh^2 \Lambda \partial_t\phi )- \partial_x(\alpha \sinh^2 \Lambda \partial_x\phi ))\left( \frac{8\alpha_x\phi_t-8\alpha_t\phi_x}{\alpha_x^2-\alpha_t^2}\right)=0.
\end{aligned}
\end{equation*}
We conclude that:
\[
\partial_t p(t,x)+\partial_x e(t,x)=0.
\]
 In the second part of the proof, we are going to show the second equation in \eqref{eqContinuidad}. For this, we will again use the notation for grouping terms in the following way 
\begin{equation*}
\begin{aligned}
\partial_t e(t,x)=&~{}{\color{black}\gamma_t \alpha_t h_1-2\gamma_t\alpha_xh_2}+{\color{black}\gamma\left( \alpha_{tt}h_1+\alpha_t\partial_t h_1-2\alpha_{xt}h_2-2\alpha_x\partial_t h_2\right)}\\
=&~{} {\color{black}Te_{\alpha}}+{\color{black}Te_{\Lambda}}+{\color{black}Te_{\phi}},\\
\partial_x p(t,x)=& \partial_x \gamma \left(\alpha_x h_1-2\alpha_t h_2\right)+\gamma\left(\alpha_{xx}h_1+\alpha_x\partial_x h_1-2\alpha_{tx}h_2-2\alpha_t\partial_xh_2\right)\\
=&~{} {\color{black}Tp_{\alpha}}+{\color{black}Tp_{\Lambda}}+{\color{black}Tp_{\phi}},
\end{aligned}
\end{equation*}
where the terms $Te_{\alpha}, Te_{\phi}, Te_{\Lambda}$ has the same form than in the before, but, this time respect to the terms for $\partial_t e$ (or $\partial_x p$ respectively) Let us simplify each term, starting with the terms in $\alpha$:
\begin{equation*}
\begin{aligned}
{\color{black}Te_{\alpha}}=&-\dfrac{\alpha \alpha_x^2\alpha_{tt}}{(\alpha_x^2-\alpha_t^2)\alpha^2}+\dfrac{2\alpha_x^2\alpha_t^2}{(\alpha_x^2-\alpha_t^2)\alpha^2}+{\color{black}\dfrac{3\alpha\alpha_t^2\alpha_{tt}}{(\alpha_x^2-\alpha_t^2)\alpha^2}}-\dfrac{2\alpha \alpha_{tx}\alpha_x\alpha_t}{(\alpha_x^2-\alpha_t^2)\alpha^2}-{\color{black}\dfrac{2\alpha_t^2\alpha_t^2}{(\alpha_x^2-\alpha_t^2)\alpha^2}}\\
&-{\color{black}\dfrac{\alpha_t^2\alpha_x^2}{(\alpha_x^2-\alpha_t^2)\alpha^2}}+{\color{black}\dfrac{\alpha_t^2\alpha_t^2}{(\alpha_x^2-\alpha_t^2)\alpha^2}}-{\color{black}\frac{2\alpha\alpha_x\alpha_t\alpha_{xt}\alpha_x^2}{(\alpha_x^2-\alpha_t^2)^2\alpha^2}}-{\color{black}\dfrac{2\alpha \alpha_x\alpha_{xt}\alpha_t^2\alpha_t}{(\alpha_x^2-\alpha_t^2)^2\alpha^2}}+{\color{black}\dfrac{2\alpha\alpha_t^2\alpha_{tt}\alpha_x^2}{(\alpha_x^2-\alpha_t^2)^2\alpha^2}}\\
&+{\color{black}\dfrac{2\alpha\alpha_t^2\alpha_{tt}\alpha_t^2}{(\alpha_x^2-\alpha_t^2)^2\alpha^2}}
+{\color{black} \frac{4\alpha\alpha_x\alpha_x^2\alpha_{xt}\alpha_t}{(\alpha_x^2-\alpha_t^2)^2\alpha^2}}-{\color{black}\dfrac{4\alpha\alpha_t^2\alpha_{tt}\alpha_x^2}{(\alpha_x^2-\alpha_t^2)^2\alpha^2}}\\
=&-\dfrac{\alpha \alpha_x^2\alpha_{tt}}{(\alpha_x^2-\alpha_t^2)\alpha^2}+\dfrac{2\alpha_x^2\alpha_t^2}{(\alpha_x^2-\alpha_t^2)\alpha^2}+{\color{black}\dfrac{\alpha_t^2}{(\alpha_x^2-\alpha_t^2)}\partial_x\left(\dfrac{\alpha_x}{\alpha}\right)}+{\color{black}\dfrac{2\alpha\alpha_t^2\alpha_{tt}}{(\alpha_x^2-\alpha_t^2)\alpha^2}}-{\color{black}\dfrac{\alpha_t^2\alpha_t^2}{(\alpha_x^2-\alpha_t^2)\alpha^2}}\\
&-\dfrac{2\alpha \alpha_{tx}\alpha_x\alpha_t}{(\alpha_x^2-\alpha_t^2)\alpha^2}-{\color{black}\frac{2\alpha \alpha_t^2\alpha_{tt}}{(\alpha_x^2-\alpha_t^2)\alpha^2}}+{\color{black}\dfrac{2\alpha\alpha_x\alpha_{xt}\alpha_t}{(\alpha_x^2-\alpha_t^2)\alpha^2}}\\
=&-\dfrac{\alpha_t^2}{(\alpha_x^2-\alpha_t^2)}\partial_x\left(\frac{\alpha_x}{\alpha}\right)-\frac{\alpha_{tt}}{\alpha}+\frac{\alpha_t^2}{(\alpha_x^2-\alpha_t^2)}\partial_t\left(\dfrac{\alpha_t}{\alpha}\right)=-\partial_t\left(\dfrac{\alpha_t}{\alpha}\right),
\end{aligned}
\end{equation*}
for another hand, using a similar simplification as above, we obtain for $Tp_{\alpha}$ the following expression:
\begin{equation*}
\begin{aligned}
{\color{black}Tp_{\alpha}}=&\frac{\alpha_x^2\alpha_x^2}{(\alpha_x^2-\alpha_t^2)\alpha^2}+\frac{\alpha_x^2\alpha_t^2}{(\alpha_x^2-\alpha_t^2)\alpha^2}-\frac{2\alpha_x^2\alpha_t^2}{(\alpha_x^2-\alpha_t^2)\alpha^2}-\frac{2\alpha\alpha_x^2\alpha_x^2\alpha_{xx}}{(\alpha_x^2-\alpha_t^2)^2\alpha^2}-\frac{2\alpha\alpha_x^2\alpha_t^2\alpha_{xx}}{(\alpha_x^2-\alpha_t^2)^2\alpha^2}\\
&+\frac{4\alpha\alpha_x^2\alpha_t^2\alpha_{xx}}{(\alpha_x^2-\alpha_t^2)^2\alpha^2}+\frac{2\alpha\alpha_t\alpha_x\alpha_{xt}\alpha_x^2}{(\alpha_x^2-\alpha_t^2)^2\alpha^2}+ \frac{2\alpha\alpha_t\alpha_x\alpha_{xt}\alpha_t^2}{(\alpha_x^2-\alpha_t^2)^2\alpha^2}-\frac{4\alpha\alpha_t\alpha_x\alpha_{xt}\alpha_t^2}{(\alpha_x^2-\alpha_t^2)^2\alpha^2}+\frac{\alpha\alpha_{xx}\alpha_x^2}{(\alpha_x^2-\alpha_t^2)\alpha^2}\\
&+\frac{\alpha\alpha_{xx}\alpha_t^2}{(\alpha_x^2-\alpha_t^2)\alpha^2}+\frac{2\alpha\alpha_{xx}\alpha_x^2}{(\alpha_x^2-\alpha_t^2)\alpha^2} 
+\frac{2\alpha\alpha_{tx}\alpha_x\alpha_t}{(\alpha_x^2-\alpha_t^2)\alpha^2}-\frac{2\alpha_x^2\alpha_x^2}{(\alpha_x^2-\alpha_t^2)\alpha^2}-\frac{2\alpha_x^2\alpha_t^2}{(\alpha_x^2-\alpha_t^2)\alpha^2}\\
&-\frac{2\alpha\alpha_x\alpha_t\alpha_{tx}}{(\alpha_x^2-\alpha_t^2)\alpha^2}
-\frac{2\alpha\alpha_t^2\alpha_{xx}}{(\alpha_x^2-\alpha_t^2)\alpha^2}-\frac{2\alpha\alpha_t\alpha_x\alpha_{tx}}{(\alpha_x^2-\alpha_t^2)\alpha^2}+\frac{4\alpha_t^2\alpha_x^2}{(\alpha_x^2-\alpha_t^2)\alpha^2}\\
=&~{}\partial_x\left( \dfrac{\alpha_x}{\alpha}\right). 
\end{aligned}
\end{equation*}
Now, for the terms in $Te_{\Lambda}$ and $Tp_{\Lambda}$ we have
 \begin{equation*}
\begin{aligned}
{\color{black}Te_{\Lambda}}=&~{} \dfrac{\alpha_t^2\Lambda_t^2}{\alpha_x^2-\alpha_t^2}+\dfrac{\alpha_t^2\Lambda_x^2}{\alpha_x^2-\alpha_t^2}-\dfrac{2\alpha_x\alpha_t\Lambda_x\Lambda_t}{\alpha_x^2-\alpha_t^2}-\dfrac{2\alpha\alpha_x\alpha_{tx}\alpha_t\Lambda_t^2}{(\alpha_x^2-\alpha_t^2)^2}-\dfrac{2\alpha\alpha_x\alpha_{tx}\alpha_t\Lambda_x^2}{(\alpha_x^2-\alpha_t^2)^2}+\dfrac{4\alpha\alpha_x^2\alpha_{xt}\Lambda_x\Lambda_t}{(\alpha_x^2-\alpha_t^2)^2}\\
&+\dfrac{2\alpha\alpha_{tt}\alpha_t^2\Lambda_t^2}{\alpha_x^2-\alpha_t^2}+\dfrac{2\alpha\alpha_t^2\alpha_{tt}\Lambda_x^2}{\alpha_x^2-\alpha_t^2}-\dfrac{4\alpha\alpha_t\alpha_x\alpha_{tt}\Lambda_x\Lambda_{t}}{\alpha_x^2-\alpha_t^2}+\dfrac{\alpha\alpha_{tt}\Lambda_t^2}{\alpha_x^2-\alpha_t^2}+\dfrac{\alpha\alpha_{xx}\Lambda_x^2}{\alpha_x^2-\alpha_t^2}-\dfrac{2\alpha \alpha_x\Lambda_x\Lambda_{tt}}{\alpha_x^2-\alpha_t^2}\\
&+\dfrac{2\alpha\alpha_x\alpha_t\Lambda_x\Lambda_{tx}}{\alpha_x^2-\alpha_t^2}+\dfrac{2\alpha\alpha_t\Lambda_t\Lambda_{tt}}{(\alpha_x^2-\alpha_t^2)^2}-\dfrac{2\alpha\alpha_{tx}\Lambda_x\Lambda_t}{(\alpha_x^2-\alpha_t^2)^2}-\dfrac{2\alpha\alpha_x\Lambda_{xt}\Lambda_t}{(\alpha_x^2-\alpha_t^2)^2},
\end{aligned}
\end{equation*}
and, for $p$ we have:
\begin{equation*}
\begin{aligned}
{\color{black}Tp_{\Lambda}}=&~{} \dfrac{\alpha_x^2\Lambda_t^2}{\alpha_x^2-\alpha_t^2}+\dfrac{\alpha_x^2\Lambda_x^2}{\alpha_x^2-\alpha_t^2}-\dfrac{2\alpha_x\alpha_t\Lambda_x\Lambda_t}{\alpha_x^2-\alpha_t^2}-\dfrac{2\alpha\alpha_x^2\alpha_{xx}\Lambda_t^2}{(\alpha_x^2-\alpha_t^2)^2}-\dfrac{2\alpha\alpha_x^2\alpha_{xx}\Lambda_x^2}{(\alpha_x^2-\alpha_t^2)^2}+\dfrac{4\alpha\alpha_x\alpha_t\alpha_{xx}\Lambda_x\Lambda_t}{(\alpha_x^2-\alpha_t^2)^2}\\
&+\dfrac{2\alpha\alpha_{xt}\alpha_x\alpha_t\Lambda_t^2}{\alpha_x^2-\alpha_t^2}+\dfrac{2\alpha\alpha_{xt}\alpha_x\alpha_t\Lambda_x^2}{\alpha_x^2-\alpha_t^2}-\dfrac{4\alpha\alpha_t^2\alpha_{xt}\Lambda_t\Lambda_{x}}{\alpha_x^2-\alpha_t^2}-\dfrac{\alpha\alpha_{xx}\Lambda_t^2}{\alpha_x^2-\alpha_t^2}+\dfrac{\alpha\alpha_{xx}\Lambda_x^2}{\alpha_x^2-\alpha_t^2}+\dfrac{2\alpha \alpha_x\Lambda_{xx}\Lambda_x}{\alpha_x^2-\alpha_t^2}\\
&+\dfrac{2\alpha\alpha_x\Lambda_t\Lambda_{tx}}{\alpha_x^2-\alpha_t^2}-\dfrac{2\alpha\alpha_{xt}\Lambda_t\Lambda_x}{(\alpha_x^2-\alpha_t^2)^2}-\dfrac{2\alpha\alpha_t\alpha_x\Lambda_t\Lambda_{xx}}{(\alpha_x^2-\alpha_t^2)^2}-\dfrac{2\alpha\alpha_t\Lambda_x\Lambda_{tx}}{(\alpha_x^2-\alpha_t^2)^2}.
\end{aligned}
\end{equation*}
If we sum up the terms and using the first equation in the system \eqref{sistema1} and after simplification we obtain
\begin{equation*}
\begin{aligned}
Te_{\Lambda}+Tp_{\Lambda}=&   \left(\dfrac{2\alpha_t\Lambda_t-2\alpha_x\Lambda_x}{\alpha_x^2-\alpha_t^2} \right)(\alpha_t\Lambda_t-\alpha_x\Lambda-\alpha\Lambda_{xx}+\alpha\Lambda_{tt})\\
&+\dfrac{(\alpha_t^2-\alpha_x^2)\Lambda_x^2}{\alpha_x^2-\alpha_t^2}+\dfrac{(\alpha_x^2-\alpha_t^2)\Lambda_t^2}{\alpha_x^2-\alpha_t^2}\\
 =& \left(\frac{2\alpha_t\Lambda_t-2\alpha_x\Lambda_x}{\alpha_x^2-\alpha_t^2} \right)\left( \alpha_t\Lambda_t+\alpha\Lambda_{tt}-\alpha_x\Lambda_x-\alpha\Lambda_{xx}\right)+\Lambda_t^2-\Lambda_x^2\\
= &  \left(\frac{2\alpha_t\Lambda_t-2\alpha_x\Lambda_x}{\alpha_x^2-\alpha_t^2} \right)\left( 2\alpha\phi_t^2\sinh(2\Lambda)-2\alpha\phi_x^2\sinh(2\Lambda)\right)+\Lambda_t^2-\Lambda_x^2.
\end{aligned}
\end{equation*}
To conclude the result, let us simplify the terms in $\phi$:
\begin{equation*}
\begin{aligned}
{\color{black}Tp_{\phi}}=&~{} \dfrac{4\alpha_x^2\sinh^2(\Lambda)\phi_t^2}{\alpha_x^2-\alpha_t^2}+\dfrac{4\alpha_x^2\sinh^2(\Lambda)\phi_x^2}{\alpha_x^2-\alpha_t^2}-\dfrac{8\alpha_t\alpha_x\sinh^2(\Lambda)\phi_x\phi_t}{\alpha_x^2-\alpha_t^2}-\dfrac{8\alpha\alpha_x^2\alpha_{xx}\sinh^2(\Lambda)\phi_t^2}{(\alpha_x^2-\alpha_t^2)^2}\\
&-\dfrac{8\alpha\alpha_x^2\alpha_{xx}\sinh^2(\Lambda)\phi_x^2}{(\alpha_x^2-\alpha_t^2)^2} +\dfrac{16\alpha\alpha_x\alpha_{xx}\alpha_t\sinh^2(\Lambda)\phi_x\phi_t}{(\alpha_x^2-\alpha_t^2)^2}+\dfrac{8\alpha\alpha_t\alpha_{tx}\alpha_x\sinh^2(\Lambda)\phi_t^2}{(\alpha_x^2-\alpha_t^2)^2}\\
&+\dfrac{8\alpha\alpha_t\alpha_{tx}\alpha_x\sinh^2(\Lambda)\phi_x^2}{(\alpha_x^2-\alpha_t^2)^2}-\dfrac{16\alpha\alpha_t^2\alpha_{tx}\sinh^2(\Lambda)\phi_t\phi_x}{(\alpha_x^2-\alpha_t^2)^2}+\dfrac{4\alpha\alpha_{xx}\sinh^2(\Lambda)\phi_t^2}{\alpha_x^2-\alpha_t^2}+\dfrac{4\alpha\alpha_{xx}\sinh^2(\Lambda)\phi_x^2}{\alpha_x^2-\alpha_t^2}\\
& +\dfrac{4\alpha\alpha_x\sinh(2\Lambda)\Lambda_x\phi_x^2}{\alpha_x^2-\alpha_t^2}+\dfrac{4\alpha\alpha_x\sinh(2\Lambda)\Lambda_x\phi_t^2}{\alpha_x^2-\alpha_t^2}+\dfrac{8\alpha\alpha_x\sinh^2(\Lambda)\phi_x\phi_{xx}}{\alpha_x^2-\alpha_t^2}+\dfrac{8\alpha\alpha_x\sinh^2(\Lambda)\phi_t\phi_{tt}}{\alpha_x^2-\alpha_t^2}\\
&-\dfrac{8\alpha\alpha_{tx}\sinh^2(\Lambda)\phi_x\phi_t}{\alpha_x^2-\alpha_t^2}-\dfrac{8\alpha\alpha_t\sinh(2\Lambda)\Lambda_x\phi_x\phi_t}{\alpha_x^2-\alpha_t^2}-\dfrac{8\alpha\alpha_t\sinh^2(\Lambda)\phi_{xx}\phi_t}{\alpha_x^2-\alpha_t^2}-\dfrac{8\alpha\alpha_t\sinh^2(\Lambda)\phi_{tx}\phi_x}{\alpha_x^2-\alpha_t^2}.
\end{aligned}
\end{equation*}
Now, let us pass to the terms in the derivative of the energy density
\begin{equation*}
\begin{aligned}
{\color{black}Te_{\phi}}=&~{} \dfrac{4\alpha_t^2\sinh^2(\Lambda)\phi_t^2}{\alpha_x^2-\alpha_t^2}+\dfrac{4\alpha_t^2\sinh^2(\Lambda)\phi_x^2}{\alpha_x^2-\alpha_t^2}-\dfrac{8\alpha_x\alpha_t\sinh^2(\Lambda)\phi_x\phi_t}{\alpha_x^2-\alpha_t^2}-\dfrac{8\alpha\alpha_x\alpha_t\alpha_{xt}\sinh^2(\Lambda)\phi_t^2}{(\alpha_x^2-\alpha_t^2)^2}\\
&-\dfrac{8\alpha\alpha_x\alpha_t\alpha_{xt}\sinh^2(\Lambda)\phi_x^2}{(\alpha_x^2-\alpha_t^2)^2} +\dfrac{16\alpha\alpha_x^2\alpha_{xt}\sinh^2(\Lambda)\phi_x\phi_t}{(\alpha_x^2-\alpha_t^2)^2}+\dfrac{8\alpha\alpha_t^2\alpha_{tt}\sinh^2(\Lambda)\phi_t^2}{(\alpha_x^2-\alpha_t^2)^2}\\
&+\dfrac{8\alpha\alpha_t^2\alpha_{xx}\sinh^2(\Lambda)\phi_x^2}{(\alpha_x^2-\alpha_t^2)^2}-\dfrac{16\alpha\alpha_t\alpha_{tt}\alpha_x\sinh^2(\Lambda)\phi_t\phi_x}{(\alpha_x^2-\alpha_t^2)^2}+\dfrac{4\alpha\alpha_{tt}\sinh^2(\Lambda)\phi_t^2}{\alpha_x^2-\alpha_t^2}+\dfrac{4\alpha\alpha_{tt}\sinh^2(\Lambda)\phi_x^2}{\alpha_x^2-\alpha_t^2}\\
& +\dfrac{4\alpha\alpha_t\sinh(2\Lambda)\Lambda_t\phi_x^2}{\alpha_x^2-\alpha_t^2}+\dfrac{4\alpha\alpha_t\sinh(2\Lambda)\Lambda_t\phi_t^2}{\alpha_x^2-\alpha_t^2}+\dfrac{8\alpha\alpha_t\sinh^2(\Lambda)\phi_x\phi_{xt}}{\alpha_x^2-\alpha_t^2}+\dfrac{8\alpha\alpha_t\sinh^2(\Lambda)\phi_t\phi_{tt}}{\alpha_x^2-\alpha_t^2}\\
&-\dfrac{8\alpha\alpha_{xt}\sinh^2(\Lambda)\phi_x\phi_t}{\alpha_x^2-\alpha_t^2}-\dfrac{8\alpha\alpha_x\sinh(2\Lambda)\Lambda_t\phi_x\phi_t}{\alpha_x^2-\alpha_t^2}-\dfrac{8\alpha\alpha_{x}\sinh^2(\Lambda)\phi_{xt}\phi_t}{\alpha_x^2-\alpha_t^2}-\dfrac{8\alpha\alpha_x\sinh^2(\Lambda)\phi_{x}\phi_{tt}}{\alpha_x^2-\alpha_t^2}.
\end{aligned}
\end{equation*}
 The last step is to perform the sum of the all term in $\Lambda,\phi$, and simplify similar terms. We will use the second equation in the system  \eqref{sistema1}, then, we can write
\begin{equation*}
\begin{aligned}
& Te_{\phi}+Tp_{\phi}+Te_{\Lambda}+Tp_{\Lambda}\\
&= \dfrac{4\alpha_t^2\sinh^2(\Lambda)\phi_t^2}{\alpha_x^2-\alpha_t^2}+\dfrac{4\alpha_t^2\sinh^2(\Lambda)\phi_x^2}{\alpha_x^2-\alpha_t^2}+\dfrac{4\alpha\alpha_t\sinh(2\Lambda)\Lambda_t\phi_x^2}{\alpha_x^2-\alpha_t^2}+\Lambda_t^2-\Lambda_x^2\\
&\quad +\dfrac{4\alpha\alpha_t\sinh(2\Lambda)\Lambda_t\phi_t^2}{\alpha_x^2-\alpha_t^2} -\dfrac{4\alpha\alpha_t\sinh(2\Lambda)\Lambda_t\phi_x^2}{\alpha_x^2-\alpha_t^2}-\dfrac{16\alpha_t\alpha_x\sinh^2(\Lambda)\phi_x\phi_t}{\alpha_x^2-\alpha_t^2}\\
&\quad+\dfrac{8\alpha\alpha_t\sinh^2(\Lambda)\phi_t\phi_{tt}}{\alpha_x^2-\alpha_t^2}-\dfrac{8\alpha\alpha_x\sinh(2\Lambda)\Lambda_t\phi_x\phi_t}{\alpha_x^2-\alpha_t^2}+\dfrac{16\alpha\alpha_{xt}\sinh^2(\Lambda)\phi_x\phi_t}{\alpha_x^2-\alpha_t^2} \\
&\quad-\dfrac{8\alpha\alpha_x\sinh^2(\Lambda)\phi_x\phi_{tt}}{\alpha_x^2-\alpha_t^2}+\dfrac{4\alpha_x^2\sinh^2(\Lambda)\phi_t^2}{\alpha_x^2-\alpha_t^2}
+\dfrac{4\alpha_x^2\sinh^2(\Lambda)\phi_x^2}{\alpha_x^2-\alpha_t^2}\\
&\quad+\dfrac{4\alpha\alpha_x\sinh(2\Lambda)\Lambda_x\phi_x^2}{\alpha_x^2-\alpha_t^2}+\dfrac{4\alpha\alpha_x\sinh(2\Lambda)\Lambda_x\phi_t^2}{\alpha_x^2-\alpha_t^2}+\dfrac{8\alpha\alpha_x\sinh^2(\Lambda)\phi_x\phi_{xx}}{\alpha_x^2-\alpha_t^2}\\
&\quad -\dfrac{8\alpha\alpha_t\sinh(2\Lambda)\Lambda_x\phi_x\phi_t}{\alpha_x^2-\alpha_t^2}-\dfrac{8\alpha\alpha_t\sinh^2(\Lambda)\phi_t\phi_{xx}}{\alpha_x^2-\alpha_t^2}+\dfrac{4\alpha\alpha_t\sinh(2\Lambda)\Lambda_t\phi_t^2}{\alpha_x^2-\alpha_t^2}\\
&\quad-\dfrac{4\alpha\alpha_x\sinh(2\Lambda)\Lambda_x\phi_t^2}{\alpha_x^2-\alpha_t^2}+\dfrac{4\alpha\alpha_x\sinh(2\Lambda)\Lambda_x\phi_x^2}{\alpha_x^2-\alpha_t^2}-\dfrac{16\alpha\alpha_{tx}\sinh^2(\Lambda)\phi_t\phi_x}{\alpha_x^2-\alpha_t^2} \\
&=4\sinh^2(\Lambda)\left(\phi_t^2-\phi_x^2\right) +\left(\dfrac{8\alpha_t\phi_t-8\alpha_x\phi_x}{\alpha_x^2-\alpha_t^2} \right)(\alpha\phi_t\Lambda_t-\alpha\phi_t\Lambda_t)\sinh(2\Lambda)\\
&\quad +\left(\dfrac{8\alpha_t\phi_t-8\alpha_x\phi_x}{\alpha_x^2-\alpha_t^2} \right) \sinh^2(\Lambda)(\alpha_t\phi_t-\alpha_x\phi_x+\alpha\phi_{tt}-\alpha\phi_{xx})+\Lambda_t^2-\Lambda_x^2\\
&=4\sinh^2(\Lambda)\left(\phi_t^2-\phi_x^2\right)+\Lambda_t^2-\Lambda_x^2.
\end{aligned}
\end{equation*}
We can then conclude that:
\begin{equation*}
\begin{aligned}
{\color{black}\partial_t e(t,x)+\partial_x p(t,x)}=& 4\sinh^2(\Lambda)\left(\phi_t^2-\phi_x^2\right)+\Lambda_t^2-\Lambda_x^2
 +\partial_x\left(\dfrac{\alpha_x}{\alpha}\right) -\partial_t\left(\dfrac{\alpha_t}{\alpha}\right),
\end{aligned}
\end{equation*}
as desired.

\subsection*{Conflict of interest} The authors declare no conflict of interest present during the elaboration of this work and its posterior publication.

\subsection*{Contributions} Both authors contributed with substancial efforts to the realization of this work.

\bibliographystyle{unsrt}

\end{document}